%% file: paper.tex
\RequirePackage{fix-cm}
\documentclass[twocolumn]{svjour3}          
\newtheorem{thm}[lemma]{Theorem}

\newtheorem{prop}[lemma]{Proposition}

\newif\iffinal 
\finaltrue 

\usepackage{color}

\definecolor{revBlue}{RGB}{0, 94, 184}

\newif\ifrevised

\ifrevised
  \newcommand{\rev}[1]{{\color{revBlue} #1}}
\else
  \newcommand{\rev}[1]{{\normalcolor #1}}
\fi

\newif\ifedited
\editedtrue

\ifedited
  
\else
  
\fi

\iffinal
  \newcommand{\hans}[1]{}
  \newcommand{\tjs}[1]{}
  \newcommand{\yell}[1]{}
  \newcommand{\ph}[1]{}
  \newcommand{\revs}[1]{}
\else
  \newcommand{\hans}[1]{{\color{blue} \textbf{[Hans says: #1]}}}
  \newcommand{\tjs}[1]{{\color{purple} \textbf{[Tim says: #1]}}}
  \newcommand{\yell}[1]{{\color{orange} \textbf{[#1]}}}
  \newcommand{\ph}[1]{{\color{green} \textbf{[Phil says: #1]}}}
  \newcommand{\revs}[1]{{\color{olive} \textbf{[Reviewers say: #1]}}}
\fi

\usepackage{tikz}
\usepackage{pgfplots}
\usepgfplotslibrary{external}
\usepgfplotslibrary{groupplots}
\usepgfplotslibrary{statistics}
\usetikzlibrary{decorations}
\usetikzlibrary{decorations.markings}
\usetikzlibrary{pgfplots.statistics}
\usetikzlibrary{shapes}
\pgfplotsset{compat = 1.3}

\usepackage{afterpage}
\usepackage{breakcites}

\usepackage{mleftright}
\mleftright

\smartqed  
\makeatletter
\def\cl@chapter{\@elt {theorem}}
\makeatother
\usepackage{graphicx}
\usepackage{amsopn}

\usepackage{amsmath}
\usepackage{amssymb}
\usepackage{pgfplots}
\usepackage{paralist}
\usepackage[applemac]{inputenc}
\usepackage{graphicx}
\usepackage{setspace}
\usepackage{enumerate}
\usepackage{natbib}
\usepackage{bbm}
\usepackage{pgfplots}

\pgfplotsset{compat=newest}
\newlength{\figwidth}
\newlength{\figheight}
\usetikzlibrary{external}

\newcommand{\R}{\mathbb{R}}

\newcommand{\vertiii}[1]{{\left\vert\kern-0.25ex\left\vert\kern-0.25ex\left\vert #1 
   \right\vert\kern-0.25ex\right\vert\kern-0.25ex\right\vert}} 

\usepackage{mathtools}

\newcommand*{\defeq}{\coloneqq}
\newcommand*{\qefed}{\eqqcolon}
\newcommand*{\BO}{\mathcal{O}}
\newcommand*{\rd}{\mathrm{d}}
\newcommand*{\quark}{\setbox0\hbox{$x$}\hbox to\wd0{\hss$\cdot$\hss}}
\newcommand*{\one}{\mathbb{I}}

\newtheorem{assumption}{Assumption}

\usepackage{wrapfig}
\usepackage{cleveref}
\crefname{assumption}{assumption}{assumptions}

\begin{document}

\title{
Convergence Rates of Gaussian ODE Filters
}


\author{Hans Kersting         \and
        T.\ J.\ Sullivan      \and
        Philipp Hennig
}


\institute{
    Hans Kersting \at University of T\"ubingen \\
    and Max Planck Institute for Intelligent Systems \\
    Maria-von-Linden-Stra{\ss}e 6, 72076 T\"ubingen, Germany\\ 
    \email{hans.kersting@uni-tuebingen.de}
    \and
    T.\ J.\ Sullivan \at \rev{University of Warwick\\
    Coventry, CV4 7AL, United Kingdom\\
    \email{t.j.sullivan@warwick.ac.uk}\\
    Zuse Institute Berlin\\
    Takustra{\ss}e 7, 14195 Berlin, Germany\\
    \email{sullivan@zib.de}}
    \and
    Philipp Hennig \at University of T\"ubingen \\
    and Max Planck Institute for Intelligent Systems \\
    Maria-von-Linden-Stra{\ss}e 6, 72076 T\"ubingen, Germany\\
    \email{philipp.hennig@uni-tuebingen.de}
}

\date{Received: date / Accepted: date}

\maketitle

\begin{abstract}
  A recently-introduced class of probabilistic (uncertainty-aware) solvers for ordinary differential equations (ODEs) applies Gaussian (Kalman) filtering to initial value problems.
  These methods model the true solution $x$ and its first $q$ derivatives \emph{a priori} as a Gauss--Markov process $\boldsymbol{X}$, which is then iteratively conditioned on information about $\dot{x}$.
  This article establishes worst-case local convergence rates of order $q+1$ for a wide range of versions of this Gaussian ODE filter, as well as global convergence rates of order $q$ in the case of $q=1$ and an integrated Brownian motion prior, and analyses how inaccurate information on $\dot{x}$ coming from approximate evaluations of $f$ affects these rates.
  Moreover, we show that, in the globally convergent case, the posterior credible intervals are well calibrated in the sense that they globally contract at the same rate as the truncation error.
  We illustrate these theoretical results by numerical experiments which might indicate their generalizability to $q \in \{2,3,\dots\}$.
\end{abstract}

\section{Introduction}

A solver of an initial value problem (IVP) outputs an approximate solution $\hat{x} \colon [0,T] \to \R^d$ of an ordinary differential equation (ODE) with initial condition:

\begin{align} \label{IVP}
  x^{(1)}(t)
  \defeq
  \frac{\rd x}{\rd t} (t)
  &=
  f \left( x(t) \right),\qquad \forall t\in [0,T], 
  \\ 
  x(0) 
  &= 
  x_0\ \in \R^d.
  \notag
\end{align}
(Without loss of generality, we simplify the presentation by restricting attention to the autonomous case.)
The numerical solution $\hat{x}$ is computed by iteratively collecting information on $x^{(1)}(t)$ by evaluating $f \colon \R^d \to \R^d$ at a numerical estimate $\hat{x}(t)$ of $x(t)$ and using these approximate evaluations of the time derivative to extrapolate along the time axis.
In other words, the numerical solution (or \emph{estimator}) $\hat{x}$ of the exact solution (or \emph{estimand}) $x$ is calculated based on evaluations of the vector field $f$ (or \emph{data}).
Accordingly, we treat $\hat{x}$ itself as an estimator, i.e.~a statistic that translates evaluations of $f$ into a probability distribution over $C^1 ([0,T];\R^d)$, the space of continuously differentiable functions from $[0,T]$ to $\R^d$.

This probabilistic interpretation of numerical computations of tractable from intractable quantities as statistical inference of latent from observable quantities applies to all numerical problems and has been repeatedly recommended in the past \citep{poincare1896,diaconis88:_bayes,skilling1991bayesian,ohagan92:_some_bayes_numer_analy,Ritter2000}.
It employs the language of probability theory to account for the epistemic uncertainty (i.e.~limited knowledge) about the accuracy of intermediate and final numerical computations, thereby yielding algorithms which can be more aware of---as well as more robust against---uncertainty over intermediate computational results.
Such algorithms can output probability measures, instead of point estimates, over the final quantity of interest.
This approach, now called \emph{probabilistic numerics (PN)} \citep{HenOsbGirRSPA2015,OatesSullivan2019}, has in recent years been spelled out for a wide range of numerical tasks, including linear algebra, optimization, integration and differential equations, thereby working towards the long-term goal of a coherent framework to propagate uncertainty through chained computations, as desirable, e.g., in statistical machine learning.

In this paper, we determine the convergence rates of a recent family of PN methods \rev{\citep{schober2014nips,Kersting2016UAI,magnani2017,schober2019,TronarpKSH2019}} which recast an IVP as a \emph{stochastic filtering problem} \citep[Chapter 6]{oksendal2003stochastic}, an approach that has been studied in other settings \citep{jazwinski1970stochastic}, but has not been applied to IVPs before.
These methods assume \emph{a priori} that the solution $x$ and its first \rev{$q \in \mathbb N$} derivatives follow a Gauss--Markov process $\boldsymbol{X}$ that solves a stochastic differential equation (SDE).

The evaluations of $f$ at numerical estimates of the true solution can then be regarded as imperfect evaluations of $\dot{x}$, which can then be used for a Bayesian update of $\boldsymbol{X}$.
Such recursive updates along the time axis yield an algorithm whose structure resembles that of Gaussian (Kalman) filtering \citep[Chapter 4]{Sarkka2013}.
These methods add only slight computational overhead compared to classical methods \citep{schober2019} and have been shown to inherit local convergence rates from equivalent classical methods in specific cases \citep{schober2014nips,schober2019}.
These equivalences (i.e.~the equality of the filtering posterior mean and the classical method) are only known to hold in the case of the integrated Brownian motion (IBM) prior and noiseless evaluations of $f$ (in terms of our later notation, the case $R\equiv 0$), as well as under the following restrictions:

Firstly, for $q \in \{1,2,3\}$, and if the first step is divided into sub-steps resembling those of Runge--Kutta methods, an equivalence of the posterior mean of the first step of the filter and the explicit Runge--Kutta method of order $q$ was established in \citet{schober2014nips} (but for $q \in \{2,3\}$ only in the limit as the initial time of the IBM tends to $-\infty$).
Secondly, it was shown \rev{by} \citet{schober2019} that, for $q=1$, the posterior mean after each step coincides with the trapezoidal rule if it takes an additional evaluation of $f$ at the end of each step, known as P(EC)1.
The same paper shows that, for $q=2$, the filter coincides with a third-order Nordsieck method \citep{nordsieck1962numerical} if the filter is in the steady state, i.e.~after the sequence of error covariance matrices has converged.
These results neither cover filters with the integrated Ornstein--Uhlenbeck process (IOUP) prior \citep{magnani2017} nor non-zero noise models on evaluations of $f$.

In this paper, we directly prove convergence rates without first fitting the filter to existing methods, and thereby lift many of the above restrictions on the convergence rates. 
While the more-recent work by \citet{TronarpSarkkaHennig_BayesianODESolvers_2020} also provide convergence rates of estimators of $x$ in the Bayesian ODE filtering/smoothing para-digm, they concern the maximum a posteriori estimator (as computed by the iterated extended Kalman ODE smoother), and therefore differ from our convergence rates of the filtering mean (as computed by the Kalman ODE filter).

\subsection{Contribution}
\label{subsec:contribution}

Our main results---Theorems \ref{lemma:Psi_with_delta} and \ref{theorem:GlobalTruncation}---provide local and global convergence rates of the ODE filter \rev{when the step size $h$ goes to zero}.
Theorem \ref{lemma:Psi_with_delta} shows local convergence rates of $h^{q+1}$ without the above-mentioned previous restrictions---i.e.~for a generic Gaussian ODE filter for all $q \in \mathbb N$, both IBM and IOUP prior, flexible Gaussian initialization (see \Cref{ass:assumption2,ass:eps0_bound}), and arbitrary evaluation noise $R \geq 0$.
As a first global convergence result, Theorem \ref{theorem:GlobalTruncation} establishes global convergence rates of $h^{q}$ in the case of $q=1$, the IBM prior and all fixed measurement uncertainty models \rev{$R$ of order $p \in [1,\infty]$} (see \Cref{ass:R}).
\rev{This} global rate of the worst-case error is matched by the contraction rate of the posterior \rev{credible intervals, as we show in Theorem \ref{theorem:contraction_of_credible_intervals}}.
\rev{Moreover}, we also give closed-form expressions for the steady states in the global case \rev{and illustrate our results as well as their possible generalizability to $q \geq 2$ by experiments in \Cref{sec:experiments}.}

\subsection{Related work on probabilistic ODE solvers}
\label{subsec:related_work}

The Gaussian ODE filter can be thought of as a self-consistent Bayesian decision agent that iteratively updates its prior belief $\boldsymbol{X}$ over $x \colon [0,T] \to \R^d$ (and its first $q$ derivatives) with information on $\dot{x}$ from evaluating $f$.\footnote{Here, the word `Bayesian' describes the algorithm in the sense that it employs a prior over the quantity of interest and updates it by Bayes rule according to a prespecified measurement model (as also used in \citet{skilling1991bayesian,o.13:_bayes_uncer_quant_differ_equat,Kersting2016UAI}).
The ODE filter is not Bayesian in the stronger sense of \citet{Cockayne2017BayesianPN}, and it remains an open problem to construct a Bayesian solver in this strong sense without restrictive assumptions, as discussed in \citet{Wang18}.}
For Gauss--Markov priors, it performs exact Bayesian inference and optimally (with respect to the $L^2$-loss) extrapolates along the time axis.
Accordingly, all of its computations are deterministic and---due to its restriction to Gaussian distributions---only slightly more expensive than classical solvers.
Experiments demonstrating competitive performance \rev{with classical methods} are provided in \citet[Section 5]{schober2019}.

Another line of work (comprising the methods from \rev{\citet{o.13:_bayes_uncer_quant_differ_equat,conrad_probability_2017,teymur2016probabilistic,Lie17,AbdulleGaregnani17,teymur2018a}}) introduces probability measures to ODE solvers in a fundamentally different way---by representing the distribution of all numerically possible trajectories with a set of sample paths.
To compute these sample paths, \citet{o.13:_bayes_uncer_quant_differ_equat} draws them from a (Bayesian) Gaussian process (GP) regression; \rev{\citet{conrad_probability_2017,teymur2016probabilistic,Lie17,teymur2018a}} perturb classical estimates after an integration step with a suitably scaled Gaussian noise; and \citet{AbdulleGaregnani17} perturbs the classical estimate instead by choosing a stochastic step-size.
While \rev{\citet{conrad_probability_2017,teymur2016probabilistic,Lie17,AbdulleGaregnani17,teymur2018a}} can be thought of as (non-Bayesian) `stochastic wrappers' around classical solvers, which produce samples with the same convergence rate, \citet{o.13:_bayes_uncer_quant_differ_equat} employs---like the filter---GP regression to represent the belief on $x$.
While the Gaussian ODE filter can convergence with polynomial order (see results in this paper), 
However, \citet{o.13:_bayes_uncer_quant_differ_equat} only \rev{show} first-order convergence rates and also construct a sample representation of numerical errors, from which samples are drawn iteratively.
A conceptual and experimental comparison between the filter and \citet{o.13:_bayes_uncer_quant_differ_equat} can be found in \citet{schober2019}.
\rev{An additional numerical test against \citet{conrad_probability_2017} was given by \citet{Kersting2016UAI}.
Moreover, \citet{TronarpKSH2019} recently introduced a particle ODE filter, which combines a filtering-based solver with a sampling-based uncertainty quantification (UQ), and compared it numerically with \citet{conrad_probability_2017} and \citet{o.13:_bayes_uncer_quant_differ_equat}.}

All of the above sampling-based methods can hence represent more expressive, non-Gaussian posteriors (as e.g.~desirable for bifurcations), but multiply the computational cost of the underlying method by the number of samples. 
\rev{ODE filters are,} in contrast, not a perturbation of known methods, but \rev{novel methods} designed for computational speed and for a robust treatment of intermediate uncertain values (such as the evaluations of $f$ at estimated points).
\rev{Unless parallelization of the samples in the sampling-based solvers is possible and inexpensive, one can spend the computational budget for generating additional samples on dividing the step size $h$ by the number of samples, and can thereby polynomially decrease the error.
Its Gaussian UQ, however, should not be regarded as the true UQ---in particular for chaotic systems whose uncertainty can be better represented by sampling-based solvers, see e.g.~\citet[Figure 1]{conrad_probability_2017} and \citet[Figure 2]{AbdulleGaregnani17}---but as a rough inexpensive probabilistic treatment of intermediate values and final errors which is supposed to, on average, guide the posterior mean towards the true $x$. 
Therefore, it is in a way more similar to classical non-stochastic solvers than to sampling-based stochastic solvers and, unlike sampling-based solvers, puts emphasis on computational speed over statistical accuracy.
Nevertheless, its Gaussian UQ is sufficient to make the forward models in ODE inverse problems more `uncertainty-aware'; see \citet[Section 3]{KerstingKraemer_godef_inverse_2020}.
}

Accordingly, \rev{the convergence results in this paper} concern the convergence rate of the posterior mean to the true \rev{solution, while} the theoretical results from \rev{\citet{teymur2016probabilistic,o.13:_bayes_uncer_quant_differ_equat,conrad_probability_2017,Lie17,AbdulleGaregnani17,teymur2018a}} provide convergence rates of the variance of the non-Gaussian empirical measure of samples (and not for an individual sample).

\rev{
\subsection{Relation to filtering theory}
\label{subsec:relation_to_filtering_theory}

While Gaussian (Kalman) filtering was first applied to the solution of ODEs by \citet{Kersting2016UAI} and \citet{schober2019}, it has previously been analysed in the filtering, data assimilation as well as linear system theory community.
The convergence results in this paper are concerned with its asymptotics when the step size $h$ (aka time step between data points) goes to zero.
In the classical filtering setting, where the data comes from an external sensor, this quantity is not treated as a variable, as it is considered a property of the data and not, like in our case, of the algorithm.
Accordingly, the standard books lack such an analysis for $h \to 0$---see \citet{jazwinski1970stochastic,anderson1979,Maybeck_79} for filtering, \citet{LawStuartZygalakis15,ReichCotter15} for data assimilation and \citet{CallierDesoer91} for linear system theory---and we believe that our convergence results are completely novel.
It is conceivable that, also for these communities, this paper may be of interest in settings where the data collection mechanism can be actively chosen, e.g.~when the frequency of the data can be varied or sensors of different frequencies can be used.}

\subsection{Outline}

The paper begins with a brief introduction to Gaussian ODE filtering in \Cref{sec:introduction_of_algorithm}.
Next, \Cref{sec:regularity_of_flow,sec:auxiliary_bounds_on_intermediate_quantities} provide auxiliary bounds on the flow map of the ODE and on intermediate quantities of the filter respectively.
With the help of these bounds, \Cref{subsection:local_truncation_error,subsection:Global_truncation_error} establish local and global convergence rates of the filtering mean respectively.
\rev{In} light of these rates, \Cref{sec:calibration_of_credible_intervals} analyses for which measurement noise models the posterior \rev{credible intervals are well calibrated.
These theoretical results are experimentally confirmed and discussed in \Cref{sec:experiments}.
\Cref{sec:discussion} concludes with a high-level discussion.}



\subsection{Notation}
\label{subsec:notation}

We will use the notation $[n] \defeq \{0,\dots,n-1\}$.
For \rev{vectors and matrices}, we will use zero-based numbering, e.g.~$x=(x_0,\dots,x_{d-1}) \in \R^d .$
For a matrix $P \in \R^{n \times m}$ and $(i,j) \in [n] \times [m]$, we will write $P_{i,:} \in \R^{1 \times m} $ for the $i$\textsuperscript{th} row and $P_{:,j}$ for the $j$\textsuperscript{th} column of $P$.
A fixed but arbitrary norm on $\R^d$ will be denoted by $\Vert \quark \Vert$.
\rev{The minimum and maximum of two real numbers $a$ and $b$ will be denoted by $a \wedge b$ and $a  \vee b$ respectively.
Vectors that span all $q$ modeled derivatives will be denoted by bold symbols, such as $\boldsymbol{x}$.}

\section{Gaussian ODE filtering}
\label{sec:introduction_of_algorithm}

This section defines how a Gaussian filter can solve the IVP \cref{IVP}.
In the various subsections, we first explain the choice of prior on $x$, then describe how the algorithm computes a posterior output from this prior (by defining a numerical integrator $\boldsymbol{\Psi}$), and add explanations on the measurement noise of the derivative observations.
To alternatively understand how this algorithm can be derived as an extension of generic Gaussian filtering in probabilistic state space models, see the concise presentation in \cite[Supplement A]{KerstingKraemer_godef_inverse_2020}.

\subsection{Prior on $\boldsymbol{x}$}
\label{subsec:prior_on_x}

In PN, it is common \citep[Section 3(a)]{HenOsbGirRSPA2015} to put a prior measure on the unknown solution \rev{$x$.
Often}, for fast Bayesian inference by linear algebra \citep[Chapter 2]{RasmussenWilliams}, this prior is Gaussian.
To enable GP inference in linear time by Kalman filtering \rev{\citep[Chapter 4.3]{Sarkka2013}}, we further restrict the prior to Markov processes.
As \rev{discussed in \citet[Chapter 12.4]{SarkkaSolin2019}}, a wide class of such Gauss--Markov processes can be captured by a law of the (strong) solution \citep[Chapter 5.3]{oksendal2003stochastic} of a linear SDE with Gaussian initial condition.
\rev{Here---as we, by \cref{IVP}, have information on at least one derivative of $x$---the prior also includes the first $q \in \mathbb N$ derivatives.
 Therefore,} for all $j \in [d]$, we define the vector of time derivatives by $\boldsymbol{X}_j = \left( X_j^{(0)}, \dots, X_j^{(q)} \right)^ \intercal$.
\rev{We define $\boldsymbol{X}_{j}$ as a $(q + 1)$-dimensional stochastic process via the SDE}
\begin{align} \label{SDE}
  &\rd \boldsymbol{X}_{j}\rev{(t)}
  =
  \left ( \rd X_{j}^{(0)}\rev{(t)} , \dots , \rd X_{j}^{(q-1)}\rev{(t)} , \rd X_{j}^{(q)}\rev{(t)}   \right)^{\intercal}
  \\
  &=
  \begin{pmatrix} 0 & 1 & 0 \dots &  0 \\ \vdots & \ddots & \ddots & 0 \\ \vdots & \ddots & 0 & 1 \\ c_0 & \dots & \dots & c_q \end{pmatrix} \begin{pmatrix} X_{j}^{(0)}\rev{(t)} \\ \vdots \\ X_{j}^{(q-1)}\rev{(t)} \\ X_{j}^{(q)}\rev{(t)} \end{pmatrix} \, \rd t + \begin{pmatrix} 0 \\ \vdots \\ 0 \\ \sigma_j \end{pmatrix} \, \rd B_{j}\rev{(t)},
  \notag
\end{align}
\rev{driven by mutually independent one-dimensional Brownian motions $\{B_{j};\ \allowbreak j\in [d]\}$} (independent of \rev{$\boldsymbol{X}(0)$) scaled by $\sigma_j > 0$, with initial condition $\boldsymbol X_{j}(0) \sim \mathcal{N} \allowbreak ( m_j(0) \allowbreak , \allowbreak P_j(0) \allowbreak )$.}
We assume that \rev{$\left \{X_{j}(0);\ j \in [d] \right \}$} are independent.
\rev{In other words, we model the unknown $i$\textsuperscript{th} derivative of the $j$\textsuperscript{th} dimension of the solution $x$ of the IVP \cref{IVP}, denoted by $x^{(i)}_j$, as a draw from a real-valued, one-dimensional GP $X^{(i)}_j$, for all $i \in [q+1]$ and $j \in [d]$, such that $X^{(q)}_j$ is defined by $(c_0,\dots,c_q)$ as well as the Brownian motion scale $\sigma_j$ and $X^{(i-1)}_j$ is defined to be the integral of $X^{(i)}_j$.
Note that, by the} independence of the components of the $d$-dimensional Brownian motion, the components \rev{$\left \{ \left\{ \boldsymbol{X}_{j}(t);\ 0 \leq t \leq T \right\};\ j \in [d]  \right\}$ of $\left\{ \boldsymbol{X}(t);\ 0\leq t\leq T \right\}$} are independent\footnote{More involved correlation models of \rev{$\left \{ \left\{ \boldsymbol{X}_{j}(t);\ 0 \leq t \leq T \right\};\ j \in [d]  \right\}$} are straightforward to incorporate into the SDE \cref{SDE}, but seem complicated to analyse. Therefore, we restrict our attention to independent dimensions. \rev{See \Cref{appendix:dependent_dimensions} for an explanation of this restriction.} Note that one can also use a state space vector $\boldsymbol{X}(t)$ which models other features of $x(t)$ than the derivatives, as demonstrated with Fourier summands in \cite{KerstingMahsereci2020}.\label{footnote:independence}}. 
The \rev{(strong)} solution of \cref{SDE} is a Gauss--Markov process with mean $m_j \colon [0,T] \to \R^{q+1}$ and covariance matrix $P_j \colon [0,T] \to \R^{(q+1)\times (q+1)}$ given by
\begin{align}
  \label{drift_matrix}
  m_j(t) 
  &= 
  A(t) m_j(0),
  \\
  P_j(t) 
  &= 
  A(t)P_j(0) A(t)^{\intercal} + \rev{Q(t)},
  \label{drift_matrix_II}
\end{align}
where the matrices \rev{$A(t),\ Q(t) \in \R^{(q+1)\times (q+1)}$} yielded by the SDE \cref{SDE} are known in closed form \citet[Theorem 2.9]{sarkka2006thesis} \rev{(see \cref{matrix:Q})}.
The precise choice of the prior stochastic process $\boldsymbol{X}$ depends on the choice of \rev{$(c_0,\dots,c_q) \in \mathbb{R}^{q+1}$} in \cref{SDE}.
\rev{While the below algorithm works for all choices of $c$, we restrict our attention to the case of
 \begin{align}
   \label{eq:a_restriction_to_IOUP/IBM}
   (c_0,\dots,c_q)
   \defeq
   (0,\dots,0,-\theta),
   \qquad
   \text{for some}
   \quad
   \theta \geq 0,
 \end{align}
where the} $q$-times integrated Brownian motion (IBM) and \rev{the} $q$-times integrated Ornstein--Uhlenbeck process (IOUP) \rev{with drift parameter $\theta$ is the unique solution of \cref{SDE}, in the case of $\theta = 0$ and $\theta > 0$ respectively \citep[Chapter 5: Example 6.8]{karatzas1991brownian}.
In this case, the matrices} $A$ and $Q$ \rev{from \cref{drift_matrix,drift_matrix_II} are given by}
\rev{
\begin{align} \label{def:A^IOUP}
  \rev{A(t)_{ij}}
  &=
  \begin{cases}
    \one_{i\leq j} \frac{t^{j-i}}{(j-i)!}, & \mbox{if }j\neq q, \\ \rev{\frac{t^{q-i}}{(q-i)!} - \theta \sum_{k=q+1-i}^{\infty} \frac{(-\theta)^{k+i-q-1} t^k}{ k! }} , & \mbox{if }j=q,
  \end{cases}
  \\
  \rev{Q(t)_{ij}}
  &=
  \sigma^2 \frac{t^{2q+1-i-j}}{(2q+1-i-j)(q-i)!(q-j)!} 
  \notag
  \\  
  &\phantom{=}+ \Theta\left(t^{2q+2-i-j}\right).
  \label{eq:Q_IOUP}
\end{align} }
(Derivations of \cref{def:A^IOUP,eq:Q_IOUP}, as well as the precise form of \rev{$Q$ without $\Theta(t^{2q+2-i-j})$,} are presented in \rev{\Cref{appendix:section_A_and_Q}.)}
Hence, for all \rev{$i \in [q+1]$, the prediction of step size $h$} of the $i$\textsuperscript{th} derivative from any state $u\in \R^{q+1}$ is \rev{given by}
\rev{
\begin{align}   
  \left [ A(t) u \right ]_i
  =
  &\sum_{k=i}^q \frac{t^{k-i}}{(k-i)!} u_k
  \notag
  \\
  &- \theta \left [ \sum_{k=q+1-i}^{\infty} \frac{(-\theta)^{k+i-q-1}}{k!} t^k   \right ] u_q.
  \label{A_prediction}
\end{align}
} 
\subsection{The algorithm}     \label{Gaussian ODE filtering}
To avoid the introduction of additional indices, we will define the algorithm $\boldsymbol{\Psi}$ for $d=1$; for statements on the general case of $d\in \mathbb N$ we will use the same symbols from \cref{eq:C^-_predict}--\cref{eq:C_update_MM} as vectors over the whole dimension---see e.g.~\cref{bound:r} for a statement about a general $r \in \R^d$.
By the independence of the dimensions of $\boldsymbol{X}$, due to \cref{SDE}, extension to $d \in \mathbb N$ amounts to applying $\boldsymbol{\Psi}$ to every dimension independently \rev{(recall \Cref{footnote:independence}).}
Accordingly, we may in many of the below proofs w.l.o.g.~assume $d=1$.
Now, as previously spelled out in \citet{Kersting2016UAI,schober2019}, Bayesian filtering of $\boldsymbol{X}$---i.e.~iteratively conditioning $\boldsymbol{X}$ on the information on $X^{(1)}$ from evaluations of $f$ at the mean of the current conditioned ${X}^{(0)}$---yields the following numerical method $\boldsymbol{\Psi}$.
Let \rev{$\boldsymbol{m}(t) = (m^{(0)}(t),\dots,m^{(q)}(t))^{\intercal} \in \R^{q+1}$} be an arbitrary state at some point in time $t \in [0,T]$ (i.e.~$m^{(i)}(t)$ is an estimate for $x^{(i)}(t)$), and let $P(t) \in \R^{(q+1) \times (q+1)}$ be the covariance matrix of $x^{(i)}(t)$.
For $t \in [0,T]$, let the current estimate of $\boldsymbol{x}(t)$ be a normal distribution $\mathcal{N}(\boldsymbol{m}(t),P(t))$, i.e.~the mean \rev{$\boldsymbol{m}(t) \in \R^{q+1}$} represents the best numerical estimate (given data $\{y(h),\allowbreak \dots,\allowbreak y(t)\}$, see \cref{eq:def_y_MAP}) and the covariance matrix $P(t) \in \R^{(q+1) \times (q+1)}$ its uncertainty.
\rev{For the} time step $t \to t+h$ of size $h>0$, the ODE filter first computes the prediction step consisting of \emph{predictive mean}
\begin{align} \label{eq:def_predictive_mean}
  \rev{\boldsymbol{m}^-(t+h)} & \rev{\defeq A(h) \boldsymbol{m}(t) \ \in \R^{q+1},  }
\end{align}
and \emph{predictive covariance}
\begin{align}
  \label{eq:C^-_predict}
  P^-(t+h)
  &\defeq
  A(h)P(t)A(h)^{\intercal} + Q\rev{(h)} \ \in \R^{(q+1)\times (q+1)},
\end{align}
with \rev{$A$ and $Q$ generally defined by \cref{matrix:Q} and, in the considered particular case of \cref{eq:a_restriction_to_IOUP/IBM}, by \cref{def:A^IOUP,eq:Q_IOUP}.
In the subsequent step, the following quantities are computed first}:
the \emph{Kalman gain}
\begin{align}
  \boldsymbol{\beta}&(t+h)
  =
  (\beta^{(0)}(t+h), \dots, \beta^{(q)}(t+h))^{\intercal}
  \notag
  \\
  &\defeq
  \frac{P^-(t+h)_{:1}}{(P^-(t+h))_{11} + R(t+h)} \in \R^{(q+1) \times 1},   
  \label{def_beta_MM}
\end{align}
the \emph{measurement/data on $\dot{x}$}
\begin{align}
  y(t+h)
  &\defeq
  f\left( m^{-,(0)}(t+h) \right) \ \in \R,
  \label{eq:def_y_MAP}
\end{align}
and \emph{innovation/residual}
\begin{align}
  r(t+h)
  &\defeq
  y(t+h) - m^{-,(1)}(t+h) \ \in \R.
  \label{def:r}
\end{align}
Here, $R$ denotes the variance of $y$ (the `measurement noise') and captures the squared difference between the data $y(t+h) = f(m^-(t+h))$ that the algorithm actually receives and the idealised data $\dot{x}(t+h) = f(x(t+h))$ that it `should' receive (see \Cref{subsec:measurement_noise}).
Finally, the mean and the covariance matrix are conditioned on \rev{this data, which yields} the \emph{updated mean}
\begin{align}
  \notag
  \rev{\boldsymbol{\Psi}_{P(t),h}(\boldsymbol{m}(t))}
  &\defeq
  \boldsymbol{m}(t+h)
  \\
  &=
  \boldsymbol{m}^-(t+h) + \boldsymbol{\beta}(t+h) r(t+h),
  \label{def:Psi}
\end{align}
and the \emph{updated covariance}
\begin{align}
  P(t+h)
  &\defeq
  P^-(t+h) -  \frac{P^-(t+h)_{:,1}P^-(t+h)_{1,:}}{P^-(t+h)_{11} + R(t+h)}.
  \label{eq:C_update_MM}
\end{align}
\rev{This concludes the step $t \to t+h$, with the Gaussian distribution $\mathcal N(\boldsymbol m(t+h),P(t+h))$ over $\boldsymbol{x}(t+h)$.}
The algorithm is iterated by computing $\boldsymbol{m}(t+2h) \defeq \boldsymbol{\Psi}_{P(t+h),h}(\boldsymbol{m}(t+h))$ as well as repeating \cref{eq:C^-_predict} and \cref{eq:C_update_MM}, with $P(t+h)$ instead of $P(t)$, to obtain $P(t+2h)$.
In the following, to avoid notational clutter, the dependence of the above quantities on $t$, $h$ and $\sigma$ will be omitted if their values are unambiguous.
Parameter adaptation reminiscent of classical methods (e.g.~for $\sigma$ s.t.~the added variance per step coincide with standard error estimates) have been explored in \citet[Section 4]{schober2019}.

This filter is essentially an iterative application of Bayes rule (see e.g.~\citet[Chapter 4]{Sarkka2013}) based on the prior $\boldsymbol{X}$ on $\boldsymbol{x}$ specified by \cref{SDE} (entering the algorithm via $A$ and $Q$) and the measurement model $y \sim \mathcal{N}(\dot{x},R)$.
Since the measurement model is a likelihood by another name and therefore forms a complete Bayesian model together with the prior $\boldsymbol{X}$, it remains to detail the measurement model (recall \cref{subsec:prior_on_x} for the choice of prior).
Concerning the data generation mechanism for $y$ \cref{eq:def_y_MAP}, we only consider the maximum-a-posteriori point estimate of $\dot{x}(t)$ given $\mathcal{N}(m^{-,(0)}(t),P_{00}^-(t))$; a discussion of more involved statistical models for $y$ as well as an algorithm box for the Gaussian ODE filter can be found in \citet[Subsection 2.2]{schober2019}.
Next, for lack of such a discussion for $R$, we will examine different choices of $R$---which have proved central to the \rev{UQ} of the filter \citep{Kersting2016UAI} and will turn out to affect global convergence properties in \Cref{subsection:Global_truncation_error}.

\subsection{Measurement noise $R$}
\label{subsec:measurement_noise}

Two sources of uncertainty add to $R(t)$: noise from imprecise knowledge of $x(t)$ and $f$.
Given $f$, previous integration steps of the filter (as well as an imprecise initial value) inject uncertainty about how close $m^-(t)$ is to $x(t)$ and how close $y = f(m^-(t))$ is to $\dot{x}(t)) = f(x(t))$.
This uncertainty stems from the discretization error $\Vert m^{-,(0)}(t) - x(t) \Vert$ and, hence, \rev{tends to increase} with $h$.
Additionally, there can be uncertainty from a misspecified $f$, e.g.~when $f$ has estimated parameters, or from numerically imprecise evaluations of $f$, which can be added to $R$---a functionality which classical solvers do not possess.
In this paper, since $R$ depends on $h$ via the numerical uncertainty on $x(t)$, we analyse the influence of noise \rev{$R$ of order $p \in [1,\infty]$} (see \Cref{ass:R}) on the quality of the solution to illuminate for which orders of noise we can trust the solution to which extent and when we should, instead of decreasing $h$, rather spend computational budget on specifying or evaluating $f$ more precisely.
The explicit dependence of the noise on its order $p$ in $h$ resembles, despite the fundamentally different role of $R$ compared to additive noise in \citet{conrad_probability_2017,AbdulleGaregnani17}, the variable $p$ in \citet[Assumption 1]{conrad_probability_2017} and \citet[Assumption 2.2]{AbdulleGaregnani17} in the sense that the analysis highlights how uncertainty of this order can still be modeled without breaking the convergence rates.
(Adaptive noise models are computationally feasible \citep{Kersting2016UAI} but lie outside the scope of our analysis.)

\section{Regularity of flow}
\label{sec:regularity_of_flow}

Before we proceed to the analysis of $\boldsymbol{\Psi}$, we provide all regularity results \rev{necessary for arbitrary $q,d \in \mathbb N$} in this section.
\begin{assumption} \label{ass:f_Global_Lipschitz}
  The vector field $f \in C^{q} ( \R^d; \R^d )$ is globally Lipschitz and all its derivatives of order up to $q$ are uniformly bounded and globally Lipschitz, i.e.~there exists some $L>0$ such that \rev{$\Vert D^{\alpha} f \Vert_{\infty} \leq L$} for all multi-indices $\alpha \in \mathbb{N}_0^{d}$ with \rev{$1 \leq \sum_i \alpha_i \leq q$, and $\Vert D^{\alpha} f(a) - D^{\alpha} f(b) \Vert \leq L \Vert a - b \Vert$ for all multi-indices $\alpha \in \mathbb{N}_0^{d}$ with $0 \leq \sum_i \alpha_i \leq q$.}
\end{assumption}
\Cref{ass:f_Global_Lipschitz} and the Picard--Lindel\"of theorem imply that the solution $x$ is a well-defined element of $C^{q+1} \allowbreak ( [0,T];\allowbreak \R^d )$.
For $i \in [q+1]$, we denote $\frac{\rd^i x}{\rd t^i}$ by $x^{(i)}$.
\rev{Recall that, by} a bold symbol, we denote the vector of these derivatives: $\boldsymbol{x} \equiv ( x^{(0)}, \dots, x^{(q)} )^{\intercal}$.
In particular, the solution $x$ of \cref{IVP} is denoted by $x^{(0)}$.
\rev{Analogously,} we denote the flow of the ODE \cref{IVP} by $\Phi^{(0)}$, i.e.~$\Phi_t^{(0)}(x_0) \equiv x^{(0)}(t)$, and, for all $i\in [q+1]$, its $i$\textsuperscript{th} partial derivative with respect to $t$ by $\Phi^{(i)}$, so that $\Phi_t^{(i)}(x_0) \equiv x^{(i)}(t)$.

\begin{lemma}
  \label{lemma:Taylor_expansion}
  Under \Cref{ass:f_Global_Lipschitz}, for all $a \in \R^d$ and all $h>0$,
  \begin{align}
  \left \Vert \Phi_h^{(i)}(a) - \sum_{k=i}^q  \frac{h^{k-i}}{(k-i)!} \Phi_0 ^{(k)}(a) \right \Vert
  \leq
  Kh^{q+1-i}.
  \label{eq:Taylor_expansion}
  \end{align}
\end{lemma}

\rev{Here, and in the sequel, $K>0$ denotes a constant independent of $h$ and $\theta$ which may change from line to line.}

\begin{proof}
  By \Cref{ass:f_Global_Lipschitz}, $\Phi^{(q+1)}$ exists and is bounded by $\Vert \Phi^{(q+1)} \Vert \leq L$, which can be seen by applying the chain rule $q$ times to both sides of \cref{IVP}.
  Now, applying $\Vert \Phi^{(q+1)} \Vert \leq L$ to the term $\Phi_{\tau}^{(q+1)}(a)$ (for some $\tau \in (0,h)$) in the Lagrange remainder of the $(q-i)$\textsuperscript{th}-order Taylor expansion of $\Phi_h^{(i)}(a)$ yields \cref{eq:Taylor_expansion}.
  \qed
\end{proof}
\begin{lemma}
  \label{lemma:Phi0_regularity}
  Under \Cref{ass:f_Global_Lipschitz} and for all sufficiently small $h>0$,
  \begin{align}  \label{eq:bound_Phi0}
    \sup_{a\neq b \in \R^d}
    \frac{\left \Vert \Phi^{(0)}_h(a) - \Phi^{(0)}_h(b) \right \Vert }
    {\left \Vert a - b  \right \Vert }
    \leq
    1+2Lh.
  \end{align}
\end{lemma}

\begin{proof}
  Immediate corollary of \citet[Theorem 2.8]{Teschl2012}.
  \qed
\end{proof}

\rev{Global convergence (\Cref{subsection:Global_truncation_error}) will require the following generalization of \Cref{lemma:Phi0_regularity}.}

\begin{lemma}
  \label{lemma:flow_map_regularity}
  Let $q=1$.
  Then, under \Cref{ass:f_Global_Lipschitz} and for all sufficiently small $h>0$,
  \begin{align}  \label{eq:bound_boldsymbol_Phi}
    \sup_{a\neq b \in \R^d} \frac{\vertiii{\boldsymbol{\Phi}_h(a) - \boldsymbol{\Phi}_h(b)}_h}{\left \Vert a - b  \right \Vert}
    \leq
    1+Kh,
  \end{align}
  where, given the norm $\Vert \quark \Vert$ on $\R^d$ and $h>0$, the new norm $\vertiii{\quark}_h$ on $\R^{(q+1)\times d}$ is defined by
  \begin{align}
    \label{def:h_norm}
    \vertiii{a}_h \defeq \sum_{i=0}^q h^i \left \Vert a_{i,:} \right \Vert.
  \end{align}
\end{lemma}

\begin{remark}
  The necessity of $\vertiii{\quark}_h$ stems from the fact that---unlike other ODE solvers---the ODE filter $\boldsymbol{\Psi}$ additionally estimates and uses the first $q$ derivatives in its state $\boldsymbol{m} \in \R^{(q+1)\times d}$, whose development cannot be bounded in $\Vert \quark \Vert$, but in $\vertiii{\quark}_h$.
  The norm $\vertiii{\quark}_h$ is used to make rigorous the intuition that the estimates of the solution's time derivative are `one order of $h$ worse per derivative'.
\end{remark}

\begin{proof}
  We bound the second summand of
  \begin{align} \label{eq:expand_norm_Phi_h}
    &\vertiii{\boldsymbol{\Phi}_h(a) - \boldsymbol{\Phi}_h(b)}_h
    \quad
    \stackrel{\text{\cref{def:h_norm}}}{=}
    \\
    &\underbrace{\Big \Vert \Phi^{(0)}_h(a) - \Phi^{(0)}_h(b) \Big \Vert}_{\leq (1+2Lh) \Vert a - b \Vert, \text{ by }\cref{eq:bound_Phi0}}
    +
    \
    h
    \Big \Vert \underbrace{\Phi^{(1)}_h(a)}_{=f\left(\Phi^{(0)}_h(a) \right)} - \underbrace{\Phi^{(1)}_h(b)}_{=f\left(\Phi^{(0)}_h(b) \right)} \Big \Vert
    \notag
  \end{align}
  by
  \begin{align}
    &\left \Vert f\left(\Phi^{(0)}_h(a)\right) - f\left(\Phi^{(0)}_h(b)  \right) \right \Vert
    \stackrel{\text{Ass. }\ref{ass:f_Global_Lipschitz}}{\leq}
    \label{eq:Phi1_Lipschitz_bound}
    \\
    &\qquad L \left \Vert \Phi^{(0)}_h(a) - \Phi^{(0)}_h(b)  \right \Vert
    \stackrel{\text{\cref{eq:bound_Phi0}}}{\leq}
    L(1+2Lh) \left \Vert a - b \right \Vert.
    \notag
  \end{align}
  Inserting \cref{eq:Phi1_Lipschitz_bound} into \cref{eq:expand_norm_Phi_h} concludes the proof.
  \qed
\end{proof}

\section{The role of the state misalignments $\delta$}
\label{sec:role_of_state_misalignment}

In Gaussian ODE filtering, the interconnection between the estimates of the ODE solution $x(t)=x^{(0)}(t)$ and its first $q$ derivatives $\{x^{(1)}(t), \dots, x^{(q)}(t)\}$ is intricate.
From a purely analytical point of view, every possible estimate $m(t)$ of $x(t)$ comes with a fixed set of derivatives, which are implied by the ODE, for the following reason:
Clearly, by \cref{IVP}, the estimate $m^{(1)}(t)$ of $x^{(1)}(t)$ ought to be $f(m(t))$. 
More generally (for $i \in [q+1]$) the estimate $m^{(i)}(t)$ of $x^{(i)}(t)$ is determined by the ODE as well. 
To see this, let us first recursively define $f^{(i)} \colon \R^d \to \R^d$ by $f^{(0)}(a) \defeq a$, $f^{(1)}(a) \defeq f(a)$ and $f^{(i)}(a) \defeq [ \nabla_x f^{(i-1)} \cdot f  ](a)$.
Now, differentiating the ODE, \cref{IVP}, $(i-1)$-times by the chain rule yields
\begin{align}
  x^{(i)}(t)
  =
  f^{(i-1)}(t)\left (x^{(0)}(t)\right ),
\end{align}
which implies that $m^{(i)}(t)$ ought to be $f^{(i-1)}(t)\left (m^{(0)}(t)\right )$
Since  
\begin{align}   \label{eq:Phi_i_=_f_i-1}
  \Phi_0^{(i)} \left(m^{(0)}(nh) \right) = f^{(i-1)}\left(m^{(0)}(nh)\right)
\end{align}
(which we prove in \Cref{subsec:Proof_of_Eq}), this amounts to requiring that
\begin{align} \label{eq:analytically_desirable_derivative_estimates}
  m^{(i)}(t)
  \overset{!}{=}
  \Phi_0^{(i)} \left(m^{(0)}(nh) \right).
\end{align}
Since $\Phi_0^{(i)}$ is (recall \Cref{sec:regularity_of_flow}) the $i$\textsuperscript{th} time derivative of the flow map $\Phi^{(0)}$ at $t=0$, this simply means that $m^{(i)}(t)$ would be set to the `true' derivatives in the case where the initial condition of the ODE, \cref{IVP}, is $x(0) = m^{(0)}(t)$ instead of $x(0)=x_0$---or, more loosely speaking, that the derivative estimates $m^{(i)}(t)$ are forced to comply with $m^{(0)}(t)$, irrespective of our belief $x^{(i)}(t) \sim \mathcal{N}(m^{(i)}(t),P_{ii}(t))$.
The Gaussian ODE filter, however, does not use this (intractable) analytical approach.
Instead, it jointly models and infers $x^{(0)}(t)$ and its first $q$ derivatives $\{x^{(1)}(t),\allowbreak \dots,\allowbreak x^{(q)}(t) \}$ in a state space $\boldsymbol{X}$, as detailed in \Cref{sec:introduction_of_algorithm}.
The thus-computed filtering mean estimates $m^{(i)}(t)$ depend not only on the ODE but also on the statistical model---namely on the prior (SDE) and measurement noise $R$; recall \Cref{subsec:prior_on_x,subsec:measurement_noise}. 
In fact, the analytically-desirable derivative estimate, \cref{eq:analytically_desirable_derivative_estimates}, is, for $i=1$, only satisfied if $R=0$ (which can be seen from \cref{def:Psi}), and generally does not hold for $i\geq 2$ since both $f^{(i-1)}$ and $\Phi^{(i)}$ are inaccessible to the algorithm.
The numerical example in \Cref{sec:illustrative_example} clarifies that $\delta^{(i)}$ is likely to be strictly positive, even after the first step $0 \to h$.  

This inevitable mismatch, between exact analysis and approximate statistics, motivates the following definition of the $i$\textsuperscript{th} state \emph{$i$\textsuperscript{th} state misalignment at time $t$}:
\begin{align}
  \label{def:delta^i}
  \delta^{(i)}(t)
  \defeq
  \left \Vert m^{(i)}(t) - \Phi_0^{(i)}\left(m^{(0)}(t)\right)  \right \Vert
  \geq
  0.
\end{align}
Intuitively speaking, $\delta^{(i)}(t)$ quantifies how large this mismatch is for the $i$\textsuperscript{th} derivative at time $t$.
Note that $\delta^{(i)}(t) = 0$ if and only if \cref{eq:analytically_desirable_derivative_estimates} holds---i.e.~for $i=1$ iff $R=0$ (which can be seen from \cref{def:Psi}) and only by coincidence for $i\geq 2$ since both $f^{(i-1)}$ and $\Phi^{(i)}_0$ are inaccessible to the algorithm.
(Since $\Phi^{(0)}_0 = \operatorname{Id}$, $\delta^{(0)}(t)=0$ for all $t$.)

The possibility of $\delta^{(i)} > 0$, for $i\geq 1$, is inconvenient for the below worst-case analysis since (if \cref{eq:analytically_desirable_derivative_estimates} held true and $\delta^{(i)} \equiv 0$) the prediction step of the drift-less IBM prediction ($\theta=0$) would coincide with a Taylor expansions of the flow map $\Phi^{(i)}_0$; see \cref{A_prediction}.
But, because $\delta^{(i)} \neq 0$ in general, we have to additionally bound the influence of $\delta \geq 0$ which complicates the below proofs further.

Fortunately, we can \emph{locally} bound the import of $\delta^{(i)}$ by the easy \Cref{lemma:initial_delta_bound} and \emph{globally} by the more complicated \Cref{lemma:delta_global_bound} (see \Cref{subsec:global_bounds_on_state_misalignments}).
Intuitively, these bounds demonstrate that the order of the deviation from a Taylor expansion of the state $\boldsymbol{m} = [m^{(0)},\allowbreak \dots,\allowbreak m^{(q)}]$ due to $\delta$ is not smaller than the remainder of the Taylor expansion.
This means, more loosely speaking, that the import of the $\delta^{(i)}$ is swallowed by the Taylor remainder.
This effect is locally captured by \Cref{lemma:Delta^-i} and globally by \Cref{lemma:truncation_error_global_bound}.
The global convergence rates of $\delta^{(i)}(T)$, as provided by \Cref{lemma:truncation_error_global_bound}, are experimentally demonstrated in \Cref{appendix:section:delta_convergence}.

\section{Auxiliary bounds on intermediate quantities}
\label{sec:auxiliary_bounds_on_intermediate_quantities}

Recall from \cref{eq:a_restriction_to_IOUP/IBM} that $\theta = 0$ and $\theta>0$ denote the cases of IBM and IOUP prior with drift coefficient $\theta$ respectively.
The ODE filter $\Psi$ iteratively computes the filtering mean $\boldsymbol{m}(nh)=( m^{(0)}(nh),\allowbreak\dots \allowbreak, m^{(q)}(nh) )^{\intercal} \in \R^{(q+1)}$ as well as error covariance matrices $P(nh) \in \R $ on the mesh $\{ nh \}_{n=0}^{T/h}$.
(Here and in the following, we assume w.l.o.g.~that $T/h \in \mathbb{N}$.)
Ideally, the truncation error over all derivatives
\begin{align}
  \label{eq:def_varpepsilon}
  \rev{  \boldsymbol{\varepsilon}(nh) \defeq (\varepsilon^{(0)}(nh), \dots, \varepsilon^{(q)}(nh)  )^{\intercal} \defeq \boldsymbol{m}(nh) - \boldsymbol{x}(nh),  }
\end{align}
falls quickly as $h \to 0$ and is \rev{estimated} by the standard deviation $\sqrt{P_{00}(nh)}$.
Next, we present a classical worst-case convergence analysis over all $f$ satisfying \Cref{ass:f_Global_Lipschitz};
see \Cref{sec:discussion} for a discussion of the desirability and feasibility of an average-case analysis.
To this end, we bound the added error of every step by intermediate values, defined in \cref{def_beta_MM,def:r},
\begin{align}
  \label{def:Delta}
  &\Delta^{(i)}((n+1)h)
  \defeq
  \left \Vert \Psi_{P(nh),h}^{(i)}(\boldsymbol{m}(nh)) - \Phi_{h}^{(i)} \left(m^{(0)}(nh) \right)    \right \Vert
  \\
  &\stackrel{\text{\cref{def:Psi}}}{\leq}
  \underbrace{\left \Vert \left(A(h) \boldsymbol{m}(nh) \right)_i - \Phi_{h}^{(i)} \left( m^{(0)}(nh) \right) \right \Vert}_{\qefed \Delta^{-(i)}((n+1)h)}
  \notag
  \\
  &\phantom{\stackrel{\text{\cref{def:Psi}}}{\leq}}  +
  \left \Vert \beta^{(i)}(\rev{(n+1)h}) \right \Vert \left \Vert r(\rev{(n+1)h}) \right \Vert,      \label{def:Delta^-}
\end{align}
and bound these quantities in the order $\Delta^{-(i)}$, $r$, $\beta^{(i)}$.
These bounds will be needed for the local and global convergence analysis in \Cref{subsection:local_truncation_error,subsection:Global_truncation_error} \rev{respectively}.
Note that, intuitively, $\Delta^{-(i)}((n+1)h)$ and $\Delta^{(i)}((n+1)h)$ denote the \emph{additional} numerical error which is added in the $(n+1)$\textsuperscript{th} step to the $i$\textsuperscript{th} derivative of the predictive mean $m^{-,(i)}(t+h)$ and the updated mean $m^{(i)}(t+h)$, respectively.
\begin{lemma}  \label{lemma:Delta^-i}
  Under \Cref{ass:f_Global_Lipschitz}, for all $i \in [q+1]$ and all $h>0$,
  \begin{align} 
    \Delta^{-(i)}((n+1)h)
    \leq
    &K \left [ 1 +  \rev{\theta \left \Vert m^{(q)}(nh)  \right \Vert} \right ] h^{q+1-i} 
    \notag
    \\
    &+ \sum_{k=i}^q \frac{h^{k-i}}{(k-i)!}\delta^{(k)}(nh).
    \label{bound:Delta_minus_IBM/IOUP}
  \end{align}
\end{lemma}

\begin{proof}
  We may assume, as explained in \Cref{Gaussian ODE filtering}, without loss of generality that $d=1$.
  We apply the triangle inequality to the definition of $\Delta^{-(i)}((n+1)h)$, as defined in \cref{def:Delta^-}, which, by \cref{A_prediction}, yields
  \begin{align}
    \Delta&^{-(i)}((n+1)h)
    \quad
    \leq
    \label{ineq:proof_Delta-_bound}
    \\
    &\sum_{k=i}^q \frac{h^{k-i}}{(k-i)!} \delta^{(k)}(nh) + \rev{K \theta \left \vert m^{(q)}(nh)  \right \vert  h^{q+1-i}}
    \notag
    \\
    &+
    \underbrace{\left \vert \sum_{l=i}^q \frac{h^{l-i}}{(l-i)!} \Phi_{0}^{(l)}\left( m^{(0)}(nh) \right) - \Phi_h^{(i)}\left( m^{(0)}(nh) \right)  \right \vert}_{\leq Kh^{q+1-i}, \text{ by \cref{eq:Taylor_expansion}}}.
    \notag
  \end{align}
  \qed
\end{proof}

\begin{lemma}
  \label{lemma:r}
  Under \Cref{ass:f_Global_Lipschitz} and for all sufficiently small $h>0$,
  \begin{align}
    \notag  
    \left \Vert r((n+1)h) \right \Vert
    \leq
    &K \left [ 1 + \rev{\theta \left \Vert m^{(q)}(nh)  \right \Vert} \right ] h^q
    \\
    &+
    K \sum_{k=1}^q \frac{h^{k-1}}{(k-1)!} \delta^{(k)}(nh).
    \label{bound:r}
  \end{align}
\end{lemma}
\begin{proof}
  See \Cref{appendix:local_technical_lemma}.
  \qed
\end{proof}
To bound the Kalman gains $\boldsymbol{\beta}(nh)$, we first need to assume that the orders of the initial covariance matrices are sufficiently high (matching the latter required orders of the initialization error; see \Cref{ass:eps0_bound}).

\begin{assumption} \label{ass:assumption2}
  The entries of the initial covariance matrix $P(0)$ satisfy, for all $k,l \in [q+1]$, $\Vert P(0)_{k,l} \Vert \leq K_0 h^{2q+1-k-l}$, where $K_0>0$ is a constant independent of $h$.
\end{assumption}

We make this assumption, as well as \Cref{ass:eps0_bound}, explicit (instead of just making the stronger assumption of exact initializations with zero variance), because it highlights how statistical or numerical uncertainty on the initial value effects the accuracy of the output of the filter---a novel functionality of PN with the potential to facilitate a management of the computational budget across a computational chain with respect to the respective perturbations from different sources of uncertainty \citep[Section 3(d)]{HenOsbGirRSPA2015}.

\begin{lemma}
  \label{lemma:Kalman_gain_local_order}
  Under \Cref{ass:assumption2}, for all $i\in [q+1]$ and for all $h>0$, $ \Vert \beta^{(i)}(h) \Vert \leq K h^{1-i}$.
\end{lemma}

\begin{proof}
  Again, w.l.o.g.~$d=1$.
  Application of the orders of $A$ and $Q$ from \cref{def:A^IOUP,eq:Q_IOUP}, the triangle inequality and \Cref{ass:assumption2} to the definition of $P^-$ in \cref{eq:C^-_predict} yields
  \begin{align}
    \left \vert P^-(h)_{k,l} \right \vert
    &\stackrel{\text{\cref{eq:C^-_predict}}}{\leq}
    \left \vert \left [ A(h)P(0)A(h)^{\intercal} \right ]_{k,l} \right \vert + \left \vert Q\rev{(h)}_{k,l} \right \vert
    \notag
    \\
    &\stackrel{\text{eqs. }\eqref{def:A^IOUP},\eqref{eq:Q_IOUP}}{\leq}
    K 
    \Bigg [
    \sum_{a=k}^{q} \sum_{b=l}^q \left \vert P(0)_{a,b}  \right \vert h^{a+b-k-l}
    \notag
    \\
    &\phantom{\stackrel{\text{eqs. }\eqref{def:A^IOUP},\eqref{eq:Q_IOUP}}{\leq}\Bigg [ } 
    + 2\theta \sum_{b=l}^{q-1} \left \vert P(0)_{q,b} \right \vert 
    \notag
    \\
    &\phantom{\stackrel{\text{eqs. }\eqref{def:A^IOUP},\eqref{eq:Q_IOUP}}{\leq}\Bigg [ }
    + \theta^2 \left \vert P(0)_{q,q} \right \vert + h^{2q+1-k-l}
    \Bigg ]
    \notag
    \\
    &\stackrel{\text{Ass.~\ref{ass:assumption2}}}{\leq}
    K \rev{[1 + \theta + \theta^2  ]} h^{2q+1-k-l}.
    \label{eq:proof_bound_C-}
  \end{align}
  \rev{Recall that $P$ and $Q$ are (positive semi-definite) covariance matrices;
  hence, $P^-(h)_{1,1} \geq Kh^{2q-1}$.
  Inserting these orders into the definition of $\beta^{(i)}$ (\cref{def_beta_MM}), recalling that $R \geq 0$, and removing the dependence on $\theta$ by reducing the fraction conclude the proof.}
  \qed
\end{proof}

\section{Local convergence rates}
\label{subsection:local_truncation_error}

With the above bounds on intermediate algorithmic quantities (involving state misalignments $\delta^{(i)}$) in place, we only need an additional assumption to proceed---via a bound on $\delta^{(i)}(0)$---to our first main result on local convergence orders of $\boldsymbol{\Psi}$.

\begin{assumption}   \label{ass:eps0_bound}
  The initial errors on the initial estimate of the $i$\textsuperscript{th} derivative $m^{(i)}(0)$ satisfy $\Vert \varepsilon^{(i)}(0) \Vert = \Vert m^{(i)}(0) - x^{(i)}(0) \Vert \leq K_0 h^{q+1-i}$.
  (This assumption is, like \Cref{ass:assumption2}, weaker than the standard assumption of exact initializations.)
\end{assumption}

\begin{lemma} \label{lemma:initial_delta_bound}
  Under \Cref{ass:f_Global_Lipschitz,ass:eps0_bound}, for all $i\in [q+1]$ and for all $h>0$, $\delta^{(i)}(0) \leq Kh^{q+1-i}.$
\end{lemma}

\begin{proof}
  The claim follows, using \Cref{ass:f_Global_Lipschitz,ass:eps0_bound}, from
  \begin{align}
    \label{eq:delta_bound_in_proof}
    \delta^{(i)}(0)
    \leq
    &\underbrace{\left \Vert m^{(i)}(0) - x^{(i)}(0) \right \Vert}_{= \Vert \varepsilon^{(i)}(0) \Vert \leq K_0 h^{q+1-i}}
    \notag
    \\
    &+ \underbrace{\left \Vert f^{(i-1)} \left( x^{(0)}(0) \right) - f^{(i-1)} \left( m^{(0)}(0) \right) \right \Vert}_{\leq L \Vert \varepsilon^{(0)}(0) \Vert \leq L K_0 h^{q+1}}.
  \end{align}
  \qed
\end{proof}
Now, we can bound the local truncation error $\varepsilon^{(0)}(h)$ as defined in \cref{eq:def_varpepsilon}.
\begin{thm}[Local Truncation Error]   \label{lemma:Psi_with_delta}
  Under the \Cref{ass:f_Global_Lipschitz,ass:assumption2,ass:eps0_bound} and for all sufficiently small $h>0$,
  \begin{align}
    \label{bound:local_varepsilon}
    \left \Vert \varepsilon^{(0)}(h) \right \Vert
    \leq
    \vertiii{\boldsymbol{\varepsilon} (h) }_h
    \leq
    K \left [ 1 + \rev{\theta \left \Vert m^{(q)}(0)  \right \Vert} \right ] h^{q+1}.
  \end{align}
\end{thm}

\begin{proof}
  By the triangle inequality for $\vertiii{\quark}_h$ and subsequent application of \Cref{lemma:flow_map_regularity} and \Cref{ass:eps0_bound} to the second summand of the  resulting inequality, we obtain
  \begin{align}
    \vertiii{\boldsymbol{\varepsilon}(h)}_h
    \leq
    &\underbrace{\vertiii{\boldsymbol{\Psi}_{P(0),h} \left ( \boldsymbol{m}(0) \right )  - \boldsymbol{\Phi}_h \left (  x^{(0)}(0) \right )  }_h}_{= \sum_{i=0}^q h^i \Delta^{(i)}(h), \text{ by \cref{def:Delta}} }
    \notag
    \\
    &+
    \underbrace{\vertiii{\boldsymbol{\Phi}_h \left (x^{(0)}(0) \right ) - \boldsymbol{\Phi}_h \left ( m^{(0)}(0)  \right ) }_h}_{\leq (1+Kh) \Vert \varepsilon^{(0)}(0) \Vert \leq Kh^{q+1} } .
    \label{eq:proof_bound_bsymb_varepsilon}
  \end{align}
  The remaining bound on $\Delta^{(i)}(h)$, for all $i \in [q+1]$ and sufficiently small $h>0$, is obtained by insertion of the bounds from \Cref{lemma:Delta^-i,lemma:r,lemma:Kalman_gain_local_order} (in the case of $n=0$), into \cref{def:Delta^-}:
  \begin{align}
    \notag
    \Delta^{(i)}(h)
    &\leq
    K \left [ 1 + \rev{\theta \left \Vert m^{(q)}(0)  \right \Vert} \right ] h^{q+1-i}
    \\
    &\phantom{\leq}
    +
    \rev{K \sum_{k=1}^q \frac{h^{k-1}}{(k-1)!} \delta^{(k)}(nh)}
    \\
    &\stackrel{\text{Lemma } \ref{lemma:initial_delta_bound}}{\leq}
    K \left [ 1 + \rev{\theta \left \Vert m^{(q)}(0)  \right \Vert} \right ] h^{q+1-i}
    .
    \label{bound:Delta}
  \end{align}
  Insertion of \cref{bound:Delta} into \cref{eq:proof_bound_bsymb_varepsilon} and $\Vert \varepsilon^{(0)}(h) \Vert \leq \vertiii{\boldsymbol{\varepsilon}(h)}_h$ (by \cref{def:h_norm}) concludes the proof.
  \qed
\end{proof}

\begin{remark}
  \label{remark_on_local_convergence_theorem}
  Theorem \ref{lemma:Psi_with_delta} establishes a bound of order $h^{q+1}$ on the local truncation error $\varepsilon^{(0)}(h)$ on $x(h)$ after one step $h$.
  Moreover, by the definition \cref{def:h_norm} of $\vertiii{\quark}_h$, this theorem also implies additional bounds of order $h^{q+1-i}$ on the error $\varepsilon^{(i)}(h)$ on the $i$\textsuperscript{th} derivative $x^{(i)}(h)$ for all $i \in [q+1]$.
  Such derivative bounds are (to the best of our knowledge) not available for classical numerical solvers, since they do not explicitly model the derivatives in the first place.
  These bounds could be be useful for subsequent computations based on the ODE trajectory \citep{HenOsbGirRSPA2015}.

  \rev{
  Unsurprisingly, as the mean prediction (recall \cref{A_prediction}) deviates from a pure $q$\textsuperscript{th} order Taylor expansion by $K\theta \Vert m^{(q)}(0) \Vert h^{q+1}$ for an IOUP prior (i.e.~$\theta > 0$ in \cref{eq:a_restriction_to_IOUP/IBM}), the constant in front of the local $h^{q+1}$ convergence rate depends on both $\theta$ and $m^{(q)}(0)$ in the IOUP case.
  A global analysis for IOUP is therefore more complicated than for IBM:
  Recall from \cref{A_prediction} that, for $q=1$, the mean prediction for $x((n+1)h)$ is
  \begin{align}
    \label{IOUP_prediction_q1}
    &\begin{pmatrix}
      m^{-,(0)}((n+1)h)
      \\
      m^{-,(1)}((n+1)h)
    \end{pmatrix}
    \quad
    \stackrel{\text{\cref{A_prediction}}}{=}
    \\
    &\phantom{=}
    \begin{pmatrix}
      m^{(0)}(nh) + h m^{(1)}(nh) - \theta \left [ \frac{h^2}{2!} + \mathcal O(h^3) \right ] m^{(1)}(nh)
      \\
      e^{-\theta h} m^{-,(1)}(nh)
    \end{pmatrix},
    \notag
  \end{align}
  which pulls both $m^{-,(0)}$ and $m^{-,(1)}$ towards zero (or some other prior mean) compared to \rev{the prediction given by its Taylor expansion} for $\theta = 0$.
  While this is useful for ODEs converging to zero, such as $\dot x = -x$, it is problematic for diverging ODEs, such as $\dot x = x$ \citep{magnani2017}.
  As shown in Theorem \ref{lemma:Psi_with_delta}, this effect is asymptotically negligible for local convergence, but it might matter globally and, therefore, might necessitate stronger assumptions on $f$ than \Cref{ass:f_Global_Lipschitz}, such as a bound on $\Vert f \Vert_{\infty}$ which would globally bound $\{ y(nh);\ n=0,\dots,T/h \}$ and thereby $\{m^{(1)}(nh);\ n=0,\dots,T/h\}$ in \cref{IOUP_prediction_q1}.
  It is furthermore conceivable that a global bound for IOUP would depend on the relation between $\theta$ and $\Vert f \Vert_{\infty}$ in a nontrivial way.
  The inclusion of IOUP ($\theta > 0$) would hence complicate the below proofs further.
  Therefore, we restrict the following first global analysis to IBM ($\theta = 0$).  }
\end{remark}

\section{Global analysis}
\label{subsection:Global_truncation_error}

As \rev{explained} in \Cref{remark_on_local_convergence_theorem}, we only consider the case of the IBM prior\rev{, i.e.~$\theta = 0$,} in this section.
Moreover, we restrict our analysis to $q=1$ in this first global analysis.
Although we only have definite knowledge for $q=1$, we believe that the convergence rates might also hold for higher $q \in \mathbb{N}$---which we experimentally test in \Cref{subsec:exp_original_WPD}.
Moreover, we believe that proofs analogous to the below proofs might work out for higher $q \in \mathbb{N}$ and that deriving a generalized version of Proposition \ref{proposition:steady_states_global} for higher $q$ is the bottleneck for such proofs.
(See \Cref{sec:discussion} for a discussion of these restrictions.)

While, for local convergence, all noise models $R$ yield the same convergence rates in Theorem \ref{lemma:Psi_with_delta}, it is unclear how the order of $R$ in $h$ (as described in \Cref{subsec:measurement_noise}) affects global convergence rates:
E.g.,~for the limiting case $R \equiv Kh^0$, the steady-state Kalman gains $\boldsymbol{\beta}^{\infty}$ would converge to zero (see \cref{lemma:steady_states_values_v,lemma:steady_states_values_vi} below) for $h \to 0$, and hence the evaluation of $f$ would not be taken into account---yielding a filter $\boldsymbol{\Psi}$ which assumes that the evaluations of $f$ are equally off, regardless of $h>0$, and eventually just extrapolates along the prior \rev{without global convergence} of the posterior mean $\boldsymbol{m}$. 
For the opposite limiting case $R\equiv \lim_{p \to \infty} K h^p \equiv 0$, it has already been shown in \citet[Proposition 1 and Theorem 1]{schober2019} that---in the steady state and for $q=1,2$---the filter $\boldsymbol{\Psi}$ inherits global convergence rates from known multistep methods in Nordsieck form \citet{nordsieck1962numerical}.
\rev{To explore a more general noise model,} we assume a fixed noise model \rev{$R \equiv K h^p$} with arbitrary order $p$.

In the following, we analyse how small $p$ can be in order for $\boldsymbol{\Psi}$ to exhibit fast global convergence (cf.~the similar role of the order $p$ of perturbations in \citet[Assumption 1]{conrad_probability_2017} and \citet[Assumption 2.2]{AbdulleGaregnani17}).
In light of Theorem \ref{lemma:Psi_with_delta}, the highest possible global convergence rate is $\BO(h)$---which will indeed be obtained for all \rev{$p \in [1,\infty]$} in Theorem \ref{theorem:GlobalTruncation}.
Since every extrapolation step of $\boldsymbol{\Psi}$ from $t$ to $t+h$ depends not only on the current state, but also on the covariance matrix $P(t)$---which itself depends on all previous steps---$\Psi$ is neither a single-step nor a multistep method.
Contrary to \citet{schober2019}, we do not restrict our theoretical analysis to the steady-state case, but provide our results under the weaker \Cref{ass:assumption2,ass:eps0_bound} that were already sufficient for local convergence in Theorem \ref{lemma:Psi_with_delta}---which is made possible by the bounds \cref{ineq:ineq:steady_state_lemma_v,ineq:ineq:steady_state_lemma_vi} in Proposition \ref{proposition:steady_states_global}.

\subsection{Outline of global convergence proof}
\label{subsec:outline_of_global_convergence_proof}

The goal of the following sequence of proofs in \Cref{subsection:Global_truncation_error} is Theorem \ref{theorem:GlobalTruncation}.
It is proved by a special version of the discrete Gr\"onwall inequality \citep{Clark87} whose prerequisite is provided in \Cref{lemma:development_epsilon}.
This \Cref{lemma:development_epsilon} follows from \Cref{lemma:flow_map_regularity} (on the regularity of the flow map $\boldsymbol{\Phi}_t$) as well as \Cref{lemma:truncation_error_global_bound} which provides a bound on the maximal increment of the numerical error stemming from local truncation errors.
For the proof of \Cref{lemma:truncation_error_global_bound}, we first have to establish

\begin{enumerate}[(i)]
  \item global bounds on the Kalman gains $\beta^{(0)}$ and $\beta^{(1)}$ by the inequalities \cref{ineq:ineq:steady_state_lemma_v,ineq:ineq:steady_state_lemma_vi} in Proposition \ref{proposition:steady_states_global}, and
  \item a global bound on the state misalignment $\delta^{(1)}$ in \Cref{lemma:delta_global_bound}.
\end{enumerate}

In \Cref{subsec:Global_bounds_on_Kalman_gains,subsec:global_bounds_on_state_misalignments,subsec:prerequisite_for_discrete_Gronwall}, we will collect these inequalities in the order of their numbering to subsequently prove global convergence in \Cref{subsec:Global_convergence_rates}.

\subsection{Global bounds on Kalman gains}
\label{subsec:Global_bounds_on_Kalman_gains}

Since we will analyse the sequence of \rev{covariance matrices and Kalman gains using contractions in Proposition \ref{proposition:steady_states_global}}, we first introduce the following \rev{generalization of Banach fixed-point theorem (BFT)}.


\begin{lemma}
  \label{lemma:sequence_of_contractions}
  Let $(\mathcal{X},d)$ be a non-empty complete metric space,
  $T_n \colon \mathcal{X} \to \mathcal{X}$, $n\in \mathbb N$, a sequence of $L_n$-Lipschitz continuous contractions with $\sup_n L_n \leq \bar L < 1$.
  Let $u_n$ be the fixed point of $T_n$, as \rev{given} by \rev{BFT}, and let $\lim_{n \to \infty} u_n = u^{\ast} \in \mathcal{X}$.
  Then, for all $x_0 \in \mathcal{X}$, the recursive sequence $x_n \defeq T_n(x_{n-1})$ converges to $u^{\ast}$ as $n \to \infty$.
\end{lemma}

\begin{proof}
  \rev{See \Cref{Appendix:Banach}.}
  \qed
\end{proof}

In the following, we will assume that $T$ is a multiple of $h$.

\begin{prop}
  \label{proposition:steady_states_global}
  \rev{For constant $R \equiv K h^p$ with $p \in [0,\infty]$}, the unique (attractive) steady states for the following quantities are
  \begin{align}
    \label{lemma:steady_states_values_i}
    &P_{11}^{-,\infty}
    \defeq
    \lim_{n \to \infty} P_{11}^-(nh)
    \\
    &\phantom{P_{11}^{-,\infty}}
    =
    \frac 12 \left (\sigma^2 h + \sqrt{4\sigma^2 R h + \sigma^4 h^2} \right ),
    \notag
    \\
    \label{lemma:steady_states_values_ii}
    &P_{11}^{\infty}
    \defeq
    \lim_{n \to \infty} \rev{P_{11}}(nh)
    \\
    &\phantom{P_{11}^{\infty}}
    =
    \frac{\left( \sigma^2 h + \sqrt{4\sigma^2 R h + \sigma^4 h^2}     \right )R}{\sigma^2 h + \sqrt{4\sigma^2 R h + \sigma^4 h^2} + 2R},
    \notag
    \\
    \label{lemma:steady_states_values_iii}
    &P_{01}^{-,\infty}
    \defeq
    \lim_{n \to \infty} P_{01}^{-}(nh)
    \\
    &\phantom{P_{01}^{-,\infty}}
    =
    \frac {\sigma^4 h^2 + (2R + \sigma^2 h) \sqrt{4\sigma^2 Rh + \sigma^4 h^2} + 4R\sigma^2 h } {2(\sigma^2 h + \sqrt{4\sigma^2 R h + \sigma^4 h^2})} h,
    \notag
    \\
    \label{lemma:steady_states_values_iv}
    &P_{01}^{\infty}
    \defeq
    \lim_{n \to \infty} P_{01}(nh)
    \\
    &\phantom{P_{01}^{\infty}}
    =
    \frac
    {R\sqrt{4R\sigma^2  h + \sigma^4 h^2} }
    {\sigma^2 h + \sqrt{4\sigma^2 R h + \sigma^4 h^2}}
    h,
    \notag
    \\
    \label{lemma:steady_states_values_v}
    &\beta^{\infty,(0)}
    \defeq
    \lim_{n \to \infty} \beta^{(0)}(nh)
    \\
    &\phantom{\beta^{\infty,(0)}}
    =
    \frac
    {\sqrt{4R\sigma^2  h + \sigma^4 h^2} }
    {\sigma^2 h + \sqrt{4\sigma^2 R h + \sigma^4 h^2}}
    h,
    \notag
    \qquad \text{and}
    \\
    \label{lemma:steady_states_values_vi}
    &\beta^{\infty,(1)}
    \defeq
    \lim_{n \to \infty} \beta^{(1)}(nh)
    \\
    &\phantom{\beta^{\infty,(1)}}
    =
    \frac {\sigma^2 h + \sqrt{4\sigma^2 R h + \sigma^4 h^2}} { \sigma^2 h + \sqrt{4\sigma^2 R h + \sigma^4 h^2} + 2R }.
    \notag
  \end{align}
  \rev{If furthermore \Cref{ass:assumption2} holds, then,} for all sufficiently small $h>0$,
  \begin{align}
    \label{ineq:steady_state_lemma_i}
    \max_{n \in [T/h + 1]}
    P_{11}^{-}(nh)
    &\leq
    Kh^{1 \wedge \frac{p+1}{2}},
    \\
    \label{ineq:steady_state_lemma_ii}
    \max_{n \in [T/h + 1]}
    P_{11}(nh)
    &\leq
    Kh^{p \vee \frac{p+1}2 },
    \\
    \label{ineq:steady_state_lemma_iv}
    \max_{n \in [T/h + 1]}
    \left \Vert P_{01}(nh) \right \Vert
    &\leq
    Kh^{p+1},
    \\
    \label{ineq:ineq:steady_state_lemma_v}
    \max_{n \in [T/h + 1]}
    \left \Vert \beta^{(0)}(nh)  \right \Vert
    &\leq
    Kh,
    \qquad \text{and}
    \\
    \label{ineq:ineq:steady_state_lemma_vi}
    \max_{n \in [T/h + 1]}
    \left \Vert 1 - \beta^{(1)}(nh)    \right \Vert
    &\leq
    Kh^{(p-1) \vee 0}.
  \end{align}
  All of these bounds are sharp in the sense that they fail for any higher order in the exponent of $h$.
\end{prop}
\begin{remark}
  The recursions for $P(nh)$ and $P^-(nh)$ given by \cref{eq:C^-_predict,eq:C_update_MM} follow a discrete algebraic Riccati equation (DARE)---a topic studied in many related settings \citep{LancasterRodman1995}.
  While the asymptotic behavior \cref{lemma:steady_states_values_i} of the completely detectable state $X^{(1)}$ can also be obtained using classical filtering theory \citep[Chapter 4.4]{anderson1979}, the remaining statements of Proposition \ref{proposition:steady_states_global} also concern the undetectable state $X^{(0)}$ and are, to the best of our knowledge, not directly obtainable from existing theory on DAREs or filtering (which makes the following proof necessary).
  Note that, in the special case of no measurement noise ($R\equiv 0$), \cref{lemma:steady_states_values_v,lemma:steady_states_values_vi} yield the equivalence of the filter in the steady state with the P(EC)1 implementation of the trapezoidal rule, which was previously shown in \citet[Proposition 1]{schober2019}.
  For future research, it would be interesting to examine whether insertion of positive choices of $R$ into \cref{lemma:steady_states_values_v,lemma:steady_states_values_vi} can reproduce known methods as well.
\end{remark}

\begin{proof}
  \rev{See \Cref{Appendix:Proposition_Proof}.}
  \qed
\end{proof}

\subsection{Global bounds on state misalignments} \label{subsec:global_bounds_on_state_misalignments}

\rev{For the following estimates, we restrict the choice of $p$ to be larger than $q=1$.

\begin{assumption}
  \label{ass:R}
  The noise model is chosen to be $R\equiv Kh^p$, for $p \in [q,\infty] = [1,\infty]$, where $Kh^{\infty} : = 0$.
\end{assumption}}

Before bounding the added deviation of $\boldsymbol{\Psi}$ from the flow \rev{$\boldsymbol{\Phi}$ per step}, a global bound on the state misalignments defined in \cref{def:delta^i} is necessary.
\rev{The result of the following lemma is discussed in \Cref{appendix:section:delta_convergence}.}

\begin{lemma}
  \label{lemma:delta_global_bound}
  Under \Cref{ass:f_Global_Lipschitz,ass:assumption2,ass:eps0_bound,ass:R} and for all sufficiently small $h>0$,
  \begin{align}
    \label{ineq:delta}
    \max_{n \in [T/h + 1]} \delta^{(1)}(nh)
    \leq
    \rev{Kh}.
  \end{align}
\end{lemma}
\begin{proof}
  See \Cref{appendix:lemma_delta}.
  \qed
\end{proof}
See \Cref{lemma:delta_global_bound} for a experimental demonstration of \cref{eq:delta_bound_in_proof}.

\subsection{Prerequisite for discrete Gr\"onwall inequality}
\label{subsec:prerequisite_for_discrete_Gronwall}

\rev{Equipped with the above bounds, we can now prove a bound on the maximal increment of the numerical error stemming from local truncation errors which is needed to prove \cref{ineq:lemma:development_epsilon}, the prerequisite for the discrete Gr\"onwall inequality.}

\begin{lemma}  \label{lemma:truncation_error_global_bound}
  Under \Cref{ass:f_Global_Lipschitz,ass:assumption2,ass:eps0_bound,ass:R} and for all sufficiently small $h>0$,
  \begin{align}
      \label{eq:lemma_global_truncation_error}
      \max_{n \in [T/h + 1]} &\vertiii{\boldsymbol{\Psi}_{P(nh),h}\left(\boldsymbol{m}(nh)  \right) - \boldsymbol{\Phi}_h \left( m^{(0)}(nh)    \right)}_h
      \notag
      \\
      &\hspace{.5cm}\leq
      \rev{Kh^2}.
    \end{align}
\end{lemma}

\begin{proof}
  By \cref{def:h_norm}, we have
  \begin{align}
    &\vertiii{\boldsymbol{\Psi}_{P(nh),h}\left(\boldsymbol{m}(nh)  \right) - \boldsymbol{\Phi}_h \left(m^{(0)}(nh)    \right)}_h
    \notag
    \\
    &\hspace{1cm}=
    S_1(h) + h S_2(h),
    \label{bound:vertiii_leq_J0+J1}
  \end{align}
  with $S_1(h)$ and $S_2(h)$ defined and bounded by
  \begin{align}
    S_1(h)
    &\defeq
    \left \Vert \Psi^{(0)}_h \left( \boldsymbol{m}(nh)  \right)  - \Phi_h^{(0)} \left( m^{(0)}(nh)  \right)   \right \Vert
    \notag
    \\
    &\stackrel{\text{\cref{def:Delta^-}}}{\leq}
    \underbrace{\Delta^{-(0)}((n+1)h)}_{\stackrel{\text{\cref{bound:Delta_minus_IBM/IOUP}}}{\leq} Kh^2 + \delta^{(0)}(nh) + h \delta^{(1)}(nh)  }
    \notag
    \\
    &+
    \underbrace{\left \Vert \beta^{(0)}((n+1)h) \right \Vert}_{ \stackrel{\text{\cref{ineq:ineq:steady_state_lemma_v}}}{\leq} Kh}
    \underbrace{\left \Vert r((n+1)h) \right \Vert}_{ \stackrel{\text{\cref{bound:r}}}{\leq} Kh + (1+Kh)\delta^{(1)}(nh) },
    \label{eq:J_0_bound}
  \end{align}
  and, analogously,
  \begin{align}
    S_2(h)
    &\defeq
    \left \Vert \Psi^{(1)}_h \left( \boldsymbol{m}(nh)  \right)  - \Phi_h^{(1)} \left( m^{(0)}(nh)  \right)   \right \Vert
    \notag
    \\
    &\stackrel{\text{\cref{def:Delta^-}}}{\leq}
    \underbrace{\Delta^{-(1)}((n+1)h)}_{\stackrel{\text{\cref{bound:Delta_minus_IBM/IOUP}}}{\leq } Kh + \delta^{(1)}(nh) }
    \notag
    \\
    &\phantom{\stackrel{\text{\cref{def:Delta^-}}}{\leq}}
    +
    \underbrace{\left \Vert \beta^{(1)}((n+1)h) \right \Vert}_{\stackrel{\text{\cref{def_beta_MM}}}{\leq} 1  }
    \underbrace{\left \Vert r((n+1)h) \right \Vert}_{ \stackrel{\text{\cref{bound:r}}}{\leq} Kh + (1+Kh)\delta^{(1)}(nh) }
    \label{eq:J_1_bound}
  \end{align}
  Insertion of \cref{eq:J_0_bound} and \cref{eq:J_1_bound} into \cref{bound:vertiii_leq_J0+J1} yields
  \begin{align}
    \notag
    &\vertiii{ \boldsymbol{\Psi}_{P(nh),h}\left(\boldsymbol{m}(nh)  \right) - \boldsymbol{\Phi}_h \left( m^{(0)}(nh)    \right) }_h
    \\
    &\hspace{2cm}
    \leq
    Kh^2
    +
    \delta^{(0)}(nh)
    +
    Kh \delta^{(1)}(nh)
    ,
  \end{align}
  which---after recalling $\delta^{(0)}(nh) = 0$ and applying Lemma \ref{lemma:delta_global_bound} to $\delta^{(1)}(nh)$---\rev{implies \cref{eq:lemma_global_truncation_error}}.
  \qed
\end{proof}

The previous lemma now implies a suitable prerequisite for a discrete Gr\"onwall inequality.

\begin{lemma}  \label{lemma:development_epsilon}
  Under \Cref{ass:f_Global_Lipschitz,ass:assumption2,ass:eps0_bound,ass:R} and for all sufficiently small $h>0$,
  \begin{align}
    \label{ineq:lemma:development_epsilon}
    \vertiii{\boldsymbol{\varepsilon}\left((n+1)h \right)}_h
    \leq
    \rev{Kh^{2} + (1+Kh) \left \Vert \varepsilon^{(0)}(nh) \right \Vert}.
  \end{align}
\end{lemma}

\begin{proof}
  We observe, by the triangle inequality for \rev{the norm} $\vertiii{\cdot}_h$, that
  \begin{align} \notag
    &\vertiii{ \boldsymbol{\varepsilon} \left( (n+1)h \right)  }_h
    \\
    &\quad =
    \vertiii{ \boldsymbol{\Psi}_{P(nh),h} (\boldsymbol{m}(nh)) - \boldsymbol{\Phi}_h \left(x^{(0)}(nh) \right) }_h
    \notag
    \\
    &\qquad \leq
    \vertiii { \boldsymbol{\Psi}_{P(nh),h} (\boldsymbol{m}(nh)) - \boldsymbol{\Phi}_{h} \left(m^{(0)}(nh) \right) }_h
    \notag
    \\
    &\phantom{\qquad \leq \quad}
    +
     \vertiii { \boldsymbol{\Phi}_{h} \left(m^{(0)}(nh) \right) - \boldsymbol{\Phi}_{h} \left(x^{(0)}(nh) \right) }_h.
  \end{align}
  The proof is concluded by applying \Cref{lemma:truncation_error_global_bound} to the first and \Cref{lemma:flow_map_regularity} to the second summand of this bound (as well as recalling \rev{from} \cref{eq:def_varpepsilon} that $\Vert \varepsilon^{(0)}(nh) \Vert = \Vert m^{(0)}(nh) - x^{(0)}(nh) \Vert $).
  \qed
\end{proof}

\subsection{Global convergence rates}
\label{subsec:Global_convergence_rates}
\rev{With the above bounds in place, we can now prove global convergence rates.}
\begin{thm}[Global truncation error]  \label{theorem:GlobalTruncation}
  Under \Cref{ass:f_Global_Lipschitz,ass:assumption2,ass:eps0_bound,ass:R} and for all sufficiently small $h>0$,
  \begin{align}
    \label{ineq:global_convergence_rates}
    \max_{n \in [T/h + 1]} \left \Vert \varepsilon^{(0)}(nh)   \right \Vert
    \leq
    \max_{n\in [T/h + 1]}  \vertiii{\boldsymbol{\varepsilon}(nh) }_h
    \leq
    \rev{K(T)h},
  \end{align}
  \rev{where $K(T) > 0$ is a constant that depends on $T$, but not on $h$.}
\end{thm}

\begin{remark}
  \label{remark:bounded_global_convergence}
  Theorem \ref{theorem:GlobalTruncation} not only implies that the truncation error $\Vert \varepsilon^{(0)}(nh) \Vert$ on the solution of \cref{IVP} has global order $h$, but also (by \cref{def:h_norm}) that the truncation error $\Vert \varepsilon^{(1)}(nh) \Vert$ on the derivative is uniformly bounded by a constant $K$ independent of $h$.
  The convergence rate of this theorem is sharp in the sense that it cannot be improved over all $f$ satisfying \Cref{ass:f_Global_Lipschitz} \rev{since it is one order worse than the local convergence rate implied by Theorem \ref{lemma:Psi_with_delta}.}
\end{remark}

\begin{proof}
  Using $\left \Vert \varepsilon^{(0)}(nh) \right \Vert \leq \vertiii{\boldsymbol{\varepsilon}(nh)}_h$ (due to \cref{def:h_norm}), the bound \cref{ineq:lemma:development_epsilon}, a telescoping sum, and $\vertiii{\boldsymbol{\varepsilon}(0)}_h \leq Kh^2$ (by \Cref{ass:eps0_bound}), we obtain, for all sufficiently small $h>0$, that
    \begin{align}
    &\vertiii{\boldsymbol{\varepsilon} ((n+1)h) }_h - \vertiii{\boldsymbol{\varepsilon}(nh)}_h
    \notag
    \\
    &\quad \stackrel{\text{\cref{def:h_norm}}}{\leq}
    \vertiii{\boldsymbol{\varepsilon} ((n+1)h) }_h - \left \Vert \varepsilon^{(0)}(nh)  \right \Vert
    \notag
    \\
    &\quad \stackrel{\text{\cref{ineq:lemma:development_epsilon}}}{\leq}
    \rev{
    Kh^{2} + Kh \left \Vert \varepsilon^{(0)}(nh)  \right \Vert
    }
    \notag
    \\
    &\quad \stackrel{\text{\cref{def:h_norm}}}{\leq} \rev{Kh^{2} + Kh \vertiii{\boldsymbol{\varepsilon}(nh)}_h}
    \notag
    \\
    &\quad \stackrel{\text{(tel.\ sum)}}{=}
    \rev{Kh^{2} + \vertiii{\boldsymbol{\varepsilon} (0) }_h }
    \notag
    \\
    &\phantom{\stackrel{\text{(tel.\ sum)}}{=}} + Kh \sum_{l=0}^{n-1}  \left( \vertiii{\boldsymbol{\varepsilon} ((l+1)h) }_h - \vertiii{\boldsymbol{\varepsilon}(lh)}_h \right) 
    \notag
    \\
    &\stackrel{(\vertiii{\boldsymbol{\varepsilon}(0)}_h \leq Kh^2)}{\leq}
    Kh^{2} 
    \notag
    \\
    &\phantom{\stackrel{(\vertiii{\boldsymbol{\varepsilon}(0)}_h \leq Kh^2)}{\leq}}
    + Kh \sum_{l=0}^{n-1}  \left( \vertiii{\boldsymbol{\varepsilon} ((l+1)h) }_h - \vertiii{\boldsymbol{\varepsilon}(lh)}_h \right).
    \label{eq:Gronwall_prerequisite}
  \end{align}
  Now, by a special version of the discrete Gr\"onwall inequality \citep{Clark87}, if $z_n$ and $g_n$ are sequences of real numbers (with $g_n \geq 0$), $c \geq 0$ is a nonnegative constant, and if
  \begin{align}
    \label{Gronwall_inequality1}
    z_n \leq c + \sum_{l=0}^{n-1} g_l z_l, \qquad \text{ for all } n \in \mathbb{N},
  \end{align}
  then
  \begin{align*}
    z_n \leq c \prod_{l=0}^{n-1} (1+g_l) \leq c \exp \left( \sum_{l=0}^{n-1} g_l \right), \qquad \text{ for all } n \in \mathbb{N}.
  \end{align*}
  Application of this inequality to \cref{eq:Gronwall_prerequisite} with $z_n \defeq \vertiii{\boldsymbol{\varepsilon} ((n+1)h) }_h - \vertiii{\boldsymbol{\varepsilon}(nh)}_h $, $g_n \defeq Kh$, and $c \defeq Kh^{2}$ yields
  \begin{align}
    \vertiii{\boldsymbol{\varepsilon} ((n+1)h) }_h - \vertiii{\boldsymbol{\varepsilon}(nh)}_h
    &\leq
    K(T)h^{2} \exp \left( nKh \right) 
    \\
    &\stackrel{n \leq T/h}{\leq} 
    K(T)h^2.
    \label{ineq:summands_bsvarpes_differences}
  \end{align}
  By another telescoping sum argument and $\vertiii{\boldsymbol{\varepsilon}(0)}_h \leq Kh^2$, we obtain
  \begin{align}
    \notag
    \vertiii{\boldsymbol{\varepsilon}(nh)}_h
    &\stackrel{\text{(tel.\ sum)}}{=}
    \sum_{l=0}^{n-1} \left( \vertiii{\boldsymbol{\varepsilon}((l+1)h)}_h - \vertiii{\boldsymbol{\varepsilon}(lh)}_h \right)
    \\
    &\phantom{\stackrel{\text{(tel.\ sum)}}{=}}
    + \vertiii{\boldsymbol{\varepsilon}(0)}_h
    \\
    &\stackrel{\text{\cref{ineq:summands_bsvarpes_differences}}}{\leq}
    nK(T)h^2 + Kh^2
    \\
    &\stackrel{n \leq T/h}{\leq} K(T)h + Kh^2 
    \\
    &\leq 
    K(T)h  +  K h^2
    ,
  \end{align}
  for all sufficiently small $h>0$.
  Recalling that \rev{$\left \Vert \varepsilon^{(0)}(nh) \right \Vert \allowbreak \leq \allowbreak \vertiii{\boldsymbol{\varepsilon}(nh)}_h$}, by \cref{def:h_norm}, concludes the proof.
  \qed
\end{proof}

\section{Calibration of credible intervals}
\label{sec:calibration_of_credible_intervals}

In PN, one way to judge calibration of a Gaussian output $\mathcal{N}(m,V)$ is to check whether the implied 0.95 \rev{credible} interval $[m-2\sqrt{V},m+2\sqrt{V}]$ contracts at the same rate as the convergence rate of the posterior mean to the true quantity of interest. 
For the filter, this would mean that the rate of contraction of $\max_n \sqrt{P_{00}(nh)}$ should contract at \rev{the same rate as} $\max_{n \in [T/h + 1]} \Vert \varepsilon^{(0)}(nh) \Vert$ \rev{(recall its rates from Theorem \ref{theorem:GlobalTruncation}).}
Otherwise, for a higher or lower rate of the interval it would eventually be under- or overconfident, as $h \to 0$.
The following proposition shows---in light of the sharp bound \cref{ineq:global_convergence_rates} on the global error---that the \rev{credible} intervals are well calibrated in this \rev{sense if $p \in [1,\infty]$.}

\begin{thm}
  \label{theorem:contraction_of_credible_intervals}
  \rev{Under \Cref{ass:assumption2} and for $R \equiv K h^p$, $p \in [0,\infty]$, as well as} sufficiently small $h>0$,
  \begin{align}
    \label{ineq:maxP00-}
    \max_{n \in [T/h + 1]} P_{00}^- ( n h )
    &\leq
    K(T)h^{(p+1) \wedge 2}, \qquad \text{and}
    \\
    \label{ineq:maxP00}
    \max_{n \in [T/h + 1]} P_{00} ( n h )
    &\leq
    K(T)h^{(p+1) \wedge 2}.
  \end{align}
\end{thm}

\begin{proof}
  See \Cref{appendix:uncertain_calibration_theorem}.
  \qed
\end{proof}

\section{Numerical experiments}
\label{sec:experiments}

\hans{We are following Schober et al in their VDP and [Abdulle and Garegnani] and [Conrad et al] in their FHN example.}

In this section, we empirically assess the following hypotheses: 

\begin{enumerate}[(i)]
  \item the worst-case convergence rates from Theorem \ref{theorem:GlobalTruncation} hold not only for $q=1$ but also for $q \in \{2,3\}$ (see \Cref{subsec:exp_original_WPD}),
  \item the convergence rates of the credible intervals from Theorem \ref{theorem:contraction_of_credible_intervals} hold true (see \Cref{subsec:exp_Calibration_of_Credible_Intervals}), and
  \item \Cref{ass:R} is necessary to get these convergence rates (see \Cref{subsec:exp_the_effect_of_R}).
\end{enumerate}

The three hypotheses are all supported by the experiments.
These experiments are subsequently discussed in \Cref{subsec:interpretation_of_experiments}.
\Cref{appendix:section:delta_convergence} contains an additional experiment illustrating the convergence rates for the state misalignment $\delta$ from \Cref{lemma:delta_global_bound}.

\subsection{Global convergence rates for $q \in \{1,2,3\}$}
\label{subsec:exp_original_WPD}

We consider the following three test IVPs: 
Firstly, a the following linear ODE
\begin{align}
  \label{eq:ex_ODE_linear}
  \dot{x}(t) = \Lambda x(t),\ \forall t \in [0,10], 
  \\
  \textrm{ with } \Lambda = \begin{pmatrix} 0 & -\pi \\ \pi & 0 \end{pmatrix} \textrm{ and } x(0) = \left ( 0 , 1 \right )^{\intercal},
  \notag
\end{align}
and has the harmonic oscillator
\begin{align}
  x(t)
  =
  e^{t\Lambda} x(0)
  =
  \begin{pmatrix}
  -\sin(t \pi)
  &
  \cos(t \pi)
  \end{pmatrix}^{\intercal}
\end{align}
as a solution.
Secondly, the logistic equation
\begin{align}
  \label{eq:ex_ODE_logistic}
  \dot x(t)
  =
  \lambda_0 x(t)\left ( 1 - x(t)/ \lambda_1  \right )
  ,
  \
  \forall t \in [0,1.5],
  \\
  \textrm{ with }
  (\lambda_0,\lambda_1) = (3,1)
  \textrm{ and }
  x(0)
  =
  0.1,
  \notag
\end{align}
which has the logistic curve
\begin{align}
  x(t)
  =
  \frac
  {\lambda_1 \exp(\lambda_0 t) x(0)}
  {\lambda_1 + x(0) (\exp(\lambda_0 t) - 1) }.
\end{align}
And, thirdly, the FitzHugh--Nagumo model
\begin{align}
  \label{eq:ex_FHN}
  \begin{pmatrix}
    x_1(t)
    \\
    x_2(t)
  \end{pmatrix}
  =
  \begin{pmatrix}
    x_1(t) - \frac{x_1(t)}{3} - x_2(t)
    \\
    \frac{1}{\tau} \left( x_1(t) + a - bx_2(t)   \right),
  \end{pmatrix},
  \forall t \in [0,10]
\end{align}
with $(a,b,c) = (0.08, 0.07, 1.25)$ and $x(0) = (1,0)$ which does not have a closed-form solution.
Its solution, which we approximate by Euler's method with a step size of $h=10^{-6}$ for the below experiments, is depicted in \Cref{fig:fhn_solution}.
\begin{figure}
    \begin{center}
        \scalebox{.5}{\input{fig/fhntrue_solution}}
        \caption{True solution of the FitzHugh--Nagumo model, \cref{eq:ex_FHN}; $x_1$ in blue and $x_2$ in orange.\label{fig:fhn_solution}}
    \end{center}
    \vspace{-15pt}
\end{figure}
We numerically solve these three IVPs with the Gaussian ODE filter for multiple step sizes $h>0$ and with a $q$-times IBM prior (i.e.~$\theta = 0$ in \cref{eq:a_restriction_to_IOUP/IBM}) for $q \in \{1,2,3\}$ and scale $\sigma=20$.
As a measurement model, we employ the minimal $R \equiv 0$ and maximal measurement variance $R \equiv K_R h^q$ (for $h \leq 1$) which are permissible under \Cref{ass:R} whose constant $K>0$ is denoted explicitly by $K_R$ in this section.
The resulting convergence rates of global errors $\left \Vert m(T) - x(T) \right \Vert$ are depicted in a work-precision diagram in \Cref{fig:wpd_1}; cf.~\citet[Chapter II.1.4]{hairer87:_solvin_ordin_differ_equat_i} for such diagrams for Runge--Kutta methods.
Now, recall from Theorem \ref{theorem:GlobalTruncation} that, for $q=1$, the global truncation error decreases at a rate of at least $h^q$ in the worst case.
\Cref{fig:wpd_1} shows that these convergence rates of $q$\textsuperscript{th} order hold true in the considered examples for values of up to $q=3$ if $R \equiv 0$ and, for values of up to $q=3$.
\begin{figure*}
    \begin{center}
        \resizebox{.9\textwidth}{!}{\input{fig/stco_exp26}}
        \caption{\rev{Work-precision diagrams for the Gaussian ODE filter with $q$-times IBM prior, for $q \in \{1,2,3\}$, applied to the linear \cref{eq:ex_ODE_logistic}, logistic ODE \cref{eq:ex_ODE_linear} and the FitzHugh--Nagumo model. The number of function evaluations (\# Evals of $f$), which is inversely proportional to the step size $h$, is plotted in color against the logarithmic global error at the final time $T$. The (dash-)dotted gray lines visualize idealized convergence rates of orders one to four. The left and right columns employ the minimal $R\equiv 0$ and maximal measurement variance $R \equiv K_R h^q$ ($K_R = 1$) which are permissible under \Cref{ass:R}.}\label{fig:wpd_1}}
    \end{center}
    \vspace{-15pt}
\end{figure*}
In the case of $R \equiv 0$, even $(q+1)$\textsuperscript{th} order convergence rates appear to hold true for all three ODEs and $q \in \{1,2,3\}$.
Note that it is more difficult to validate these convergence rates for $q=4$, for all three test problems and small $h>0$, since numerical instability can contaminate the analytical rates.
\hans{Note that the deviation from perfect straight line convergence rates for large step-sizes is perfectly known as can, e.g., be seen in the work-precision diagrams of \citet{hairer87:_solvin_ordin_differ_equat_i} and \citet{hairer96:_solvin_ordin_differ_equat_ii}!}

\subsection{Calibration of credible intervals}
\label{subsec:exp_Calibration_of_Credible_Intervals}

To demonstrate the convergence rates of the posterior credible intervals proved in Theorem \ref{theorem:contraction_of_credible_intervals}, we now restrict our attention to the case of $q=1$, that was considered therein.
As in \Cref{subsec:exp_original_WPD}, we numerically solve the IVPs \cref{eq:ex_ODE_logistic,eq:ex_ODE_linear} with the Gaussian ODE filter with a once IBM prior with fixed scale $\sigma=1$.
We again employ the minimal $R \equiv 0$ and maximal measurement variance $R \equiv K_R h^q$ (for $h \leq 1$) which are permissible under \Cref{ass:R} as a measurement model.
\Cref{fig2:wpd_2_credible_intervals} depicts the resulting convergence rates in work-precision diagrams.
\begin{figure*}
  \begin{center}
  \resizebox{.9\textwidth}{!}{\input{fig/exp30_final_version.tex}}
  \caption{\rev{Work-precision diagrams for the Gaussian ODE filter with $q$-times IBM prior, for $q=1$, applied to the linear \cref{eq:ex_ODE_linear} and logistic ODE \cref{eq:ex_ODE_logistic} in the upper and lower row, respectively. The number of function evaluations (\# Evals of $f$), which is inversely proportional to the step size $h$, is plotted in color against the logarithmic global error at the final time $T$. The (dash-)dotted gray lines visualize idealized convergence rates of orders one and two. The dashed blue lines show the posterior standard deviations calculated by the filter. The left and right columns, respectively, employ the minimal $R\equiv 0$ and maximal measurement variance $R \equiv K_R h^q$ ($K_R = 5.00 \times 10^3$) which are permissible under \Cref{ass:R}.}\label{fig2:wpd_2_credible_intervals}}
  \end{center}
  \vspace{-15pt}
\end{figure*}
As the parallel standard deviation (std.~dev.) and $h^1$ convergence curves show, the credible intervals asymptotically contract at the rate of $h^1$ guaranteed by Theorem \ref{theorem:contraction_of_credible_intervals}.
In all four diagrams of \Cref{fig2:wpd_2_credible_intervals}, the global error shrinks at a faster rate than the width of the credible intervals.
This is unsurprising for $R \equiv 0$ as we have already observed convergence rates of $h^{q+1}$ in this case.
While this effect is less pronounced for $R \equiv K_R h^q$, it still results in underconfidence as $h \to 0$.
Remarkably, the shrinking of the standard deviations seems to be `adaptive' to the numerical error---by which we mean that, as long as the numerical error hardly decreases (up to $10^{1.75}$ evaluations of $f$), the standard deviation also stays almost constant, before adopting its $h^1$ convergence asymptotic (from $\approx 10^{2.00}$).

\subsection{Necessity of \Cref{ass:R}}
\label{subsec:exp_the_effect_of_R}

Having explored the asymptotic properties under \Cref{ass:R} in \Cref{subsec:exp_original_WPD,subsec:exp_Calibration_of_Credible_Intervals}, we now turn our attention to the question of whether this assumption is necessary to guarantee the convergence rates from Theorems \ref{theorem:GlobalTruncation} and \ref{theorem:contraction_of_credible_intervals}.
This question is of significance, because \Cref{ass:R} is weaker than the $R \equiv 0$ assumption of the previous theoretical results (i.e.~Proposition 1 and Theorem 1 in \citet{schober2019}) and it is not self-evident that it cannot be further relaxed.
To this end, we numerically solve the logistic ODE \cref{eq:ex_ODE_logistic} with the Gaussian ODE filter with a once IBM prior with fixed scale $\sigma = 1$ and measurement variance $R \equiv K_R h^{1/2}$, which is impermissible under \Cref{ass:R}, for increasing choices of $K_R$ from $0.00 \times 10^0$ to $1.00 \times 10^7$.
In the same way as in \Cref{fig2:wpd_2_credible_intervals}, the resulting work-precision diagrams are plotted in \Cref{fig3:wpd_with_non-permissible_R}.

\begin{figure*}
  \begin{center}
  \resizebox{.9\textwidth}{!}{\input{fig/exp32_Section83_final}}
  \caption{\rev{Work-precision diagrams for the Gaussian ODE filter with $q$-times IBM prior, for $q=1$ and $R\equiv K_R h^{1/2}$, applied to the logistic ODE \cref{eq:ex_ODE_logistic} for increasing values of $K_R$. The number of function evaluations (\# Evals of $f$), which is inversely proportional to the step size $h$, is plotted in blue against the logarithmic global error at the final time $T$. The (dash-)dotted gray lines visualize idealized convergence rates of orders one and two. The dashed blue lines show the posterior standard deviations calculated by the filter.}\label{fig3:wpd_with_non-permissible_R}}
  \end{center}
  \vspace{-15pt}
\end{figure*}

In contrast to the lower left diagram in \Cref{fig2:wpd_2_credible_intervals}, which presents the same experiment for $R \equiv K_R h^q$ (the maximal measurement variance permissible under \Cref{ass:R}), the rate of $h^2$, that is again observed for $K_R = 0$ in the first diagram, is already missed for $K_R = 1.00 \times 10^0$ in the second diagram.
With growing constants, the convergence rates of the actual errors as well as the expected errors (standard deviation) decrease from diagram to diagram.
In the center diagram with $K_R=3.73 \times 10^3$, the rates are already slightly worse than the $h^1$ convergence rates guaranteed by Theorems \ref{theorem:GlobalTruncation} and \ref{theorem:contraction_of_credible_intervals} under \Cref{ass:R}, whereas, for $K_R = 5.00 \times 10^3$, the convergence rates in the lower left plot of \Cref{fig2:wpd_2_credible_intervals} were still significantly better than $h^1$.
For the greater constants up to $K_R = 1.00 \times 10^7$, the rates even become significantly lower.
Notably, as in the lower right diagram of \Cref{fig2:wpd_2_credible_intervals}, the slope of the standard deviation curve matches the slope of the global error curve, as can be seen best in the lower right subfigure---thereby asymptotically exhibiting neither over- nor underconfidence.
These experiments suggest that the convergence rates from Theorems \ref{theorem:GlobalTruncation} and \ref{theorem:contraction_of_credible_intervals} do not hold in general for $R \equiv K_R h^{1/2}$.
Hence, it seems likely that \Cref{ass:R} is indeed necessary for our results and cannot be further relaxed without lowering the implied worst-case convergence rates.

\subsection{Discussion of experiments}
\label{subsec:interpretation_of_experiments}

Before proceeding to our overall conclusions, we close this section with a comprehensive discussion of the above experiments.
First and foremost, the experiments in \Cref{subsec:exp_original_WPD} suggest that Theorem \ref{theorem:GlobalTruncation}, the main result of this paper, might be generalizable to $q \in \{2,3\}$ and potentially even higher $q \in \mathbb N$---although unresolved issues with numerical instability for small step sizes prevent us from confidently asserting that these theoretical results would hold in practice for $q \geq 4$.
Moreover, we demonstrated the contraction rates of the posterior credible intervals from Theorem \ref{theorem:contraction_of_credible_intervals} and evidence for the necessity of \Cref{ass:R} in \Cref{subsec:exp_the_effect_of_R,subsec:exp_Calibration_of_Credible_Intervals}.
The asymptotics revealed by these experiments can be divided by the employed measurement model into three cases: the zero-noise case $R \equiv 0$, the permissible non-zero case $R \leq K_R h^q$ (under \Cref{ass:R}) and the non-permissible case $R \nleq K_R h^q$.
First, if $R \equiv 0$, the diagrams in the left column of \Cref{fig:wpd_1} reaffirm the $h^{q+1}$ convergence reported for $q \in \{1,2\}$ in \citet[Figure 4]{schober2019} and extend them to $q=3$ (see \Cref{sec:discussion} for a discussion on why we expect the above global convergence proofs to be extensible to $q \geq 2$)

The contraction rates of the credible intervals, for $q = 1$, appear to be asymptotically underconfident in this case as they contract faster than the error.
This underconfidence is not surprising in so far as the posterior standard deviation is a  worst-case bound for systems modeled by the prior, while the convergence proofs require smoothness of the solution of one order higher than sample paths from the prior.
This is a typical result that highlights an aspect known to, but on the margins of classic analysis:
The class of problems for which the algorithm converges is rougher than the class on which convergence order proofs operate.
How to remedy such overly-cautious UQ remains an open research question in PN as well as classical numerical analysis.

Secondly, in the case of $R>0$, as permissible under \Cref{ass:R}, the convergence rates are slightly reduced compared to the case $R \equiv 0$, exhibiting convergence between $h^q$ and $h^{q+1}$.
The asymptotic underconfidence of the credible intervals, however, is either reduced or completely removed as depicted in the right column of \Cref{fig2:wpd_2_credible_intervals}.
Thirdly, in the final case of an impermissibly large $R>0$, the $h^q$ convergence speed guaranteed by Theorem \ref{theorem:GlobalTruncation} indeed does not necessarily hold anymore---as depicted in \Cref{fig3:wpd_with_non-permissible_R}.
Note, however, that even then the convergence rate is only slightly worse than $h^q$.
The asymptotic UQ matches the observed global error in this case, as the parallel standard deviation and the $h^1$ curves in all but the upper left $R \equiv 0$ diagram show.

Overall, the experiments suggest that, in absence of statistical noise on $f$, a zero-variance measurement model yields the best convergence rates of the posterior mean.
Maybe this was expected as, in this case, $R$ only models the inaccuracy from the truncation error, that ideally should be treated adaptively \citep[Section 2.2]{Kersting2016UAI}.
The convergence rates of adaptive noise models should be assessed in future work.
As the observed convergence rates in practice sometimes outperform the proved worst-case convergence rates, we believe that an average-case analysis of the filter in the spirit of \citet{Ritter2000} may shed more light upon the expected practical performance.
Furthermore, it appears that the UQ becomes asymptotically accurate as well as adaptive to the true numerical error as soon as the $R>0$ is large enough.
This reinforces our hope that these algorithms will prove useful for IVPs when $f$ is estimated itself \citep[Section 3(d)]{HenOsbGirRSPA2015}, thereby introducing a $R>0$.
\section{Conclusions}
\label{sec:discussion}
We presented a worst-case convergence rate analysis of the Gaussian ODE filter, comprising both local and global convergence rates.
While local convergence rates of $h^{q+1}$ were shown to hold for all $q \in \mathbb N$, IBM and IOUP prior as well as any noise model $R \geq 0$, our global convergence results is restricted to the case of $q=1$, IBM prior and fixed noise model $R \equiv Kh^p$ with \rev{$p \in [1,\infty]$.
While a} restriction of the noise model seems inevitable, we believe that the other two restrictions can be lifted:
In light of Theorem \ref{lemma:Psi_with_delta}, global convergence rates for the IOUP prior might only require an additional assumption that ensures that all possible data sequences $\{y(nh);n=1,\dots,T/h\}$ (and thereby all possible $q$\textsuperscript{th}-state sequences $\{m^{(q)}(nh);n=0,\dots,T/h\}$) remain uniformly bounded \rev{(see discussion in \Cref{remark_on_local_convergence_theorem}).}
For the case of $q \geq 2$, it seems plausible that a proof analogous to the presented one would already yield glo\-bal convergence rates of order $h^q$\rev{,\footnote{According to \citet{Loscalzo67}, the filter might, however, suffer from numerical instability for high choices of $q$. (See \citet[Section 3.1]{schober2019} for an explanation of how such results on spline-based methods concern the ODE filter.)} as suggested for $q \in \{2,3\}$ by the experiments in \Cref{subsec:exp_original_WPD}.}

The orders of the predictive \rev{credible} intervals can also help to intuitively explain the threshold of $p=1$ \rev{(or maybe more generally: $p=q$; see \Cref{fig:wpd_1}) below which the performance} of the filter \rev{is not as good, due to} \cref{ineq:steady_state_lemma_i,ineq:steady_state_lemma_ii,ineq:steady_state_lemma_iv,ineq:ineq:steady_state_lemma_v,ineq:ineq:steady_state_lemma_vi}:
According to \citet[Equation (20)]{Kersting2016UAI}, the `true' (push-forward) variance on $y(t)$ given the predictive distribution $\mathcal{N}(m^-(t),\allowbreak P^-(t))$ is equal to the integral of $f f^{\intercal}$ with respect to $\mathcal{N}(m^-(t),\allowbreak P^-(t))$, whose maximum over all time steps, by \cref{ineq:maxP00-}, has order $\BO(h^{\frac{p+1}2 \wedge 1})$ if $ff^{\intercal}$ is globally Lipschitz---since $P^-(t)$ enters the argument of the integrand $f f^{\intercal}$, after a change of variable, only under a square root.
Hence, the added `statistical' noise $R$ on the evaluation of $f$ is of lower order than the accumulated `numerical' variance $P^-(t)$ (thereby preventing numerical convergence) if and only if $p<1$.
Maybe this, in the spirit of \citet[Subsection 3(d)]{HenOsbGirRSPA2015}, can serve as a criterion for vector fields $f$ that are too roughly approximated for a numerical solver to output a trustworthy result, even as $h \to 0$.

\rev{Furthermore, the} competitive practical performance of the filter, as \rev{numerically} demonstrated in \citet[Section 5]{schober2019}, might only \rev{be} completely captured by an average-case analysis in the sense of \citet{Ritter2000}\rev{, where the average error is computed with respect to some distribution $p(f)$, i.e.~over a distribution of ODEs.
To comprehend this idea,} recall that the \rev{posterior filtering mean is the Bayes estimator with minimum mean squared error} in linear dynamical systems with \rev{Gauss--Markov prior} (as defined by \rev{the SDE} \cref{SDE}), i.e.~when the data is not evaluations of $f$ but real i.i.d.~measurements, \rev{as well as} in the special case of $\dot x(t) = f(t)$, when the IVP simplifies to a quadrature problem---see \rev{\citet{Solak_2003}} and \citet[Section 2.2]{o1991bayes} respectively.
In fact, the entire purpose of the update step is to correct the prediction in the (on average) \rev{correct} direction, while a worst-case analysis must assume that it corrects in the worst possible direction in every \rev{step---which} we execute by the application of the triangle inequality in \cref{def:Delta^-} resulting in a worst-case upper bound that is the sum of the worst-case errors from prediction and update step.
An analysis of the probabilities of `good' vs.~`bad' updates might therefore pave the way for such an average-case analysis in the setting of this paper.
Since, in practice, truncation errors of ODE solvers tend to be significantly smaller than the worst \rev{case---as mirrored by the experiments in \Cref{sec:experiments}---such} an analysis might be useful for applications.

Lastly, we hope that the presented convergence analysis can lay the foundations for similar results for the novel ODE filters (extended KF, unscented KF, particle filter) introduced in \citet{TronarpKSH2019}, and can advance the research on uncertainty-aware likelihoods for inverse problems by ODE filtering \citep[Section 3]{KerstingKraemer_godef_inverse_2020}.

\begin{acknowledgements}
  The authors are grateful to Han Cheng Lie for discussions and feedback to early versions of what is now \Cref{sec:regularity_of_flow,sec:auxiliary_bounds_on_intermediate_quantities} of this work, as well as \Cref{subsec:Global_convergence_rates}.
  The authors also thank Michael Schober for valuable discussions and helpful comments on the manuscript.

  \rev{TJS's work has been} partially supported by the Freie Universit\"{a}t Berlin within the Excellence Initiative of the German Research Foundation (DFG), by the DFG through grant CRC 1114 ``Scaling Cascades in Complex Systems'', and by the National Science Foundation under grant DMS-1127914 to the Statistical and Applied Mathematical Sciences Institute (SAMSI) and SAMSI's QMC Working Group II ``Probabilistic Numerics''.
  \rev{HK and PH gratefully acknowledge financial support by the German Federal Ministry of Education and Research through BMBF grant 01IS18052B (ADIMEM).
  PH also gratefully acknowledges support through ERC StG Action 757275 / PANAMA.}
  
  Any opinions, findings, and conclusions or recommendations expressed in this article are those of the authors and do not necessarily reflect the views of the above-named institutions and agencies.
\end{acknowledgements}

%
\section*{Conflict of interest}

The authors declare that they have no conflict of interest.

\bibliographystyle{spbasic}       

\bibliography{bibfile.bib}

\appendix

\section{Derivation of $A$ and $Q$}
\label{appendix:section_A_and_Q}
As derived in \citet[\rev{S}ection 2.2.6]{sarkka2006thesis} the solution of the SDE \cref{SDE}, i.e.
\begin{align} \label{SDE_appendix}
  \rd \rev{\boldsymbol{X}(t)}
  &=
  \begin{pmatrix} \rd \rev{X^{(0)}(t)} \\ \vdots \\ \rd \rev{X^{(q-1)}(t)} \\ \rd \rev{X^{(q)}(t)} \end{pmatrix}
  \\
  &=
  \underbrace{\begin{pmatrix} 0 & 1 & 0 \dots &  0 \\ \vdots & \ddots & \ddots & 0 \\ \vdots &  & \ddots & 1 \\ c_0 & \dots & \dots & c_q \end{pmatrix}}_{\qefed F} \underbrace{\begin{pmatrix} \rev{X^{(0)}(t)} \\ \vdots \\ \rev{X^{(q-1)}(t)} \\ \rev{X^{(q)}(t)} \end{pmatrix}}_{= \rev{\boldsymbol{X} (t) }} \ \rd t + \underbrace{\begin{pmatrix} 0 \\ \vdots \\ 0 \\ \sigma   \end{pmatrix}}_{\qefed L}\ \rd B\rev{(t)},
  \notag
\end{align}
\rev{where we omitted the index $j$ for simplicity, is a Gauss--Markov} process with mean $m(t)$ and covariance matrix $P(t)$ given by
\begin{align}
  m(t) = A(t) m(0),
  \qquad
  P(t) = A(t)P(0) A(t)^{\intercal} + Q\rev{(t)},
\end{align}
where the matrices \rev{$A,\ Q \in \R^{(q+1)\times (q+1)}$} are explicitly defined by
\begin{align}
  A(t) =& \exp(tF),
  \\
  Q\rev{(t)} 
  \defeq&  
  \int_0^t \exp(F(\rev{t}-\tau))\rev{L L^{\intercal}} \exp(F(\rev{t}-\tau))^{\intercal}\ \rd\tau. \label{matrix:Q}
\end{align}

\rev{Parts of the following calculation can be found in \citet{magnani2017}.}
If we choose \rev{$c_0, \dots, c_{q-1} = 0$ and $c_q= - \theta$ (for $\theta \geq 0$) in \cref{SDE_appendix}} the unique strong solution of the SDE is a $q$-times \rev{IOUP, if $\theta>0$, and a $q$-times IBM, if $\theta=0$; see e.g.~\citet[Chapter 5: Example 6.8]{karatzas1991brownian}.
By} \cref{matrix:Q} and
\begin{align}
  \rev{
  \begin{pmatrix} (tF)^k \end{pmatrix}_{i,j}
  =
  t^k
  \left [
  \one_{j-i=k} + (- \theta)^{k+i-q} \one_{\{j=q,\ i+k\geq q\}}
  \right ],   }
\end{align}
\rev{it} follows that
\begin{align}
  \label{A_IOUP_appendix}
  A(t)_{ij}
  &=
  \begin{pmatrix} \sum_{k=0}^{\infty} \frac {(tF)^k} {k!} \end{pmatrix}_{i,j}
  \\
  &=
  \begin{cases}
    \one_{i\leq j} \frac{t^{j-i}}{(j-i)!}, & \mbox{if }j\neq q, \\ \frac{1}{(-\theta)^{q-i}} \sum_{k=q-i}^{\infty} \frac{(- \theta t)^k}{k!} , & \mbox{if }j=q,
  \end{cases}
  \notag
  \\
  &\rev{=
  \begin{cases}
    \one_{i\leq j} \frac{t^{j-i}}{(j-i)!}, & \mbox{if }j\neq q, \\ \frac{t^{q-i}}{(q-i)!} - \theta \sum_{k=q+1-i}^{\infty} \frac{(-\theta)^{k+i-q-1} t^k}{ k! } , & \mbox{if }j=q.
  \end{cases} }
  \notag
\end{align}
\rev{Analogously, it follows that}
\begin{align}
  \label{F_IOUP_appendix}
  &\exp(F(t-\tau))
  \\
  &\ = \begin{cases}
  \one_{i\leq j} \frac{(t-\tau)^{j-i}}{(j-i)!}, & \mbox{if }j\neq q, \\ \frac{(t-\tau)^{q-i}}{(q-i)!} - \theta \sum_{k=q+1-i}^{\infty} \frac{(-\theta)^{k+i-q-1} (t-\tau)^k}{ k! } , & \mbox{if }j=q,
  \end{cases}.
  \notag
\end{align}
If we insert \rev{\cref{F_IOUP_appendix} into \cref{matrix:Q}}, \rev{then} we obtain, by the sparsity of $L$, that
\begin{align}
  &{Q(t)}_{ij}
  \\
  \notag
  &\ =
  \frac{\sigma^2}{\rev{(-\theta)}^{2q-i-j}} \int_0^t \left ( \sum_{k=q-i}^{\infty} \frac{(-\theta \tau)^k}{k!}  \right) \left( \sum_{l=q-j}^{\infty} \frac{(-\theta \tau)^l}{l!}  \right) \ \rd \tau,
\end{align}
\rev{and the dominated convergence theorem (with dominating function $\tau \mapsto e^{2\theta \tau}$) yields
\begin{align}
  Q(t)_{ij}
  &=
  \frac{\sigma^2}{(-\theta)^{2q-i-j}} \sum_{k=q-i}^{\infty} \sum_{l=q-j}^{\infty} \int_0^t \frac{(-\theta \tau)^{k+l}}{k!l!} \ \rd \tau
  \notag
  \\
  &=
  \frac{\sigma^2}{(-\theta)^{2q-i-j}} \sum_{k=q-i}^{\infty} \sum_{l=q-j}^{\infty} (-\theta)^{k+l} \frac{ t^{k+l+1}}{(k+1+l)k!l!}.
  \label{Q_IOUP_appendix}
\end{align} }

Now, by extracting the first term and noticing that the rest of the series is in $\Theta(t^{2q+2-i-j})$, it follows that
\begin{align}
  Q(t)_{ij}
  &=
  \sigma^2 \frac{t^{2q+1-i-j}}{(2q+1-i-j)(q-i)!(q-j)!}
  \notag
  \\  
  &\phantom{=} + \Theta\left(t^{2q+2-i-j}\right).
\end{align}

\section{Extension to $x$ with dependent dimensions}
\label{appendix:dependent_dimensions}

The algorithm in \Cref{Gaussian ODE filtering} employs a prior $\boldsymbol X$ with independent dimensions $\boldsymbol{X}_j = \left( X_j^{(0)}, \dots, X_j^{(q)} \right)^ \intercal$, $j \in [d]$, by \cref{SDE}.
While this constitutes a loss of generality for our new theoretical results, which do not immediately carry over to the case of $x$ with dependent dimensions, it is not a restriction to the class of models the algorithm can employ.
To construct such a prior $\boldsymbol X$, we first stack its dimensions into the random vector $\boldsymbol X = (\boldsymbol X_0^{\intercal},\dots,\boldsymbol X_{d-1}^{\intercal})^{\intercal}$, choose symmetric positive semi-definite matrices $K_x,K_{\varepsilon} \in \mathbb R^{d \times d}$, and define, using the Kronecker product $\otimes$, its law according to the SDE
\begin{align}
  \label{eq:SDE_with_dependent_dimensions}
  \rd \boldsymbol{X}(t)
  =
  \left [ K_x \otimes F \right ] \boldsymbol{X}(t)\, \rd t
  +
  \left [ K_{\varepsilon} \otimes L \right ] \, \rd B(t),
\end{align}
with initial condition $\boldsymbol X(0) \sim \mathcal{N}(m(0),P(0))$, mean $m(0) \in \mathbb R^{d(q+1)}$ and covariance matrix $P(0) \in \mathbb R^{d(q+1) \times d(q+1)}$, as well as an underlying $d$-dimensional Brownian motion $B$ (independent of $\boldsymbol{X}(0)$).
Now, insertion of $K_x \otimes F$ and ${K_{\varepsilon}} \otimes L$ for $F$ and $L$ into \cref{matrix:Q} yields new predictive matrices $\tilde A$ and $\tilde Q$.
If we now choose $K_x = I_d $ and $K_{\varepsilon} = I_d$, substitute $\tilde A$ and $\tilde Q$ for $A$ and $Q$ in \cref{eq:def_predictive_mean,eq:C^-_predict}, and use the $d(q+1)$-dimensional GP $\boldsymbol{X}$ from \cref{eq:SDE_with_dependent_dimensions} with $m(0) \in \mathbb R^{d(q+1)}$ and $P(0) \in \mathbb R^{d(q+1) \times d(q+1)}$ as a prior, we have equivalently defined the version of Gaussian ODE filtering with independent dimensions from \Cref{Gaussian ODE filtering}.
If we, however, choose different symmetric positive semi-definite matrices for $K_x$ and $K_{\varepsilon}$, we introduce, via $\tilde A$ and $\tilde Q$, a correlation in the development of the solution dimensions $(x_0,\dots,x_{d-1})^{\intercal}$ as well as the error dimensions $(\varepsilon_0,\dots,\varepsilon_d)^{\intercal}$ respectively.
Note that, while $K_{\varepsilon}$ plays a similar role as $C^h$ in \citet[Assumption 1]{conrad_probability_2017} in correlating the numerical errors, the matrix $K_x$ additionally introduces a correlation of the numerical estimates, that is $m$, along the time axis.
Even more flexible correlation models (over all modeled derivatives) can be employed by inserting arbitrary matrices (of the same dimensionality) for $K_x \otimes F$ and ${K_{\varepsilon}} \otimes L$ in \cref{eq:SDE_with_dependent_dimensions}, but such models seem hard to interpret.
For future research, it would be interesting to examine whether such GP models with dependent dimensions are useful in practice.
There are first publications \citep{biol2018bq,gessner2019bq} on this topic for integrals, but not yet for ODEs.

\section{Illustrative example}
\label{sec:illustrative_example}

To illustrate the algorithm defined in \Cref{Gaussian ODE filtering}, we apply it to a special case of the Riccati equation \citep[p.~73]{Davis62}
\begin{align} \label{eq:Riccati_ODE}
  \frac{\rd x}{\rd t} (t)
  =
  f(x(t))
  &=
  - \frac{(x(t))^3}{2},
  \quad
  x(0)
  =
  1,
  \\
  \Big(\textrm{solution: } x(t) 
  &= 
  (t+1)^{-1/2} \Big),
\end{align}
with step size $h=0.1$, measurement noise $R=0.0$ (for simplicity) as well as prior hyperparameters $q=1$, $\sigma^2 = 10.0$ and $c_i = 0$ for all $i \in [q+1]$ (recall \cref{SDE}), i.e.~with a 1-times integrated Brownian motion prior whose drift and diffusion matrices are, by \cref{A_prediction}, given by
\begin{align}
  A(h)
  =
  \begin{pmatrix}
    1 & h \\
    0 & 1
  \end{pmatrix},
  \qquad
  Q(h)
  =
  \begin{pmatrix}
    1/300 & 1/20 \\
    1/20 & 1
  \end{pmatrix}.
\end{align}

As the ODE \cref{eq:Riccati_ODE} is one-dimensional (i.e.~$d=1$), the dimension index $j \in [d]$ is omitted in this section.
Since the initial value and derivative are certain at $x(0) = 1$ and $\dot x(0) = f(x_0) = -1/2$, our prior GP is initialized with a Dirac distribution (i.e.~$\boldsymbol{X}(0) = ( X^{(0)}(0), X^{(1)}(0) )^{\intercal}  \sim \delta_{(x_0,f(x_0))} = \delta_{(1,-1/2)}$).
Therefore, $\boldsymbol{m}(0) = (1,-1/2)^{\intercal}$ and $P(0) = 0 \in \R^{2 \times 2}$ for the initial filtering mean and covariance matrix.
Now, the Gaussian ODE Filter computes the first integration step by executing the prediction step \cref{eq:def_predictive_mean,eq:C^-_predict}
\begin{align}    
  \boldsymbol{m}^-(h)
  &=
  A(h) \boldsymbol{m}^-(0)
  \notag
  \\
  &=
  \left( m^{(0)}(0) + h m^{(1)}(0),m^{(1)}(0) \right )^{\intercal}
  \notag
  \\
  &=
  \left( 19/20, -1/2 \right )^{\intercal}, \qquad \text{and}
  \label{concrete_example_prediction:mean}
  \\
  P^-(h)
  &=
  0+Q(h)
  =
  \begin{pmatrix}
    1/300 & 1/20 \\
    1/20 & 1
  \end{pmatrix}.
  \label{concrete_example_prediction:covariance}
\end{align}
Note that, for all $i \in [q+1]$, $m^{-,(i)}(h)$ is obtained by a $(q-i)$\textsuperscript{th}-order Taylor expansion of the state $\boldsymbol{m}(0) = (x_0,f(x_0))^{\intercal} \in \mathbb R^{q+1}$. 
Based on this prediction, the data is then generated by
\begin{align} 
  y(h)
  &=
  f\left(m^{-,(0)}(h)\right)
  \stackrel{\text{\cref{concrete_example_prediction:mean}}}{=}
  f(19/20)
  \notag
  \\
  &\stackrel{\text{\cref{eq:Riccati_ODE}}}{=}
  -6859/16000
  \label{eq:concrete_example_y}
\end{align}
with variance $R = 0.0$.
In the subsequent update step \cref{def_beta_MM,eq:def_y_MAP,def:r,eq:def_predictive_mean}, a Bayesian conditioning of the predictive distribution \cref{concrete_example_prediction:mean,concrete_example_prediction:covariance} on this data is executed:
\begin{align}
  \boldsymbol{\beta} (h)
  &=
  \left( \beta^{(0)}(h), \beta^{(1)}(h)  \right)^{\intercal}
  \notag
  \\
  &=
  \left ( \frac{P^-(h)_{01}}{(P^-(h))_{11} + R}, \frac{P^-(h)_{11}}{(P^-(h))_{11} + R}  \right )^{\intercal}
  \notag
  \\
  &\stackrel{\text{\cref{concrete_example_prediction:covariance}}}{=}
  \left ( \frac{1}{20}, 1  \right )^{\intercal},
  \label{concrete_example:beta}
  \\
  r(h)
  &=
  y(h) - m^{-,(1)}(h)
  \notag
  \\
  &\stackrel{\text{eqs. }\eqref{concrete_example_prediction:mean},\eqref{eq:concrete_example_y}}{=}
  -6859/16000 + 1/2
  \notag
  \\
  &=
  1141/16000
  ,
  \label{concrete_example:r}
  \\
  \boldsymbol{m}(h)
  &\stackrel{\text{\cref{eq:def_predictive_mean}}}{=}
  \begin{pmatrix}
    m^{-,(0)}(h) + \beta^{(0)}(h) r(h)
    \\
    m^{-,(1)}(h) + \beta^{(1)}(h) r(h)
  \end{pmatrix}
  \notag
  \\
  &\stackrel{\text{eqs. }\eqref{concrete_example_prediction:mean},\eqref{concrete_example:beta},\eqref{concrete_example:r}}{=}
  \begin{pmatrix} 305141/320000 \\ -6859/16000 \end{pmatrix},
  \label{eq:conrete_example_update_step}
\end{align}
which concludes the step from $0$ to $h$.
The next step $h \to 2h$ starts with computing $m^{-,(i)}(2h)$ by a $(q-i)$\textsuperscript{th}-order Taylor expansion of the $i$\textsuperscript{th} state $m^{(i)}(h)$, for all $i \in [q+1]$.
Note that, now, there is 
a non-zero \emph{state misalignment} (recall \cref{def:delta^i}):
\begin{align}
  \delta^{(1)}(h)
  &\overset{\cref{def:delta^i}}{=}
  \left \vert m^{(1)}(h) - f\left ( m^{(0)}(h) \right )  \right \vert
  \\
  &=
  \left \vert - \frac{6859}{16000} - \frac{1}{2} \left(\frac{305141}{320000} \right )^3  \right \vert
  \\
  &\approx
  0.00485
  >0
\end{align}
which confirms the exposition on the possibility of $\delta^{(i)} > 0$ from \Cref{sec:role_of_state_misalignment}.
Note that $\delta$ tends to increase with $R$; e.g., if $R=1.0$ in the above example, then $\delta^{(1)}(h) \approx 0.03324$.

\section{Experiment: Global convergence of state misalignments $\delta$}
\label{appendix:section:delta_convergence}

\begin{figure}
  \begin{center}
  \resizebox{.5\textwidth}{!}{\input{fig/exp33_AppendixB_qmax=4_final}}
  \vspace{-20pt}
  \caption{Work-precision diagram plotting the number of function evaluations (\# Evals of $f$) against the final state misalignment $\delta^{(1)}(T)$; cf.~\Cref{fig:wpd_1}.\label{fig:state_misalignment_wpd}}
  \end{center}
  \vspace{-15pt}
\end{figure}

\Cref{fig:state_misalignment_wpd} depicts the global convergence of the state misalignment $\delta^{(1)}(T)$ in the above example \cref{eq:Riccati_ODE}, as detailed in \Cref{sec:illustrative_example}, for different choices of $q$.
The plotting is analogous to \Cref{fig:wpd_1}.
The resulting convergence rates of $h^{q+1}$ confirm \Cref{lemma:delta_global_bound} and suggest that it may also be generalizable to $q \geq 2$.

\section{Proof of \Cref{eq:Phi_i_=_f_i-1}}
\label{subsec:Proof_of_Eq}

We prove the stronger statement
\begin{align}
  \label{eq:appendix_more_general}
  \Phi_t^{(i+1)}(a)
  =
  f^{(i)} \left ( \Phi_t^{(0)}(a)  \right ),
\end{align}
from which \cref{eq:Phi_i_=_f_i-1} follows by inserting $t=0$ and $\Phi_0^{(0)}(a) = a$.
Hence, it remains to show \cref{eq:appendix_more_general}.

\begin{proof}[of \cref{eq:appendix_more_general}]
  By induction over $i \in \{0,\dots,q\}$.
  The base case $(i=0)$ is obtained using the fundamental theorem of calculus and $f^{(1)} = f$:
  $
    \Phi_t^{(1)}(a)
    =
    f\left(\Phi_t^{(0)}(a) \right)
    =
    f^{(1)}\left(\Phi_t^{(0)}(a) \right).
  $
  For the inductive step $(i-1) \to i$, we conclude (using the inductive hypothesis (IH), the chain rule (CR), the base case (BC) and $f^{(i)} = \nabla_x f^{(i-1)} \cdot f $) that
  \begin{align}
    \Phi_t^{(i+1)}(a)
    &=
    \frac{\rd}{\rd t} \Phi_t^{(i)}(a)
    \notag
    \\
    &\stackrel{\text{(IH)}}{=}
    \frac{\rd}{\rd t} f^{(i-1)}\left(\Phi_t^{(0)}(a) \right)
    \notag
    \\
    &\stackrel{\text{(CR)}}{=}
    \nabla_x f^{(i-1)}\left(\Phi_t^{(0)}(a) \right) \frac{\rd}{\rd t} \Phi_t^{(0)}(a)
    \notag
    \\
    &=
    \nabla_x f^{(i-1)}\left(\Phi_t^{(0)}(a) \right) \cdot f \left(\Phi_t^{(0)}(a) \right)
    \notag
    \\
    &=
    \left [ \nabla_x f^{(i-1)} \cdot f \right ] \left(\Phi_t^{(0)}(a) \right)
    \notag
    \\
    &\stackrel{\text{(BC)}}{=}
    f^{(i)} \left ( \Phi_t^{(0)}(a)  \right ).
  \end{align}
  \qed
\end{proof}

\section{Proof of \Cref{lemma:r}}
\label{appendix:local_technical_lemma}

\begin{proof}
  Again, w.l.o.g.~$d=1$.
  Recall that, by \cref{def:r}, $r$ is implied by the values of $m^{-,(0)}$ and $m^{-,(1)}$.
  By insertion of 
  \begin{align}
    &m^{-,(i)}((n+1)h)
    \notag
    \\
    &\quad =
    \sum_{k=i}^q \frac{h^{k-i}}{(k-i)!} m^{(k)}(nh) + K \theta \left \vert m^{(q)}(nh) \right \vert h^{q+1-i}
  \end{align}
  (due to \cref{A_prediction,def:Psi}) into the definition \cref{def:r} of $r((n+1)h)$, we obtain the following equality which we then bound by repeated application of the triangle inequality:
  \begin{align}
    \notag
    &\left \vert r((n+1)h) \right \vert
    =
    \Bigg \vert f\left(\sum_{k=0}^q \frac{h^k}{k!} m^{(k)}(nh)
    +
    \rev {K \theta \left \vert m^{(q)}(nh) \right \vert  h^{q+1} } \right)
    \\
    &
    \hspace{2cm}- \left( \sum_{k=1}^q \frac{h^{k-1}}{(k-1)!} m^{(k)}(nh) + \rev{K \theta \left \vert m^{(q)}(nh) \right \vert h^{q}} \right) \Bigg \vert
    \notag
    \\
    &\quad \leq
    \Bigg \vert f\left(\sum_{k=0}^q \frac{h^k}{k!} m^{(k)}(nh) + \rev{K \theta \left \vert m^{(q)}(nh) \right \vert h^{q+1}} \right)
    \notag
    \\ &\qquad \quad
    - 
    \left( \sum_{k=1}^q \frac{h^{k-1}}{(k-1)!} m^{(k)}(nh) \right) \Bigg \vert
    \ + \
    \rev{K \theta \left \vert m^{(q)}(nh) \right \vert h^{q}}
    \notag
    \\
    & \stackrel{\text{\cref{def:delta^i}}}{\leq}
    I_1(h) + I_2(h) + I_3(h) 
    \notag
    \\ &\quad \phantom{\stackrel{\text{\cref{def:delta^i}}}{\leq}}
    + \sum_{k=1}^q \frac{h^{k-1}}{(k-1)!} \delta^{(k)}(nh)
    + \rev{K \theta \left \vert m^{(q)}(nh) \right \vert h^{q}},
    \label{bound:r_proof}
  \end{align}
  \hans{
  \emph{Expanded version of \cref{bound:r_proof}}
  \begin{align}
    \left \vert r((n+1)h) \right \vert
    =
    \Bigg \vert f\left(\sum_{k=0}^q \frac{h^k}{k!} m^{(k)}(nh) + K \theta \left \vert m^{(q)}(nh) \right \vert h^{q+1} \right)
    \\
    - \left( \sum_{k=1}^q \frac{h^{k-1}}{(k-1)!} m^{(k)}(nh) + K \theta \left \vert m^{(q)}(nh) \right \vert h^{q} \right) \Bigg \vert
    \\
    \leq
    \underbrace{\left \vert f\left(\sum_{k=0}^q \frac{h^k}{k!} m^{(k)}(nh) + K \theta \left \vert m^{(q)}(nh) \right \vert h^{q+1} \right)
        - \left( \sum_{k=1}^q \frac{h^{k-1}}{(k-1)!} m^{(k)}(nh) \right) \right \vert}_{\leq I_1(h) + I_2(h) + I_3(h) + \sum_{k=1}^q \frac{h^{k-1}}{(k-1)!} \delta^{(k)}(nh), \text{ by } \cref{def:delta^i} }
    + K \theta \left \vert m^{(q)}(nh) \right \vert h^{q}
  \end{align}
   }
where $I_1$, $I_2$, and $I_3$ are defined and bounded as follows, using \Cref{ass:f_Global_Lipschitz} and \Cref{lemma:Taylor_expansion}:
  \begin{align}
    I_1(h)
    &\defeq
    \Bigg \vert f \left( \sum_{k=0}^q \frac{h^k}{k!} m^{(k)}(nh) + \rev{K \theta \left \vert m^{(q)}(nh) \right \vert h^{q+1}} \right)
    \notag
    \\
    &\phantom{defeq \Bigg \vert}
    - f\left( \sum_{k=0}^q \frac{h^k}{k!} \Phi_0^{(k)}\left( m^{(0)}(nh) \right) \right)   \Bigg \vert
    \notag
    \\
    &\leq
    L \sum_{k=0}^q \frac{h^k}{k!} \delta^{(k)}(nh)
    +
    \rev{L K \theta \left \vert m^{(q)}(nh) \right \vert h^{q+1}},
    \label{bound:I_1}
  \end{align}
  \begin{align}
    I_2(h)
    &\defeq
    \Big \vert f \left( \sum_{k=0}^q \frac{h^k}{k!} \Phi_0^{(k)}\left( m^{(0)}(nh) \right) \right) - f\left( \Phi_h^{(0)} \left( m^{(0)}(nh)   \right) \right) \Big \vert
    \notag
    \\
    &\leq
    L \left \vert \sum_{k=0}^q \frac{h^k}{k!} \Phi_0^{(k)}\left( m^{(0)}(nh) \right) - \Phi_h^{(0)} \left(m^{(0)} (nh) \right)   \right \vert
    \notag
    \\
    &\stackrel{\text{\cref{eq:Taylor_expansion}}}{\leq}
    Kh^{q+1}, \label{bound:I_2}
  \end{align}
  and
  \begin{align}
    I_3(h)
    &\defeq
    \left \vert \Phi_h^{(1)} \left( m^{(0)}(nh)  \right) - \sum_{k=1}^q \frac{h^{k-1}}{(k-1)!} \Phi_0^{(k)} \left( m^{(0)}(nh) \right) \right \vert
    \notag
    \\
    &\stackrel{\text{\cref{eq:Taylor_expansion}}}{\leq}
    Kh^q. \label{bound:I_3}
  \end{align}
  Inserting \cref{bound:I_1}, \cref{bound:I_2}, and \eqref{bound:I_3} into \cref{bound:r_proof} (and recalling $\delta^{(0)}=0$) yields \cref{bound:r}.
  \qed
\end{proof}

\section{Proof of \Cref{lemma:sequence_of_contractions}}
\label{Appendix:Banach}


\begin{proof}
  Let $\tilde u_0 = u^{\ast}$ and $\tilde u_n = T_n(\tilde u_{n-1})$, for $n \in \mathbb N$.
  Then,
  \begin{align}
    d(u^{\ast},x_n)
    \leq
    \underbrace{d(u^{\ast},u_n)}_{\to 0}
    +
    \underbrace{d(u_n,\tilde u_n)}_{\qefed a_n}
    +
    \underbrace{d(\tilde u_n,x_n)}_{\to 0},
  \end{align}
  where the last summand goes to zero by
  \begin{align*}
    d(\tilde u_n,x_n)
    &=
    d\left((T_n \circ \dots \circ T_1) (u^{\ast}), (T_n \circ \dots \circ T_1)(x_0)\right)
    \\
    &\leq 
    \bar L^n d(u^{\ast},x_0)\ \to\ 0, \qquad \text{ as } n \to \infty.
  \end{align*}
  Hence, it remains to show that $\lim_{n\to \infty} a_n = 0$.
  The $\bar L$-Lipschitz continuity of $T_n$ and the triangle inequality yield that
  \begin{align}
    a_n
    &=
    d(T_n(u_{n}), T_n(\tilde u_{n-1}))
    \notag
    \\
    &\leq
    \bar L\left [ d(u_n, u_{n-1}) + d( u_{n-1}, \tilde u_{n-1}) \right ]
    \notag
    \\
    &=
    \bar L a_{n-1} + b_{n-1}
    ,
    %
  \end{align}
  where $b_{n} \defeq \bar L d(u_{n+1},u_{n}) \to 0$.
  Now, for all $m \in \mathbb N$, let $a^{(m)}_0 \defeq a_0 $ and $a^{(m)}_n \defeq \bar La^{(m)}_{n-1} + b_m$.
  By BFT, $\lim_{n \to \infty} a^{(m)}_n \allowbreak = \allowbreak b_m / (1-\bar L)$.
  Since, for all $m \in \mathbb N$, $a_n \leq a^{(m)}_n$ for sufficiently large $n$, it follows that
  \begin{align}
    0
    \leq
    \limsup_{n \to \infty} a_n
    \leq
    \lim_{n \to \infty} a^{(m)}_n
    =
    \frac{b_m}{1-\bar L},
    \qquad
    \forall m \in \mathbb N.
  \end{align}
  Since the convergent sequence $u_n$ is in particular a Cauchy sequence, $\lim_{m \to \infty} b_m = 0$ and, hence, $0 \leq \lim_{n \to \infty} a_n = \limsup_{n \to \infty} a_n \leq 0 $.
  Hence, $\lim_{n \to \infty} a_n = 0$.
  \qed
\end{proof}

\section{Proof of Proposition \ref{proposition:steady_states_global}} 
\label{Appendix:Proposition_Proof}

\hans{
The following proof is quite technical.
I'm confident that it's correct: the steady-state values from \cref{lemma:steady_states_values_i,lemma:steady_states_values_ii,lemma:steady_states_values_iii,lemma:steady_states_values_iv,lemma:steady_states_values_v,lemma:steady_states_values_vi} are excessively unit-tested (I can send you the code, if you like) and the bounds mostly inherit the rates from these steady-state values and their initializations.
I hope that I found a balance between a concise description and readability.
Here is a chronological list of the steps in the following proof.
(I also added some explanation boxes at dense parts to make it easier for you.)
I hope it makes it easier to read the actual proof then.
The claims are shown in the following order: \cref{lemma:steady_states_values_i}, \cref{ineq:steady_state_lemma_i}, \cref{lemma:steady_states_values_ii}, \cref{ineq:steady_state_lemma_ii}, \cref{lemma:steady_states_values_iii}, \cref{lemma:steady_states_values_v}, \cref{lemma:steady_states_values_vi}, \cref{lemma:steady_states_values_iv}, \cref{ineq:ineq:steady_state_lemma_vi}, \cref{ineq:ineq:steady_state_lemma_v}, \cref{ineq:steady_state_lemma_iv}.
\begin{enumerate}[Step 1.]
  \item We begin by expressing $P_{11}(nh)$ in terms of $P_{11}^-(nh)$ in \cref{eq:P_11_expressed_by_P_11^-}.
  \item Analogously, we express $P_{11}^-((n+1)h)$ in terms of $P_{11}(nh)$ in \cref{eq:P_11^-_expressed_by_P_11}
  \item We validate \cref{lemma:steady_states_values_i} by pushing $P_{11}^{-,\infty}$ through the recursion for $P_{11}^-$ consisting of \cref{eq:P_11_expressed_by_P_11^-} and \cref{eq:P_11^-_expressed_by_P_11} from Step 1 and 2 and checking that it's still the same.
  \item We show \cref{ineq:steady_state_lemma_i} by bounding $\vert P_{11}^-(nh) \vert$ against the sum of its initial distance of $P_{11}^-$ to the steady state and the steady state itself
  \item Plugging $P_{11}^{-,\infty}$ for $P^-_{11}(nh)$ into the expression \cref{eq:P_11_expressed_by_P_11^-} from Step 1 yields \cref{lemma:steady_states_values_ii}
  \item Since $P_{11}(nh)$ monotonously increases in $P_{11}^-(nh)$, \cref{ineq:steady_state_lemma_i} implies \cref{ineq:steady_state_lemma_ii}, which we execute in \cref{proof:P_11_bound}
  \item Then, in \cref{eq:recursion:P_01^-}, we construct the recursion for $P_{01}^-((n+1)h)$ (whose precise form depends on $P_{11}^-(nh)$). Hence, for every $n$, we get a different contraction $T_n$ that maps $P_{01}^-(nh)$ to $P_{01}^-((n+1)h)$, whose fixed points fortunately converge, for $n \to \infty$. Hence, application of \Cref{lemma:sequence_of_contractions} yields \cref{lemma:steady_states_values_iii}.
  \item Now, \cref{lemma:steady_states_values_v,lemma:steady_states_values_vi} are yielded by plugging the already shown steady states $P_{01}^{\infty}$ and $P_{01}^{-,\infty}$ into the definitions of $\beta^{\infty,(0)}(nh)$ and $\beta^{\infty,(1)}(nh)$, which we execute in \cref{eq_in_proof_beta_0,eq_in_proof_beta_1}
  \item The steady state \cref{lemma:steady_states_values_iv} is then obtained via the equality $P_{01}(nh) = R \beta^{(0)}(nh)$ (see \cref{eq:proof_P01^inf_lemma})
  \item \cref{ineq:ineq:steady_state_lemma_vi} is then by bounding $\vert 1 - \beta^{(1)}(nh) \vert$ using $\inf_n P_{11}^-(nh) \geq \sigma^2 h$ (see \cref{proof:1-beta1_bound})
  \item Since jointly maximizing over both $P_{01}^-$ and $P_{11}^-$ does not yield a sharp bound, we show \cref{ineq:ineq:steady_state_lemma_v} by induction (which is valid since we fix the constant K to be $\left ( \frac{2K_0}{\sigma^2} + \frac 12 \right ) \vee 1$; see also explanation next to \cref{ineq:ineq:steady_state_lemma_v_hat_beta}).
  \item insertion of the bound \cref{ineq:ineq:steady_state_lemma_v} into $\vert P_{01}(nh) \vert = R \vert \beta^{(0)}(nh) \vert$ yields the missing inequality \cref{ineq:steady_state_lemma_iv}
\end{enumerate}
(Reasons for why the fixed points are attractive and why the inequalities are sharp are given in the text.)
}

\begin{proof}
    Again, w.l.o.g.~$d=1$.
    We prove the claims in the following order: \cref{lemma:steady_states_values_i}, \cref{ineq:steady_state_lemma_i}, \cref{lemma:steady_states_values_ii}, \cref{ineq:steady_state_lemma_ii}, \cref{lemma:steady_states_values_iii}, \cref{lemma:steady_states_values_v}, \cref{lemma:steady_states_values_vi}, \cref{lemma:steady_states_values_iv}, \cref{ineq:ineq:steady_state_lemma_vi}, \cref{ineq:ineq:steady_state_lemma_v}, \cref{ineq:steady_state_lemma_iv}.
    The sharpness of these bounds is shown, directly after they are proved.
    As a start, for \cref{lemma:steady_states_values_i}, we show that $P_{11}^{-,\infty}$ is indeed the unique fixed point of the recursion for $\{ P_{11}^{-}(nh) \}_n$ by checking that, if $P_{11}^-(nh) = \frac 12 \left (\sigma^2 h + \sqrt{4\sigma^2 R h + \sigma^4 h^2} \right )$, then also $P_{11}^-((n+1)h) = \frac 12 \left (\sigma^2 h + \sqrt{4\sigma^2 R h + \sigma^4 h^2} \right )$:
    \hans{Here is an expanded \emph{derivation} for the below statement \cref{eq:P_11_expressed_by_P_11^-}:
    \begin{align}
      P_{11}((nh))
      \stackrel{\text{\cref{eq:C_update_MM}}}{=}&
      P_{11}^-(nh) \left ( 1 - \frac{P_{11}^-(nh)}{P_{11}^-(nh) + R}   \right )
      =
      P_{11}^-(nh) \bigg ( \frac{R}{P_{11}^-(nh)  + R}  \bigg )
      \\
      =&
      \left [ \frac{\sigma^2 h + \sqrt{4\sigma^2 R h + \sigma^4 h^2}}{2}  \right ] \cdot \left [ \frac{R}{\frac{\sigma^2 h + \sqrt{4\sigma^2 R h + \sigma^4 h^2} }{2} + R  }  \right ]
      \notag
      \\
      =&
      \frac{\left( \sigma^2 h + \sqrt{4\sigma^2 R h + \sigma^4 h^2}     \right )R}{\sigma^2 h + \sqrt{4\sigma^2 R h + \sigma^4 h^2} + 2R}
      .
      \notag
    \end{align}}
    \begin{align}
      P_{11}((nh))
      &\stackrel{\text{\cref{eq:C_update_MM}}}{=}
      P_{11}^-(nh) \left ( 1 - \frac{P_{11}^-(nh)}{P_{11}^-(nh) + R}   \right )
      \notag
      \\
      \label{eq:P_11_expressed_by_P_11^-}
      &=
      \frac
      {\left( \sigma^2 h + \sqrt{4\sigma^2 R h + \sigma^4 h^2}     \right )R}
      {\sigma^2 h + \sqrt{4\sigma^2 R h + \sigma^4 h^2} + 2R}
      , \quad \text{and}
      \\
      P_{11}^-((n+1)h)
      &=
      P_{11}(nh) + \sigma^2 h
      \notag
      \\
      &\stackrel{\text{\cref{eq:P_11_expressed_by_P_11^-}}}{=}
      \frac 12 \left (\sigma^2 h + \sqrt{4\sigma^2 R h + \sigma^4 h^2} \right )
      \notag
      \\
      &=
      P_{11}^-(nh).
      \label{eq:P_11^-_expressed_by_P_11}
    \end{align}
       \hans{
  Here is a \emph{derivation} for the above statement \cref{eq:P_11^-_expressed_by_P_11}:
        \begin{align}
      P_{11}^-((n+1)h)
      =&
      P_{11}(nh) + \sigma^2 h
      \stackrel{\text{\cref{eq:P_11_expressed_by_P_11^-}}}{=}
      \\
      =&
      \frac{ \left ( \sigma^2 h + \sqrt{4 \sigma^2 R h + \sigma^4 h^2}  \right ) R } {\sigma^2 h + \sqrt{4 \sigma^2 R h + \sigma^4 h^2} + 2R}
      \notag
      \\
      &+
      \frac {\sigma^2 h \left( \sigma^2 h + \sqrt{4 \sigma^2 R h + \sigma^4 h^2} + 2R  \right )} {\sigma^2 h + \sqrt{4 \sigma^2 R h + \sigma^4 h^2} + 2R}
      \notag
      \\
      =&
      \frac { \left(\sigma^2 h + R\right) \left( \sigma^2 h + \sqrt{4 \sigma^2 R h + \sigma^4 h^2}  \right ) + 2R\sigma^2 h + 2R^2 - 2R^2 } {\sigma^2 h + \sqrt{4 \sigma^2 R h + \sigma^4 h^2} + 2R}
      \notag
      \\
      =&
      \frac { \left(\sigma^2 h + R\right) \left( \sigma^2 h + \sqrt{4 \sigma^2 R h + \sigma^4 h^2} + 2R  \right ) - 2R^2 } {\sigma^2 h + \sqrt{4 \sigma^2 R h + \sigma^4 h^2} + 2R}
      \notag
      \\
      =&
      \sigma^2 h + R - \frac {2R^2 \left( [ \sigma^2 h + 2R ] - \sqrt{4 \sigma^2 R h + \sigma^4 h^2}    \right) } { [ \sigma^2 h + 2R ]^2 - (4 \sigma^2 R h + \sigma^4 h^2) }
      \notag
      \\
      =&
      \sigma^2 h + R - \frac {2R^2 \sigma^2 h + 4R^3 - 2R^2\sqrt{4\sigma^2Rh + \sigma^4 h^2}} {4R^2}
      \notag
      \\
      =&
      \sigma^2 h + R - \frac{1}{2} \sigma^2 h - R + \frac 12 \sqrt{4 \sigma^2 R h + \sigma^4 h^2}
      \notag
      \\
      =&
      \frac 12 \left (\sigma^2 h + \sqrt{4\sigma^2 R h + \sigma^4 h^2} \right )
      =
      P_{11}^-(nh).
    %
    \end{align}
  }
    After combining \cref{eq:P_11_expressed_by_P_11^-} and \cref{eq:P_11^-_expressed_by_P_11}, the recursion for $P_{11}^-$ is given by
    \begin{align}
      P_{11}^-((n+1)h)
      &=
      \underbrace{\left ( \frac{R}{P_{11}^-(nh) + R}  \right )}_{\qefed \alpha(nh)} P_{11}^-(nh) + \sigma^2 h 
      \\
      &\qefed \tilde T \left( P_{11}^-(nh)  \right ).
    \end{align}
    Since $R$ and $P_{11}^-(nh)$ are positive variances, we know that $\inf_{n \in [T/h + 1]} P_{11}^-(nh) \geq \sigma^2 h$, and hence $\max_{n \in [T/h + 1]} \alpha(nh) \allowbreak \leq R / (\sigma^2 h + R) < 1$.
    Hence, $\tilde T $ is a contraction.
    By BFT, $P_{11}^{-,\infty}$ is the unique (attractive) fixed point of $\tilde T $, and the sequence $\{ \vert P_{11}^{-}(nh) - P_{11}^{-,\infty} \vert \}_n $ is strictly decreasing.
    Since, by \cref{eq:C_update_MM}, \cref{def:A^IOUP} \rev{with $\theta = 0$} and \Cref{ass:assumption2},
    \begin{align}
      P_{11}^-(h)
      =
      P_{11}(0) + \sigma^2 h
      \leq
      Kh,
    \end{align}
    we can, using the reverse triangle inequality and the (by BFT) strictly decreasing sequence $\{ \vert P_{11}^{-}(nh) - P_{11}^{-,\infty} \vert \}_n $, derive \cref{ineq:steady_state_lemma_i}:
    \begin{align}
      \left \vert P_{11}^-(nh) \right \vert
      &\leq
      \underbrace{\left \vert P_{11}^-(nh) - P_{11}^{-,\infty} \right \vert}_{\leq \left \vert P_{11}^-(h) - P_{11}^{-,\infty} \right \vert} 
      + \left \vert P_{11}^{-,\infty} \right \vert
      \\
      &\leq
      \underbrace{P_{11}^-(h)}_{\leq Kh} + \underbrace{2 P_{11}^{-,\infty}}_{\leq Kh^{1 \wedge \frac{p+1}2 }, \text{ by } \cref{lemma:steady_states_values_i}}
      \\
      &\leq
      Kh^{1 \wedge \frac{p+1}2 },
    \end{align}
    which is sharp because it is estimated against the maximum of the initial $P_{11}^-$ and the steady state that can both be attained.
    Recall that, by \cref{eq:P_11_expressed_by_P_11^-}, $P_{11}(nh)$ depends continuously on $P^-_{11}(nh)$, and, hence, inserting \cref{lemma:steady_states_values_i} into \cref{eq:P_11_expressed_by_P_11^-} yields \cref{lemma:steady_states_values_ii}---the necessary computation was already performed in \cref{eq:P_11_expressed_by_P_11^-}.
    Since $P_{11}(nh)$ monotonically increases in $P_{11}^-(nh)$ (because the derivative of $P_{11}(nh)$ with respect to $P_{11}^-(nh)$ is non-negative for all $P_{11}^-(nh)$ due to $R\geq 0$; see \cref{eq:P_11_expressed_by_P_11^-}), we obtain \cref{ineq:steady_state_lemma_ii}:
    \begin{align}      %
      P_{11}(nh)
      &\stackrel{\text{\cref{eq:P_11_expressed_by_P_11^-}}}{\leq}
      \frac
      {\left( \max_n P_{11}^-(nh) \right) R}
      {\max_n P_{11}^-(nh) + R}
      \\
      &\stackrel{R \sim h^p}{\leq}
      \frac
      {Kh^{1 \wedge \frac{p+1}{2}} Kh^p}
      {Kh^{1 \wedge \frac{p+1}2 } + Kh^p }
      \\
      &\leq
      \frac
      {Kh^{(p+1)\wedge \frac{3p+1}{2}}}
      {Kh^{1 \wedge p}}
      \\
      &\leq
      \begin{cases}
        Kh^{\frac{p+1}2}, & \mbox{if } p\leq 1,
        \\
        Kh^{p} , & \mbox{if } p \geq 1,
      \end{cases}
      \\
      &\leq
      Kh^{p \vee \frac{p+1}2},
      \label{proof:P_11_bound}
    \end{align}
    which is sharp because the steady state \cref{ineq:steady_state_lemma_i} has these rates.
    For \cref{lemma:steady_states_values_iii}, we again first construct the following recursion (from \cref{eq:C^-_predict}, \cref{eq:C_update_MM} and \cref{def:A^IOUP} \rev{with $\theta = 0$})
    \begin{align}
      P^-_{01}((n+1)h)
      &=
      \underbrace{\frac{R}{P^-_{11}(nh) + R}}_{= \alpha(nh)} P^-_{01}\left( nh \right )
      \notag
      \\
      &\qquad +
      \underbrace{\left ( P_{11}(nh) + \frac{\sigma^2 h}{2} \right )h}_{\qefed g(nh)}
      \\
      &=
      T_n\left(P^{-}_{01}(nh) \right )
      ,
      \label{eq:recursion:P_01^-}
    \end{align}
    where the $\alpha(nh)$-Lipschitz continuous contractions $T_n$ satisfy the prerequisites of \Cref{lemma:sequence_of_contractions}, since $\sup_n \alpha(nh) \leq R/(\sigma^2 h + R) < 1$ (due to $\inf_n P_{11}^-(nh) \geq \sigma^2 h$) and the sequence of fixed points $(1-\alpha(nh))^{-1}g(nh)$ of $T_n$ (defined by BFT) converges.
    Both $\alpha(nh)$ and $g(nh)$ depend continuously on $P_{11}^-(nh)$.
    Hence, insertion of the limits \cref{lemma:steady_states_values_i,lemma:steady_states_values_ii} yield
    \begin{align}
      \label{limit:1minusalpha}
      \lim_{n \to \infty} \left ( 1 - \alpha(nh) \right )^{-1}
      =
      \frac
      {\sigma^2 h + \sqrt{4\sigma^2 R h + \sigma^4 h^2} + 2R}
      {\sigma^2 h + \sqrt{4\sigma^2 R h + \sigma^4 h^2}},
    \end{align}
    and
    \begin{align}
      \lim_{n \to \infty} &g(nh)
      \label{limit:g}
      \\
      &=
      \frac
      {(\sigma^4 h^2 + (2R + \sigma^2 h)  \sqrt{4\sigma^2 Rh + \sigma^4 h^2} + 4R\sigma^2 h )}
      {2 (\sigma^2 h + \sqrt{4 \sigma^2 R h + \sigma^4 h^2}  + 2R )}
      h.
      \notag
    \end{align}
    Now, application of \Cref{lemma:sequence_of_contractions} implies convergence of the recursion \cref{eq:recursion:P_01^-} to the product of these two limits \cref{limit:1minusalpha,limit:g}, i.e.~\cref{lemma:steady_states_values_iii}:
    \begin{align*}
      \lim_{n \to \infty} &P_{01}^{-}(nh)
      =
      \lim_{n \to \infty} \left ( 1 - \alpha(nh) \right )^{-1}
      \times
      \lim_{n \to \infty} g(nh)
      \\
      &=
      \frac
      {\sigma^4 h^2 + (2R + \sigma^2 h) \sqrt{4\sigma^2 Rh + \sigma^4 h^2} + 4R\sigma^2 h }
      {2(\sigma^2 h + \sqrt{4\sigma^2 R h + \sigma^4 h^2})} h.
    \end{align*}
    \hans{
    Here is an expanded version of this argument:
    Since $\alpha(nh) = \frac{R}{P_{11}^-(nh) + R}$, it follows that $(1-\alpha(nh))^{-1} = \frac{P_{11}^-(nh) + R}{P_{11}^-(nh)}$.
    Plugging in $\lim_{n \to \infty} P_{11}^-(nh) $ from \cref{lemma:steady_states_values_i} yields
    \begin{align}
      \lim_{n \to \infty} \left ( 1 - \alpha(nh) \right )^{-1}
      =
      \frac
      {\sigma^2 h + \sqrt{4\sigma^2 R h + \sigma^4 h^2} + 2R}
      {\sigma^2 h + \sqrt{4\sigma^2 R h + \sigma^4 h^2}}.
    \end{align}
    Analogously, we can insert $\lim_{n \to \infty} P_{11}(nh) $ from \cref{lemma:steady_states_values_ii} into $g(nh)$, which yields:
    \begin{align}
      \lim_{n \to \infty} g(nh)
      &=
      \left [ P_{11}^{\infty} + \frac{\sigma^2 h}{2} \right ] h
      \\
      &\stackrel{\text{\cref{lemma:steady_states_values_ii}}}{=}
      \left [ \frac{\left( \sigma^2 h + \sqrt{4\sigma^2 R h + \sigma^4 h^2}     \right )R}{\sigma^2 h + \sqrt{4\sigma^2 R h + \sigma^4 h^2} + 2R} + \frac{\sigma^2 h}{2}   \right ] h
      \\
      &=
      \left [ \frac{ 2R \left( \sigma^2 h + \sqrt{4\sigma^2 R h + \sigma^4 h^2} \right ) +\sigma^2 h \left( \sigma^2 h + \sqrt{4\sigma^2 R h + \sigma^4 h^2} + 2R \right)    }{2 \left( \sigma^2 h + \sqrt{4\sigma^2 R h + \sigma^4 h^2} + 2R \right) }     \right ]
      h
      \\
      &=
      \frac
      {\left( \sigma^2 h + 2R  \right ) \left ( \sigma^2 h + \sqrt{4\sigma^2 R h + \sigma^4 h^2}   \right ) + 2R \sigma^2 h}
      {2\left ( \sigma^2 h + \sqrt{4\sigma^2 R h + \sigma^4 h^2}  \right )}
      h
      \\
      &=
      \frac
      {(\sigma^4 h^2 + (2R + \sigma^2 h)  \sqrt{4\sigma^2 Rh + \sigma^4 h^2} + 4R\sigma^2 h )}
      {2 (\sigma^2 h + \sqrt{4 \sigma^2 R h + \sigma^4 h^2}  + 2R )}
      h
    \end{align}
    Now, by \Cref{lemma:sequence_of_contractions},
    \begin{align}
      \lim_{n \to \infty} P_{01}^{-}(nh)
      &=
      \lim_{n \to \infty} \left ( 1 - \alpha(nh) \right )^{-1}
      \cdot
      \lim_{n \to \infty} g(nh)
      \\
      &=
      \frac
      {\sigma^2 h + \sqrt{4\sigma^2 R h + \sigma^4 h^2} + 2R}
      {\sigma^2 h + \sqrt{4\sigma^2 R h + \sigma^4 h^2}}
      \\
      &\qquad \cdot
      \frac
      {(\sigma^4 h^2 + (2R + \sigma^2 h)  \sqrt{4\sigma^2 Rh + \sigma^4 h^2} + 4R\sigma^2 h )}
      {2 (\sigma^2 h + \sqrt{4 \sigma^2 R h + \sigma^4 h^2}  + 2R )}
      h
      \\
      &=
      \frac
      {\sigma^4 h^2 + (2R + \sigma^2 h) \sqrt{4\sigma^2 Rh + \sigma^4 h^2} + 4R\sigma^2 h }
      {2(\sigma^2 h + \sqrt{4\sigma^2 R h + \sigma^4 h^2})} h.
    \end{align}
    }
    For \cref{lemma:steady_states_values_v,lemma:steady_states_values_vi}, we can simply insert \cref{lemma:steady_states_values_i,lemma:steady_states_values_iii} for $P^-_{01}(nh)$ and $P^-_{11}(nh)$ respectively into their definition \cref{def_beta_MM}:
    \begin{align}
      \beta^{\infty,(0)}
      &\stackrel{\text{\cref{def_beta_MM}}}{=}
      \frac{P_{01}^{-,\infty}}{P_{11}^{-,\infty} + R}
      \\
      &\stackrel{\text{\cref{lemma:steady_states_values_i,lemma:steady_states_values_iii}}}{=}
      \frac{\sqrt{4R \sigma^2 h + \sigma^4 h^2}}{ \sigma^2 h + \sqrt{4R \sigma^2 h + \sigma^4 h^2} }
      h,
      \label{eq_in_proof_beta_0}
    \end{align}
    and
    \hans{
    Here is a derivation of \cref{eq_in_proof_beta_0}:
    \begin{align*}
      \beta^{\infty,(0)}
      &\stackrel{\text{\cref{def_beta_MM}}}{=}
      \frac{P_{01}^{-,\infty}}{P_{11}^{-,\infty} + R}
      \\
      &\stackrel{\text{\cref{lemma:steady_states_values_i},\cref{lemma:steady_states_values_iii}}}{=}
      \frac
      { \left ( \sigma^4 h^2 + \left( 2R + \sigma^2 h \right) \sqrt{4R\sigma^2  h + \sigma^4 h^2} + 4R \sigma^2 h  \right )\cdot 2}
      {2\left( \sigma^2 h + \sqrt{4R\sigma^2 h + \sigma^4 h^2}  \right) \left  ( \sigma^2 h + \sqrt{4\sigma^2 R h + \sigma^4 h^2} + 2R \right )}
      h
      \notag
      \\
      &=
      \frac
      {\sqrt{4R \sigma^2 h + \sigma^4 h^2}\left( \sigma^2 h + \sqrt{4R \sigma^2 h + \sigma^4 h^2} + 2R  \right) }
      {\left(\sigma^2 h + \sqrt{4R \sigma^2 h + \sigma^4 h^2}\right ) \left ( \sigma^2 h + \sqrt{4R \sigma^2 h + \sigma^4 h^2} + 2R   \right ) }
      h
      \notag
      \\
      &=
      \frac{\sqrt{4R \sigma^2 h + \sigma^4 h^2}}{\sigma^2 h + \sqrt{4R \sigma^2 h + \sigma^4 h^2} }
      h, \qquad \text{and} \qquad
    \end{align*}
    }
    \begin{align}
      \label{eq_in_proof_beta_1}
      \beta^{\infty,(1)}
      \stackrel{\text{\cref{def_beta_MM,lemma:steady_states_values_i}}}{=}
      \frac {\sigma^2 h + \sqrt{4\sigma^2 R h + \sigma^4 h^2}} { \sigma^2 h + \sqrt{4\sigma^2 R h + \sigma^4 h^2} + 2R }.
    \end{align}
    \hans{
    Here is a longer version of \cref{eq_in_proof_beta_1}
    \begin{align}
      \beta^{\infty,(1)}
      \stackrel{\text{\cref{def_beta_MM},\cref{lemma:steady_states_values_i}}}{=}
      \frac
      {\frac 12 \left ( \sigma^2 h + \sqrt{4 \sigma^2 R h + \sigma^4 h^2} \right )}
      {\frac 12 \left( \sigma^2 h + \sqrt{4 \sigma^2 R h + \sigma^4 h^2}  \right ) + R}
      =
      \frac {\sigma^2 h + \sqrt{4\sigma^2 R h + \sigma^4 h^2}} { \sigma^2 h + \sqrt{4\sigma^2 R h + \sigma^4 h^2} + 2R }.
    \end{align}
    }
  These steady states \cref{lemma:steady_states_values_v,lemma:steady_states_values_vi} are again unique and attractive because $\beta^{(0)}(nh)$ and $\beta^{(1)}(nh)$ depend continuously on $P_{11}^-(nh)$ and $P_{01}^-(nh)$.
  Next, recall that 
    \begin{align}
    \label{eq:proof_P01^inf_lemma}
    P_{01}(nh)
    &\stackrel{\text{\cref{eq:C_update_MM}}}{=}
    \left ( 1-\frac{P_{11}^{-}(nh)}{P_{11}^{-}(nh) + R} \right ) P_{01}^{-}(nh)
    \\
    &=
    R \frac{P_{01}^{-}(nh)}{P_{11}^{-}(nh) + R}
    \stackrel{\text{\cref{def_beta_MM}}}{=}
    R \beta^{(0)}(nh),
  \end{align}
  which, since $P_{01}(nh)$ depends continuously on $\beta^{(0)}(nh)$, implies the unique (attractive) fixed point $P^{\infty}_{01}(nh) = R \beta^{\infty,(0)}$, which yields \cref{lemma:steady_states_values_iv}.
  Now, exploiting \cref{def_beta_MM} and $\inf_n P_{11}^-(nh) \geq \sigma^2 h$ yields \cref{ineq:ineq:steady_state_lemma_vi}:
  \begin{align}
    \label{proof:1-beta1_bound}
    \left \vert 1 - \beta^{(1)}(nh) \right \vert
    &=
    \frac{R}{P_{11}^-(nh) + R}
    \\
    &\leq
    \frac{R}{\sigma^2 h + R }
    \\
    &\stackrel{\rev{R \sim h^p}}{=}
    \frac{Kh^p}{Kh + Kh^p}
    \\
    &\leq
    Kh^{(p-1) \vee 0},
  \end{align}
  which is sharp because $\inf_n P_{11}^-(nh) \geq K h$ is sharp (due to \cref{eq:C^-_predict,def:A^IOUP}).
  And since, for $\beta^{(0)}$, maximizing over both $P_{01}^-(nh)$ and $P_{11}^-(nh)$ at the same time does not yield a sharp bound (while above in \cref{proof:1-beta1_bound,proof:P_11_bound} the maximization over just one quantity does), we prove \cref{ineq:ineq:steady_state_lemma_v} by inductively showing that
  \begin{align}
    \label{ineq:ineq:steady_state_lemma_v_hat_beta}
    \left \vert \beta^{(0)}(nh)  \right \vert \leq \hat \beta h, \qquad \forall n \in \mathbb N, \\ \text{with} \qquad \hat \beta \defeq \left ( \frac{2K_0}{\sigma^2} + \frac 12 \right ) \vee 1>0,
  \end{align}
  where $K_0 > 0$ is the constant from \Cref{ass:assumption2}.
  The constant $\hat \beta$ is independent of $n$ and a possible choice for $K$ in \cref{ineq:ineq:steady_state_lemma_v}.
  The base case ($n=1$) follows from
  \begin{align}
    \label{ineq:base_clause_induction_beta0}
    \left \vert \beta^{(0)}(h)  \right \vert
    &=
    \frac{\left \vert P_{01}^-(h)  \right \vert}{P_{11}^-(h) + R}
    \\
    &\stackrel{\text{\cref{eq:C^-_predict}}}{\leq}
    \frac
    { \left \vert P_{01}(0) \right \vert + h P_{11}(0) + \frac{\sigma^2}{2} h^2 }
    {\sigma^2 h}
    \\
    &\stackrel{\text{Ass.~}\ref{ass:assumption2}}{\leq}
    \left ( \frac{2K_0}{\sigma^2} + \frac 12 \right ) h
    \\
    &\leq
    \hat \beta h.
  \end{align}
  \hans{
  Here is an expanded form of the above inequality \cref{ineq:base_clause_induction_beta0}:
    \begin{align}
    \left \vert \beta^{(0)}(h)  \right \vert
    &=
    \frac{\left \vert P_{01}^-(h)  \right \vert}{P_{11}^-(h) + R}
    \stackrel{\text{\cref{eq:C^-_predict}}}{=}
    \frac
    {\left \vert P_{01}(0) + h P_{11}(0) + \frac {\sigma^2}{2}h^2   \right \vert }
    {P_{11}(0) + \sigma^2 h }
    \\
    &\leq
    \frac
    { \left \vert P_{01}(0) \right \vert + h P_{11}(0) + \frac{\sigma^2}{2} h^2 }
    {\sigma^2 h}
    \stackrel{\text{Ass. }\ref{ass:assumption2}}{\leq}
    \left ( \frac{2K_0}{\sigma^2} + \frac 12 \right ) h
    \leq
    \hat \beta h.
    \notag
  \end{align}
  }
  In the following inductive step ($n-1 \to n$) we, to avoid notational clutter, simply denote $P^-((n-1)h)_{ij}$ by $P^-_{ij}$ which leaves us---by \cref{def_beta_MM}, \cref{eq:C^-_predict} and \cref{eq:C_update_MM}---with the following term to bound:
  \begin{align}
    \left \vert \beta^{(0)}(nh) \right \vert
    &=
    \frac {\left \vert P^-_{01}(nh) \right \vert} {P^-_{11}(nh) + R}
    \\
    &\leq
    \frac {\left \vert P^-_{01} \right \vert \alpha(nh) + hP^-_{11} \alpha(nh) + \frac{\sigma^2}{2} h^2} {P^-_{11}\alpha(nh) + \sigma^2 h + R},
  \end{align}
  with $\alpha(nh) = \left( 1 - \frac{P^-_{11}}{P^-_{11}+R} \right) =  \frac{R}{P^-_{11}+R} $.
  Application of the inductive hypothesis (i.e.~$P_{01}^- \leq \hat \beta ( P^-_{11} + R )$) yields, after some rearrangements, that
  \begin{align}
    \left \vert \beta^{(0)}(nh) \right \vert
    &\leq
    \frac {\hat \beta \left( P^-_{11} + R \right) h \alpha(nh) + hP^-_{11} \alpha(nh) + \frac{\sigma^2}{2} h^2} {P^-_{11}\alpha(nh) + \sigma^2 h + R}
    \notag
    \\
    &=
    \frac
    {2 \hat \beta P^-_{11} R + \sigma^2 h \left( P^-_{11} + R \right) + 2 P^-_{11} R + 2 \hat \beta R^2}
    {2\left(P^-_{11} R + \sigma^2 h \left( P^-_{11} + R \right) + P^-_{11} R + R^2 \right ) }
    h
    \notag
    \\
    &=
    \frac{2(\hat \beta + 1)\Lambda_1 + \Lambda_2 + 2 \hat \beta \Lambda_3}{4\Lambda_1 + 2\Lambda_2 + 2\Lambda_3} h,
    \label{bound:beta0inJ}
  \end{align}
  with $\Lambda_1 \defeq 2 P^-_{11}R$, $\Lambda_2  \defeq  \sigma^2 h \left( P^-_{11} + R \right)$, and $\Lambda_3  \defeq  R^2$.
  Now, application of $\hat \beta \geq 1$  yields $\vert \beta^{(0)}(nh) \vert \leq \hat \beta h$, which completes the inductive proof of \cref{ineq:ineq:steady_state_lemma_v_hat_beta}.
  This implies \cref{ineq:ineq:steady_state_lemma_v}, which is sharp because it is the order of $\beta^{(0)}$ in the steady state \cref{lemma:steady_states_values_v}, for all \rev{$p \in [0,\infty]$}.
  Now, insertion of \cref{ineq:ineq:steady_state_lemma_v} into \cref{eq:proof_P01^inf_lemma} immediately yields \cref{ineq:steady_state_lemma_iv}, which---by \cref{eq:proof_P01^inf_lemma}---inherits the sharpness of \cref{ineq:ineq:steady_state_lemma_v}.
  \qed
\end{proof}

\section{Proof of \Cref{lemma:delta_global_bound}}
\label{appendix:lemma_delta}

\begin{proof}
  For all $n \in [T/h + 1]$, we can estimate
  \begin{align}
      \delta^{(1)}(nh)
      &=
      \left \Vert m^{(1)}\left(nh\right) - f\left( m^{(0)}\left(nh\right)  \right) \right \Vert
      \\
      &=
      \left \Vert \Psi_\rev{h}^{(1)}(\boldsymbol{m}((n-1)h) - f\left( m^{(0)}\left(nh\right)  \right)  \right \Vert
      \\
      &\leq
      \underbrace{\left \Vert \Psi_\rev{h}^{(1)}(\boldsymbol{m}((n-1)h) -  f\left( m^{-,(0)}\left(nh\right) \right ) \right \Vert }_{\qefed J_1(h)}
      \notag
      \\
      &\phantom{\leq} +
      \underbrace{\left \Vert  f\left( m^{-,(0)}\left(nh\right) \right )  - f \left ( m^{(0)} \left ( nh \right ) \right ) \right \Vert}_{\defeq J_2(h)},
      \label{bound:delta_by_J1/2}
  \end{align}
    bound $J_1$, \rev{using} the definition \cref{def:Psi} of \rev{$\Psi_h^{(1)}(\boldsymbol{m}((n-1)h)$} as well as the definition \cref{def:r} of $r(nh)$, \rev{by} 
    \begin{align}
      J_1(h)
      &=
      \Bigg \Vert m^{\rev{-,(1)}}(nh) - f \left ( m^{\rev{-,(0)}}(nh)  \right ) 
      \\
      &\phantom{= \Bigg \Vert}+ \beta^{(1)}(nh) \left [ f \left( m^{\rev{-,(0)}}(nh)  \right ) - m^{\rev{-,(1)}}(nh)    \right ] 
      \Bigg \Vert
      \notag
      \\
      &\leq
      \left \Vert 1 - \beta^{(1)}(nh)  \right \Vert \rev{\left \Vert r(nh) \right \Vert}
      \\
      &\stackrel{\text{\cref{ineq:ineq:steady_state_lemma_vi}}}{\leq}
      Kh^{(p-1)\vee 0} \rev{\left \Vert r(nh) \right \Vert}
      \label{eq:delta_proof_bound_J1}
    \end{align}
    and bound $J_2$, by exploiting $L$-Lipschitz continuity of $f$, inserting the definition \cref{def:Psi} of \rev{$\Psi^{(0)}_h(\boldsymbol{m}((n-1)h)$} and applying \cref{ineq:ineq:steady_state_lemma_v} to $\left \Vert \beta^{(0)}(nh) \right \Vert$, 
    \begin{align}
      J_2(h)
      &\leq
      L \left \Vert m^{(0)}(nh) - m^{\rev{-,(0)}}(nh)  \right \Vert
      \\
      &\leq
      L \left \Vert \beta^{(0)}(nh) \right \Vert \rev{\left \Vert r(nh) \right \Vert}
      \\
      &\stackrel{\text{\cref{ineq:ineq:steady_state_lemma_v}}}{\leq}
      Kh \rev{\left \Vert r(nh) \right \Vert} .
      %
      \label{eq:delta_proof_bound_J2}
    \end{align}
  Altogether, after inserting these bounds into \cref{bound:delta_by_J1/2},
  \begin{align}
    &\delta^{(1)}(nh)
    \leq
    \left(Kh^{(p-1) \vee 0} + Kh \right) \rev{\left \Vert r(nh) \right \Vert}
    \\
    &\quad \leq
    Kh^{((p-1)\vee 0) \wedge 1} \rev{\left \Vert r(nh) \right \Vert}
    \\
    &\quad \stackrel{\text{\cref{bound:r}}}{\leq}
    Kh^{(p \vee 1) \wedge 2}
    \\
    &\quad \phantom{\stackrel{\text{\cref{bound:r}}}{\leq}}
    + \left( Kh^{((p-1)\vee 0) \wedge 1} + Kh^{(p \vee 1) \wedge 2}  \right)
    \delta^{(1)}((n-1)h)
    \notag
    \\
    &\quad \qefed
    \bar{T} \left( \delta^{(1)}((n-1)h) \right).
    \label{bound:delta_by_T_of_delta}
  \end{align}
  \rev{As $p\geq 1$ (by \Cref{ass:R}), BFT is} applicable for all sufficiently small $h>0$ such that \rev{$ Kh^{((p-1)\vee 0) \wedge 1} + Kh^{(p \vee 1) \wedge 2}<1$} and so $\bar{T}$ is a contraction \rev{with} a unique fixed point $\delta^{\infty}$ of order
  \begin{align}
    \label{eq:delta^infty_orders}
    \delta^{\infty}
    &\leq
    \frac
    { Kh^{(p \vee 1) \wedge 2 } }
    { 1 - \left( Kh^{((p-1)\vee 0) \wedge 1} + Kh^{(p \vee 1) \wedge 2}  \right) }
    \\
    &\leq
    Kh^{(p \vee 1) \wedge 2 }.
  \end{align}
  We proceed with showing by induction that, for all $n \in [T/h]$,
  \begin{align}
    \label{claim_to_be_proved_by_induction}
    \delta^{(1)}(nh)
    \leq
    \delta^{(1)}(0) \vee 2\delta^{\infty}.
  \end{align}
  \hans{The use of induction here is valid, because the inductively shown statement already specifies the constant $\delta^{(1)}(0) \vee 2\delta^{\infty}$ for which it holds for all $n$. The constant hence does not change with $n$. Hence, it is valid for the same reason why the inductive proof of \cref{ineq:ineq:steady_state_lemma_v_hat_beta} was valid. Also, note that intuitively \cref{claim_to_be_proved_by_induction} is kind of clear. $\delta^{(1)}(nh)$ is bounded by the maximum of its initialization and two times the limit of its upper bounding contraction \cref{bound:delta_by_T_of_delta}.}
  The base case $n=0$ is trivial.
  For the inductive step, we distinguish two cases.
  If $\delta^{(1)}((n-1)h) \leq \delta^{\infty}$, then $\bar{T}(\delta^{(1)}((n-1)h)) < 2 \delta^{\infty}$, since
  \begin{align}
    \bar{T}(\delta^{(1)}((n-1)h)) - \delta^{\infty}
    &\leq
    \left \vert \delta^{\infty} - \bar{T}(\delta^{(1)}((n-1)h))  \right \vert
    \\
    &<
    \delta^{\infty} - \underbrace{\delta^{(1)}((n-1)h)}_{\geq 0}
    \\
    &\leq
    \delta^{\infty}.
  \end{align}
  In this case,
  \begin{align}
    \delta^{(1)}(nh)
    &\stackrel{\text{\cref{bound:delta_by_T_of_delta}}}{\leq}
    \bar{T} \left ( \delta^{(1)}((n-1)h) \right )
    \\
    &<
    2 \delta^{\infty}
    \\
    &\leq
    \delta^{(1)}(0) \vee 2\delta^{\infty},
  \end{align}
  where the last inequality follows from the inductive hypothesis.
  In the other case, namely $\delta^{(1)}((n-1)h) > \delta^{\infty}$, it follows that
  \begin{align}
    \delta^{(1)}(nh) - \delta^{\infty}
    &\stackrel{\text{\cref{bound:delta_by_T_of_delta}}}{\leq}
    \bar{T}(\delta^{(1)}((n-1)h)) - \delta^{\infty}
    \\
    &\leq
    \left \vert \bar{T}(\delta^{(1)}((n-1)h)) - \delta^{\infty} \right \vert
    \\
    &\leq
    \left \vert \delta^{(1)}((n-1)h) - \delta^{\infty}   \right \vert
    \\
    &=
    \delta^{(1)}((n-1)h) - \delta^{\infty},
  \end{align}
  which, after adding $\delta^{\infty}$ and applying the inductive hypothesis, completes the inductive step.
  Hence, \cref{claim_to_be_proved_by_induction} holds.
  Since this bound is uniform in $n$, inserting the orders of $\delta^{(1)}(0)$ from \Cref{lemma:initial_delta_bound} and of $\delta^{\infty}$ from \cref{eq:delta^infty_orders} \rev{yields} \cref{ineq:delta}.
  \qed
\end{proof}

\section{Proof of Theorem \ref{theorem:contraction_of_credible_intervals}}
\label{appendix:uncertain_calibration_theorem}

\begin{proof}
  Again, w.l.o.g.~$d=1$.
  We first show that the bounds \cref{ineq:maxP00-,ineq:maxP00} hold and then argue that they are sharp.
  The recursion for $P_{00}^-(nh)$ is given by
  \begin{align}
    P_{00}^- ((n+1)h)
    &\stackrel{\text{eqs. } \eqref{eq:C^-_predict},\eqref{def:A^IOUP}}{=}
    P_{00}(nh) + 2hP_{01}(nh)
    \notag
    \\
    &\phantom{\stackrel{\text{eqs. } \eqref{eq:C^-_predict},\eqref{def:A^IOUP}}{=}} 
    + h^2 P_{11}(nh) + \frac{\sigma^2}3 h^3
    \\
    &=
    P_{00}^-(nh)
    -
    \beta^{(0)}(nh) P_{01}^- (nh)
    +
    \frac{\sigma^2}3 h^3,
    \notag
    \\
    &\phantom{=}
    +
    2hR \beta^{(0)}(nh)
    +
    h^2 R \beta^{(1)}(nh)
    \label{eq:recursion_P00-}
  \end{align}
  where we used $P_{00}(nh) \allowbreak = \allowbreak P_{00}^-(nh) - \beta^{(0)} P_{01}^-(nh)$ and $P_{11}(nh) \allowbreak = \allowbreak R \beta^{(1)}(nh)$ (both due to \cref{eq:C_update_MM} and \cref{def_beta_MM}), as well as $P_{01}(nh)=R \beta^{(0)}(nh)$ (see \cref{eq:proof_P01^inf_lemma}), for the last equality in \cref{eq:recursion_P00-}.
  By $P_{01}^-(nh) \leq P_{01}(nh)$ and $\vert \beta^{(1)} \vert \leq 1$ (due to \cref{def_beta_MM}), application of the triangle inequality to \cref{eq:recursion_P00-} yields
  \begin{align}
    P_{00}^-\left ((n+1)h \right)
    &\leq
    P_{00}^- ( n h )
    +
    \left \vert \beta^{(0)}(nh) \right \vert \left \vert P_{01} (nh) \right \vert
    \notag
    \\
    &\phantom{\leq}
    +
    2hR \left \vert \beta^{(0)}(nh)  \right \vert
    +
    h^2 R
    +
    \frac{\sigma^2}{3} h^3,
  \end{align}
  which, by \cref{ineq:steady_state_lemma_iv,ineq:ineq:steady_state_lemma_v}, implies
  \begin{align}
    P_{00}^-( (n+1) h )
    \leq
    P_{00}^-(nh)
    +
    Kh^{(p+2) \wedge 3}.
  \end{align}
  This, by $N=T/h$, implies \cref{ineq:maxP00-}.
  Since $P_{00}(nh) \leq P_{00}^-(nh)$, this bound is also valid for $P_{00}$, i.e.~\cref{ineq:maxP00} holds.
  The bound \cref{ineq:maxP00-} is sharp, since, e.g.~when the covariance matrices are in the steady state, the covariance matrix keeps growing by a rate of $Kh^{(p+2) \wedge 3}$ for all sufficiently small $h>0$, since the only negative summand in \cref{eq:recursion_P00-} is given by
  \begin{align}   \label{eq:product_S_1:3}
    \beta^{\infty,(0)} P_{01}^{\infty}
    =
    S_1(h)
    \times
    S_2(h)
    \times
    S_3(h)
    \ \in 
    \Theta(h^{5 \wedge \frac{3p + 7}{2} })
    ,
  \end{align}
  where the factors have, due to $R \equiv Kh^p$, the following orders:
  \begin{align}
    S_1(h)
    &=
    \frac 12 h^2 \ \in \Theta (h^2),
    \label{eq:S_1_order}
    \\
    S_2(h)
    &=
    \sqrt{(\sigma^2 h)^2 + 4 (\sigma^2 h) R}, \ \in \Theta ( h^{1 \wedge \frac{p+1}{2} }),
    \label{eq:S_2_order}
    \\
    S_3(h)
    &=
    ((\sigma^2 h) + 2 R) \sqrt{(\sigma^2 h)^2 + 4 (\sigma^2 h) R}
    \notag
    \\
    &\phantom{=} + (\sigma^2 h)^2 + 4 (\sigma^2 h) R
    \ \in 
    \Theta(h^{2 \wedge (p+1)})
    .    
    \label{eq:S_3_order}
  \end{align}
  The orders in \cref{eq:S_1_order,eq:S_2_order,eq:S_3_order} imply the order in \cref{eq:product_S_1:3}. 
  Hence, the sole negative summand $- \beta^{\infty,(0)} P_{01}^{\infty}$ of \cref{eq:recursion_P00-} is in $\Theta(h^{5 \wedge \frac{3p + 7}{2} })$ and thereby of higher order than the remaining positive summands of \cref{eq:recursion_P00-}:
  \begin{align}
    \underbrace{2hR}_{\in \Theta(h^{p+1})} \underbrace{\beta^{\infty,(0)}(nh)}_{\in \Theta(h)}
    &\ \in \Theta(h^{p+2}),
    \\
    \underbrace{h^2 R}_{\in \Theta(h^{p+2})} \underbrace{\beta^{\infty,(1)}(nh)}_{\in \Theta(1), \text{ by \cref{lemma:steady_states_values_vi}}}
    &\ \in \Theta(h^{p+2}),
    \\
    \frac{\sigma^2}3 h^3
    &\ \in \Theta\left(h^{3}\right).
  \end{align}
  Hence, for all sufficiently small $h>0$, it still holds in the steady state that $P_{00}^-((n+1)h) - P_{00}^-(nh) \geq Kh^{(p+2) \wedge 3} $, and therefore \cref{ineq:maxP00-} is sharp.
  The sharpness of \cref{ineq:maxP00-} is inherited by \cref{ineq:maxP00} since, in the steady state, by \cref{eq:C_update_MM,def_beta_MM}, $P_{00}(nh) = P_{00}^{-}(nh) - \beta^{(0),\infty} P_{01}^{-,\infty}$ and the subtracted quantity $\beta^{(0),\infty} P_{01}^{-,\infty}$ is---as shown above---only of order $\Theta(h^{5 \wedge \frac{3p + 7}{2} })$.
  \qed
\end{proof}

\end{document}

%% file: fig/fhntrue_solution.tex
\begin{tikzpicture}

\definecolor{color0}{rgb}{0.12156862745098,0.466666666666667,0.705882352941177}
\definecolor{color1}{rgb}{1,0.498039215686275,0.0549019607843137}

\begin{axis}[
tick align=outside,
tick pos=left,
x grid style={white!69.0196078431373!black},
xmin=-0.5095, xmax=10.6995,
xtick style={color=black},
y grid style={white!69.0196078431373!black},
ymin=-2.09435801201193, ymax=2.13657880377597,
ytick style={color=black}
]
\addplot [semithick, color0]
table {%
0 1
0.01 1.00662324355892
0.02 1.01315873447071
0.03 1.01960515364381
0.04 1.02596121938957
0.05 1.03222568813636
0.06 1.03839735508685
0.07 1.04447505481746
0.08 1.05045766181957
0.09 1.05634409098167
0.1 1.062133298012
0.11 1.06782427980185
0.12 1.07341607472891
0.13 1.07890776290111
0.14 1.08429846634107
0.15 1.08958734911162
0.16 1.09477361738273
0.17 1.09985651944079
0.18 1.10483534564084
0.19 1.10970942830286
0.2 1.11447814155306
0.21 1.11914090111145
0.22 1.12369716402708
0.23 1.12814642836208
0.24 1.13248823282633
0.25 1.13672215636424
0.26 1.1408478176953
0.27 1.14486487481018
0.28 1.14877302442422
0.29 1.15257200139029
0.3 1.15626157807283
0.31 1.1598415636851
0.32 1.16331180359174
0.33 1.16667217857858
0.34 1.16992260409189
0.35 1.17306302944901
0.36 1.1760934370226
0.37 1.17901384140054
0.38 1.18182428852349
0.39 1.18452485480239
0.4 1.18711564621759
0.41 1.18959679740203
0.42 1.19196847071018
0.43 1.19423085527488
0.44 1.19638416605376
0.45 1.19842864286732
0.46 1.2003645494305
0.47 1.20219217237918
0.48 1.2039118202938
0.49 1.20552382272134
0.5 1.20702852919736
0.51 1.20842630826985
0.52 1.20971754652593
0.53 1.21090264762317
0.54 1.21198203132649
0.55 1.21295613255236
0.56 1.21382540042091
0.57 1.21459029731751
0.58 1.21525129796466
0.59 1.21580888850518
0.6 1.21626356559758
0.61 1.21661583552456
0.62 1.21686621331515
0.63 1.21701522188156
0.64 1.21706339117091
0.65 1.21701125733288
0.66 1.21685936190339
0.67 1.216608251005
0.68 1.21625847456433
0.69 1.21581058554677
0.7 1.2152651392089
0.71 1.21462269236871
0.72 1.21388380269381
0.73 1.2130490280079
0.74 1.21211892561524
0.75 1.21109405164364
0.76 1.20997496040543
0.77 1.20876220377682
0.78 1.20745633059524
0.79 1.20605788607474
0.8 1.20456741123913
0.81 1.20298544237291
0.82 1.20131251048951
0.83 1.19954914081673
0.84 1.19769585229922
0.85 1.19575315711753
0.86 1.19372156022347
0.87 1.1916015588916
0.88 1.18939364228626
0.89 1.18709829104398
0.9 1.18471597687095
0.91 1.18224716215487
0.92 1.17969229959105
0.93 1.1770518318224
0.94 1.17432619109257
0.95 1.17151579891229
0.96 1.16862106573795
0.97 1.16564239066267
0.98 1.16258016111869
0.99 1.15943475259125
1 1.15620652834327
1.01 1.15289583915033
1.02 1.14950302304581
1.03 1.14602840507545
1.04 1.1424722970611
1.05 1.13883499737329
1.06 1.13511679071193
1.07 1.1313179478951
1.08 1.12743872565534
1.09 1.12347936644296
1.1 1.11944009823614
1.11 1.11532113435747
1.12 1.11112267329636
1.13 1.10684489853715
1.14 1.10248797839251
1.15 1.09805206584173
1.16 1.09353729837358
1.17 1.08894379783357
1.18 1.08427167027497
1.19 1.07952100581376
1.2 1.07469187848667
1.21 1.06978434611239
1.22 1.06479845015569
1.23 1.05973421559395
1.24 1.05459165078606
1.25 1.04937074734343
1.26 1.0440714800026
1.27 1.03869380649985
1.28 1.0332376674469
1.29 1.02770298620805
1.3 1.02208966877835
1.31 1.01639760366262
1.32 1.01062666175527
1.33 1.00477669622084
1.34 0.998847542374861
1.35 0.992839017565417
1.36 0.986750921054754
1.37 0.980583033901359
1.38 0.974335118842162
1.39 0.968006920174904
1.4 0.961598163640743
1.41 0.955108556306934
1.42 0.948537786449739
1.43 0.941885523437533
1.44 0.935151417614179
1.45 0.928335100182796
1.46 0.921436183089912
1.47 0.914454258910265
1.48 0.907388900732285
1.49 0.900239662044503
1.5 0.893006076623027
1.51 0.885687658420363
1.52 0.878283901455788
1.53 0.870794279707589
1.54 0.863218247007487
1.55 0.855555236937532
1.56 0.847804662729936
1.57 0.839965917170173
1.58 0.832038372503845
1.59 0.824021380347818
1.6 0.815914271606126
1.61 0.807716356391208
1.62 0.799426923951141
1.63 0.791045242603493
1.64 0.78257055967656
1.65 0.774002101458692
1.66 0.765339073156565
1.67 0.75658065886332
1.68 0.747726021537414
1.69 0.738774302993239
1.7 0.729724623904659
1.71 0.720576083822398
1.72 0.711327761206709
1.73 0.701978713476466
1.74 0.692527977076108
1.75 0.682974567561867
1.76 0.673317479708812
1.77 0.663555687640338
1.78 0.65368814498182
1.79 0.64371378504027
1.8 0.633631521011867
1.81 0.623440246219438
1.82 0.613138834382016
1.83 0.602726139918705
1.84 0.592200998289262
1.85 0.581562226373846
1.86 0.5708086228946
1.87 0.559938968881814
1.88 0.548952028187534
1.89 0.537846548049668
1.9 0.526621259709746
1.91 0.515274879087717
1.92 0.503806107517167
1.93 0.492213632544602
1.94 0.48049612879665
1.95 0.468652258919023
1.96 0.456680674591385
1.97 0.444580017622383
1.98 0.432348921129217
1.99 0.419986010806383
2 0.40748990628827
2.01 0.394859222610514
2.02 0.382092571775168
2.03 0.36918856442482
2.04 0.356145811631041
2.05 0.342962926802574
2.06 0.329638527718847
2.07 0.316171238694493
2.08 0.302559692880677
2.09 0.288802534709019
2.1 0.274898422484095
2.11 0.260846031130425
2.12 0.246644055099906
2.13 0.232291211445654
2.14 0.217786243068108
2.15 0.203127922139246
2.16 0.18831505371055
2.17 0.173346479510275
2.18 0.15822108193528
2.19 0.142937788242466
2.2 0.127495574944501
2.21 0.111893472414155
2.22 0.0961305697010517
2.23 0.080206019564169
2.24 0.0641190437227684
2.25 0.0478689383277223
2.26 0.0314550796544556
2.27 0.0148769300177778
2.28 -0.00186595609207006
2.29 -0.0187739256539688
2.3 -0.0358472205440709
2.31 -0.0530859707519175
2.32 -0.070490187390304
2.33 -0.0880597555063184
2.34 -0.105794426702626
2.35 -0.123693811579979
2.36 -0.141757372013866
2.37 -0.159984413280343
2.38 -0.17837407604833
2.39 -0.196925328258024
2.4 -0.215636956907535
2.41 -0.234507559772518
2.42 -0.253535537086218
2.43 -0.272719083210149
2.44 -0.292056178328566
2.45 -0.311544580202716
2.46 -0.331181816023921
2.47 -0.350965174407455
2.48 -0.37089169757221
2.49 -0.390958173754016
2.5 -0.411161129903373
2.51 -0.431496824721043
2.52 -0.451961242087536
2.53 -0.472550084944925
2.54 -0.49325876969149
2.55 -0.514082421151574
2.56 -0.535015868184485
2.57 -0.556053639997306
2.58 -0.577189963227233
2.59 -0.598418759858986
2.6 -0.619733646042616
2.61 -0.641127931875677
2.62 -0.662594622212336
2.63 -0.68412641855936
2.64 -0.705715722115805
2.65 -0.727354638009301
2.66 -0.749034980777205
2.67 -0.770748281135064
2.68 -0.792485794068836
2.69 -0.814238508279895
2.7 -0.835997157003945
2.71 -0.857752230216356
2.72 -0.879493988226872
2.73 -0.901212476656736
2.74 -0.922897542780516
2.75 -0.944538853203921
2.76 -0.966125912837417
2.77 -0.987648085113642
2.78 -1.00909461338496
2.79 -1.0304546434256
2.8 -1.0517172469511
2.81 -1.0728714460569
2.82 -1.09390623846652
2.83 -1.11481062347042
2.84 -1.13557362842649
2.85 -1.15618433568574
2.86 -1.17663190979859
2.87 -1.1969056248523
2.88 -1.21699489178458
2.89 -1.23688928551585
2.9 -1.25657857174128
2.91 -1.27605273322331
2.92 -1.29530199542773
2.93 -1.31431685134919
2.94 -1.33308808537785
2.95 -1.35160679606463
2.96 -1.36986441765156
2.97 -1.38785274024262
2.98 -1.4055639285017
2.99 -1.42299053877631
3 -1.44012553455856
3.01 -1.45696230020865
3.02 -1.47349465288101
3.03 -1.48971685260738
3.04 -1.50562361050706
3.05 -1.52121009510871
3.06 -1.53647193678396
3.07 -1.5514052303071
3.08 -1.56600653556925
3.09 -1.58027287648897
3.1 -1.59420173817364
3.11 -1.60779106239755
3.12 -1.62103924147306
3.13 -1.6339451106006
3.14 -1.64650793879124
3.15 -1.65872741846258
3.16 -1.67060365381401
3.17 -1.68213714809197
3.18 -1.69332878985829
3.19 -1.70417983837727
3.2 -1.71469190823681
3.21 -1.72486695331914
3.22 -1.73470725023461
3.23 -1.74421538132986
3.24 -1.75339421737808
3.25 -1.76224690005495
3.26 -1.77077682429944
3.27 -1.77898762065267
3.28 -1.78688313766305
3.29 -1.79446742443901
3.3 -1.801744713425
3.31 -1.80871940346921
3.32 -1.81539604324552
3.33 -1.82177931508483
3.34 -1.82787401926513
3.35 -1.83368505880263
3.36 -1.83921742478031
3.37 -1.84447618224419
3.38 -1.84946645669183
3.39 -1.85419342117233
3.4 -1.85866228401144
3.41 -1.8628782771713
3.42 -1.86684664524948
3.43 -1.8705726351179
3.44 -1.87406148619885
3.45 -1.87731842137157
3.46 -1.8803486385003
3.47 -1.88315730257176
3.48 -1.88574953842766
3.49 -1.88813042407643
3.5 -1.89030498456584
3.51 -1.89227818639758
3.52 -1.89405493246309
3.53 -1.89564005747955
3.54 -1.89703832390405
3.55 -1.89825441830346
3.56 -1.89929294815742
3.57 -1.90015843907175
3.58 -1.90085533237948
3.59 -1.90138798310715
3.6 -1.90176065828423
3.61 -1.9019775355738
3.62 -1.90204270220339
3.63 -1.90196015417529
3.64 -1.90173379573628
3.65 -1.90136743908753
3.66 -1.90086480431601
3.67 -1.90022951952973
3.68 -1.8994651211794
3.69 -1.89857505455086
3.7 -1.89756267441211
3.71 -1.89643124580085
3.72 -1.89518394493838
3.73 -1.89382386025671
3.74 -1.89235399352693
3.75 -1.89077726107669
3.76 -1.88909649508657
3.77 -1.8873144449547
3.78 -1.88543377872044
3.79 -1.88345708453842
3.8 -1.88138687219444
3.81 -1.87922557465599
3.82 -1.87697554965031
3.83 -1.87463908126351
3.84 -1.87221838155489
3.85 -1.86971559218115
3.86 -1.86713278602523
3.87 -1.86447196882567
3.88 -1.86173508080216
3.89 -1.8589239982735
3.9 -1.85604053526496
3.91 -1.85308644510166
3.92 -1.85006342198546
3.93 -1.84697310255307
3.94 -1.8438170674131
3.95 -1.84059684266022
3.96 -1.83731390136491
3.97 -1.83396966503727
3.98 -1.83056550506382
3.99 -1.82710274411612
4 -1.82358265753056
4.01 -1.82000647465837
4.02 -1.81637538018563
4.03 -1.8126905154226
4.04 -1.80895297956215
4.05 -1.80516383090717
4.06 -1.80132408806686
4.07 -1.79743473112178
4.08 -1.79349670275786
4.09 -1.78951090936952
4.1 -1.78547822213201
4.11 -1.78139947804344
4.12 -1.77727548093645
4.13 -1.77310700246044
4.14 -1.7688947830342
4.15 -1.76463953276974
4.16 -1.76034193236765
4.17 -1.75600263398454
4.18 -1.75162226207303
4.19 -1.74720141419475
4.2 -1.74274066180705
4.21 -1.73824055102392
4.22 -1.73370160335156
4.23 -1.72912431639933
4.24 -1.72450916456658
4.25 -1.71985659970587
4.26 -1.7151670517633
4.27 -1.71044092939641
4.28 -1.70567862057015
4.29 -1.70088049313176
4.3 -1.69604689536476
4.31 -1.69117815652277
4.32 -1.68627458734387
4.33 -1.6813364805456
4.34 -1.67636411130165
4.35 -1.67135773770026
4.36 -1.66631760118519
4.37 -1.66124392697955
4.38 -1.65613692449316
4.39 -1.65099678771359
4.4 -1.64582369558176
4.41 -1.64061781235219
4.42 -1.63537928793849
4.43 -1.63010825824458
4.44 -1.62480484548188
4.45 -1.61946915847296
4.46 -1.61410129294218
4.47 -1.60870133179336
4.48 -1.60326934537519
4.49 -1.59780539173454
4.5 -1.59230951685804
4.51 -1.58678175490231
4.52 -1.58122212841313
4.53 -1.57563064853389
4.54 -1.57000731520351
4.55 -1.56435211734435
4.56 -1.55866503304004
4.57 -1.55294602970392
4.58 -1.54719506423798
4.59 -1.54141208318275
4.6 -1.53559702285835
4.61 -1.52974980949686
4.62 -1.52387035936638
4.63 -1.51795857888678
4.64 -1.51201436473753
4.65 -1.50603760395779
4.66 -1.50002817403879
4.67 -1.49398594300894
4.68 -1.48791076951154
4.69 -1.48180250287549
4.7 -1.47566098317916
4.71 -1.46948604130738
4.72 -1.46327749900185
4.73 -1.45703516890516
4.74 -1.45075885459833
4.75 -1.44444835063228
4.76 -1.43810344255313
4.77 -1.43172390692153
4.78 -1.42530951132617
4.79 -1.41886001439153
4.8 -1.41237516577997
4.81 -1.40585470618829
4.82 -1.39929836733876
4.83 -1.3927058719649
4.84 -1.38607693379193
4.85 -1.3794112575119
4.86 -1.3727085387539
4.87 -1.36596846404909
4.88 -1.35919071079067
4.89 -1.35237494718925
4.9 -1.34552083222294
4.91 -1.33862801558303
4.92 -1.33169613761474
4.93 -1.32472482925334
4.94 -1.31771371195569
4.95 -1.31066239762725
4.96 -1.3035704885445
4.97 -1.29643757727298
4.98 -1.28926324658091
4.99 -1.2820470693485
5 -1.27478860847287
5.01 -1.26748741676884
5.02 -1.26014303686546
5.03 -1.25275500109839
5.04 -1.24532283139823
5.05 -1.23784603917479
5.06 -1.2303241251974
5.07 -1.22275657947123
5.08 -1.21514288110987
5.09 -1.20748249820395
5.1 -1.19977488768619
5.11 -1.19201949519262
5.12 -1.18421575492038
5.13 -1.17636308948184
5.14 -1.16846090975542
5.15 -1.16050861473296
5.16 -1.15250559136389
5.17 -1.14445121439633
5.18 -1.13634484621496
5.19 -1.12818583667617
5.2 -1.11997352294032
5.21 -1.1117072293013
5.22 -1.10338626701366
5.23 -1.09500993411737
5.24 -1.08657751526031
5.25 -1.07808828151891
5.26 -1.0695414902168
5.27 -1.06093638474198
5.28 -1.05227219436258
5.29 -1.04354813404153
5.3 -1.03476340425028
5.31 -1.02591719078201
5.32 -1.01700866456454
5.33 -1.00803698147331
5.34 -0.999001282144664
5.35 -0.989900691789987
5.36 -0.980734320011019
5.37 -0.971501260616695
5.38 -0.962200591442124
5.39 -0.952831374170059
5.4 -0.943392654155508
5.41 -0.933883460253938
5.42 -0.924302804653784
5.43 -0.914649682713804
5.44 -0.904923072806068
5.45 -0.895121936165234
5.46 -0.885245216744921
5.47 -0.875291841082054
5.48 -0.865260718170004
5.49 -0.85515073934154
5.5 -0.844960778162518
5.51 -0.834689690337464
5.52 -0.824336313628133
5.53 -0.81389946778629
5.54 -0.803377954501956
5.55 -0.792770557368572
5.56 -0.782076041866348
5.57 -0.771293155365586
5.58 -0.760420627151377
5.59 -0.749457168471505
5.6 -0.738401472609345
5.61 -0.727252214983686
5.62 -0.71600805327753
5.63 -0.704667627597959
5.64 -0.693229560669409
5.65 -0.681692458062667
5.66 -0.670054908462183
5.67 -0.658315483974253
5.68 -0.646472740478985
5.69 -0.634525218028833
5.7 -0.622471441296918
5.71 -0.610309920078286
5.72 -0.598039149847496
5.73 -0.585657612376127
5.74 -0.57316377641387
5.75 -0.560556098437169
5.76 -0.547833023469346
5.77 -0.534992985976582
5.78 -0.522034410844033
5.79 -0.50895571443677
5.8 -0.495755305750288
5.81 -0.482431587655447
5.82 -0.468982958243147
5.83 -0.455407812273835
5.84 -0.44170454273753
5.85 -0.427871542529824
5.86 -0.413907206249801
5.87 -0.399809932125789
5.88 -0.385578124075055
5.89 -0.371210193903678
5.9 -0.356704563652981
5.91 -0.342059668098959
5.92 -0.327273957411263
5.93 -0.312345899978295
5.94 -0.297273985405097
5.95 -0.282056727690621
5.96 -0.26669266859098
5.97 -0.251180381175221
5.98 -0.235518473579997
5.99 -0.219705592969395
6 -0.203740429705953
6.01 -0.187621721738623
6.02 -0.171348259213065
6.03 -0.154918889309333
6.04 -0.138332521311425
6.05 -0.121588131912704
6.06 -0.104684770760453
6.07 -0.0876215662421575
6.08 -0.0703977315151959
6.09 -0.0530125707806838
6.1 -0.0354654858011357
6.11 -0.0177559826604213
6.12 0.000116321236862629
6.13 0.0181516899309095
6.14 0.036350261443057
6.15 0.0547120400906164
6.16 0.0732368885897635
6.17 0.0919245199580675
6.18 0.110774489230493
6.19 0.12978618500497
6.2 0.148958820836197
6.21 0.168291426498962
6.22 0.18778283914501
6.23 0.207431694380479
6.24 0.227236417293859
6.25 0.2471952134677
6.26 0.267306060010383
6.27 0.287566696647725
6.28 0.307974616917423
6.29 0.328527059512771
6.3 0.349220999825425
6.31 0.370053141740373
6.32 0.391019909739311
6.33 0.412117441372004
6.34 0.433341580157792
6.35 0.45468786898244
6.36 0.476151544057624
6.37 0.497727529512617
6.38 0.51941043268934
6.39 0.54119454021309
6.4 0.563073814912078
6.41 0.585041893658805
6.42 0.607092086205957
6.43 0.629217375088085
6.44 0.651410416658472
6.45 0.673663543327496
6.46 0.695968767065441
6.47 0.71831778422793
6.48 0.740701981756657
6.49 0.763112444801737
6.5 0.785539965804649
6.51 0.807975055072387
6.52 0.830407952864256
6.53 0.852828643002829
6.54 0.875226868009621
6.55 0.897592145754485
6.56 0.919913787595758
6.57 0.942180917975206
6.58 0.964382495419016
6.59 0.986507334882633
6.6 1.00854413136393
6.61 1.03048148469586
6.62 1.05230792541682
6.63 1.07401194160425
6.64 1.09558200654517
6.65 1.11700660710644
6.66 1.1382742726569
6.67 1.15937360438573
6.68 1.18029330485242
6.69 1.20102220759935
6.7 1.22154930665235
6.71 1.24186378573267
6.72 1.26195504700289
6.73 1.28181273917032
6.74 1.3014267847746
6.75 1.32078740649077
6.76 1.33988515228651
6.77 1.35871091927995
6.78 1.37725597615519
6.79 1.3955119840043
6.8 1.41347101547773
6.81 1.43112557213912
6.82 1.44846859993632
6.83 1.46549350271604
6.84 1.48219415372657
6.85 1.49856490507003
6.86 1.51460059508233
6.87 1.53029655363671
6.88 1.54564860538242
6.89 1.56065307094713
6.9 1.57530676614636
6.91 1.58960699925754
6.92 1.60355156643028
6.93 1.61713874531564
6.94 1.63036728700863
6.95 1.64323640640725
6.96 1.65574577109922
6.97 1.66789548889449
6.98 1.67968609412577
6.99 1.69111853284373
7 1.70219414703508
7.01 1.71291465799244
7.02 1.72328214896488
7.03 1.73329904721568
7.04 1.74296810561145
7.05 1.75229238386279
7.06 1.76127522953213
7.07 1.76992025891865
7.08 1.77823133792465
7.09 1.7862125630008
7.1 1.79386824226088
7.11 1.80120287684929
7.12 1.80822114263754
7.13 1.81492787231793
7.14 1.82132803795557
7.15 1.82742673405218
7.16 1.83322916116817
7.17 1.83874061014217
7.18 1.84396644694044
7.19 1.84891209816263
7.2 1.8535830372235
7.21 1.8579847712251
7.22 1.86212282852839
7.23 1.86600274702833
7.24 1.86963006313203
7.25 1.87301030143556
7.26 1.87614896509126
7.27 1.87905152685414
7.28 1.88172342079326
7.29 1.88417003465119
7.3 1.88639670283303
7.31 1.88840870000414
7.32 1.8902112352746
7.33 1.89180944694729
7.34 1.89320839780531
7.35 1.89441307091424
7.36 1.89542836591381
7.37 1.89625909577391
7.38 1.89690998398949
7.39 1.89738566218898
7.4 1.89769066813164
7.41 1.8978294440691
7.42 1.89780633544723
7.43 1.89762558992495
7.44 1.89729135668744
7.45 1.89680768603175
7.46 1.89617852920384
7.47 1.89540773846668
7.48 1.89449906738025
7.49 1.89345617127487
7.5 1.89228260790021
7.51 1.89098183823355
7.52 1.88955722743139
7.53 1.88801204590947
7.54 1.88634947053745
7.55 1.88457258593487
7.56 1.88268438585618
7.57 1.88068777465345
7.58 1.8785855688058
7.59 1.87638049850587
7.6 1.87407520929388
7.61 1.8716722637308
7.62 1.86917414310276
7.63 1.86658324914941
7.64 1.86390190580959
7.65 1.86113236097807
7.66 1.858276788268
7.67 1.85533728877387
7.68 1.85231589283025
7.69 1.84921456176245
7.7 1.84603518962505
7.71 1.84277960492509
7.72 1.83944957232695
7.73 1.83604679433603
7.74 1.83257291295926
7.75 1.82902951133988
7.76 1.82541811536509
7.77 1.82174019524488
7.78 1.81799716706059
7.79 1.8141903942823
7.8 1.81032118925396
7.81 1.80639081464548
7.82 1.80240048487142
7.83 1.79835136747557
7.84 1.79424458448117
7.85 1.79008121370679
7.86 1.78586229004737
7.87 1.78158880672092
7.88 1.77726171648055
7.89 1.77288193279227
7.9 1.76845033097881
7.91 1.76396774932949
7.92 1.75943499017692
7.93 1.75485282094064
7.94 1.75022197513831
7.95 1.7455431533648
7.96 1.74081702423983
7.97 1.73604422532473
7.98 1.73122536400875
7.99 1.72636101836561
8 1.72145173798087
8.01 1.71649804475081
8.02 1.71150043365321
8.03 1.70645937349096
8.04 1.70137530760894
8.05 1.69624865458476
8.06 1.69107980889422
8.07 1.68586914155182
8.08 1.68061700072715
8.09 1.67532371233778
8.1 1.66998958061897
8.11 1.66461488867125
8.12 1.65919989898607
8.13 1.65374485395028
8.14 1.64824997633
8.15 1.64271546973439
8.16 1.63714151905985
8.17 1.63152829091521
8.18 1.62587593402853
8.19 1.62018457963581
8.2 1.61445434185216
8.21 1.60868531802609
8.22 1.60287758907721
8.23 1.59703121981767
8.24 1.59114625925814
8.25 1.58522274089845
8.26 1.57926068300341
8.27 1.57326008886414
8.28 1.56722094704554
8.29 1.56114323161974
8.3 1.55502690238663
8.31 1.54887190508107
8.32 1.54267817156769
8.33 1.5364456200232
8.34 1.53017415510683
8.35 1.52386366811895
8.36 1.51751403714829
8.37 1.51112512720801
8.38 1.50469679036094
8.39 1.49822886583402
8.4 1.49172118012248
8.41 1.48517354708386
8.42 1.47858576802199
8.43 1.47195763176135
8.44 1.46528891471181
8.45 1.45857938092415
8.46 1.45182878213634
8.47 1.44503685781095
8.48 1.43820333516369
8.49 1.43132792918341
8.5 1.42441034264353
8.51 1.41745026610524
8.52 1.41044737791249
8.53 1.40340134417893
8.54 1.39631181876692
8.55 1.38917844325883
8.56 1.38200084692062
8.57 1.37477864665781
8.58 1.36751144696416
8.59 1.36019883986283
8.6 1.35284040484045
8.61 1.34543570877399
8.62 1.33798430585058
8.63 1.33048573748035
8.64 1.32293953220252
8.65 1.31534520558455
8.66 1.30770226011478
8.67 1.30001018508829
8.68 1.29226845648634
8.69 1.28447653684938
8.7 1.2766338751436
8.71 1.26873990662135
8.72 1.26079405267533
8.73 1.25279572068657
8.74 1.24474430386657
8.75 1.23663918109345
8.76 1.22847971674231
8.77 1.22026526050983
8.78 1.21199514723335
8.79 1.20366869670436
8.8 1.19528521347675
8.81 1.18684398666971
8.82 1.17834428976553
8.83 1.16978538040256
8.84 1.16116650016324
8.85 1.15248687435755
8.86 1.14374571180203
8.87 1.13494220459448
8.88 1.12607552788478
8.89 1.11714483964162
8.9 1.10814928041593
8.91 1.09908797310086
8.92 1.0899600226888
8.93 1.0807645160256
8.94 1.07150052156255
8.95 1.06216708910622
8.96 1.05276324956675
8.97 1.04328801470489
8.98 1.0337403768783
8.99 1.02411930878756
9 1.01442376322242
9.01 1.00465267280896
9.02 0.99480494975804
9.03 0.984879485615941
9.04 0.974875151017692
9.05 0.964790795444035
9.06 0.954625246982646
9.07 0.944377312094619
9.08 0.934045775387143
9.09 0.9236293993932
9.1 0.913126924359631
9.11 0.902537068044517
9.12 0.891858525525131
9.13 0.881089969017847
9.14 0.870230047711357
9.15 0.859277387614689
9.16 0.848230591421686
9.17 0.837088238393558
9.18 0.825848884261474
9.19 0.814511061150928
9.2 0.803073277530155
9.21 0.791534018184693
9.22 0.779891744220422
9.23 0.768144893097587
9.24 0.756291878698484
9.25 0.744331091431547
9.26 0.732260898374837
9.27 0.720079643462118
9.28 0.707785647714782
9.29 0.695377209523275
9.3 0.682852604981585
9.31 0.670210088278902
9.32 0.65744789215255
9.33 0.644564228406558
9.34 0.631557288500549
9.35 0.61842524421379
9.36 0.605166248389558
9.37 0.591778435765179
9.38 0.578259923893406
9.39 0.564608814161082
9.4 0.550823192911231
9.41 0.536901132675188
9.42 0.522840693521389
9.43 0.508639924528001
9.44 0.494296865386709
9.45 0.479809548145342
9.46 0.465175999097211
9.47 0.450394240825423
9.48 0.435462294410702
9.49 0.42037818181142
9.5 0.405139928424908
9.51 0.389745565839285
9.52 0.374193134785333
9.53 0.358480688298048
9.54 0.342606295097747
9.55 0.32656804320071
9.56 0.310364043769448
9.57 0.293992435212719
9.58 0.277451387545423
9.59 0.260739107018466
9.6 0.2438538410285
9.61 0.226793883317341
9.62 0.20955757947048
9.63 0.192143332723772
9.64 0.174549610086949
9.65 0.156774948791897
9.66 0.138817963073057
9.67 0.120677351286316
9.68 0.10235190337188
9.69 0.0838405086653717
9.7 0.0651421640600827
9.71 0.0462559825218321
9.72 0.0271812019561089
9.73 0.00791719442530668
9.74 -0.0115365242882941
9.75 -0.0311802847805604
9.76 -0.0510142537938183
9.77 -0.0710384236676772
9.78 -0.0912526014876362
9.79 -0.111656397950132
9.8 -0.132249215966311
9.81 -0.153030239030704
9.82 -0.173998419385095
9.83 -0.195152466012278
9.84 -0.216490832499057
9.85 -0.238011704812583
9.86 -0.259712989039246
9.87 -0.281592299140491
9.88 -0.303646944785273
9.89 -0.325873919324261
9.9 -0.348269887976373
9.91 -0.37083117630365
9.92 -0.393553759055858
9.93 -0.416433249471381
9.94 -0.439464889125971
9.95 -0.462643538425556
9.96 -0.485963667843546
9.97 -0.509419350006883
9.98 -0.53300425273805
9.99 -0.556711633162782
10 -0.580534332994759
10.01 -0.604464775109035
10.02 -0.628494961515654
10.03 -0.652616472843155
10.04 -0.67682046943887
10.05 -0.701097694188464
10.06 -0.725438477151677
10.07 -0.749832742103731
10.08 -0.774270015063304
10.09 -0.79873943487731
10.1 -0.823229765920924
10.11 -0.847729412957601
10.12 -0.872226438188694
10.13 -0.896708580505843
10.14 -0.921163276941258
10.15 -0.945577686292122
10.16 -0.969938714875065
10.17 -0.994233044346143
10.18 -1.01844716149998
10.19 -1.04256738994047
};
\addplot [semithick, color1]
table {%
0 0
0.01 0.00866412491310726
0.02 0.0173760203725251
0.03 0.0261349525100355
0.04 0.0349401774675676
0.05 0.0437909417048229
0.06 0.0526864823122119
0.07 0.0616260273286454
0.08 0.0706087960637038
0.09 0.0796339994237253
0.1 0.0887008402413241
0.11 0.0978085136078717
0.12 0.106956207208452
0.13 0.116143101658827
0.14 0.125368370843914
0.15 0.134631182257321
0.16 0.143930697341458
0.17 0.153266071827766
0.18 0.162636456076593
0.19 0.172040995416284
0.2 0.181478830481016
0.21 0.190949097546978
0.22 0.200450928866444
0.23 0.209983452999322
0.24 0.219545795141824
0.25 0.229137077451806
0.26 0.23875641937045
0.27 0.248402937939916
0.28 0.258075748116606
0.29 0.267773963079721
0.3 0.277496694534783
0.31 0.287243053011865
0.32 0.297012148158167
0.33 0.30680308902476
0.34 0.316614984347173
0.35 0.326446942819642
0.36 0.336298073362786
0.37 0.346167485384505
0.38 0.35605428903391
0.39 0.365957595448157
0.4 0.37587651699199
0.41 0.385810167489894
0.42 0.395757662450706
0.43 0.40571811928463
0.44 0.415690657512532
0.45 0.425674398967453
0.46 0.435668467988293
0.47 0.44567199160562
0.48 0.455684099719582
0.49 0.465703925269892
0.5 0.475730604397898
0.51 0.485763276600756
0.52 0.495801084877716
0.53 0.505843175868535
0.54 0.515888699984149
0.55 0.525936811529512
0.56 0.535986668818828
0.57 0.546037434283133
0.58 0.55608827457037
0.59 0.566138360638023
0.6 0.57618686783842
0.61 0.586232975996775
0.62 0.596275869482167
0.63 0.606314737271469
0.64 0.616348773006416
0.65 0.626377175043931
0.66 0.636399146499798
0.67 0.64641389528587
0.68 0.656420634140941
0.69 0.666418580655382
0.7 0.676406957289725
0.71 0.686384991387335
0.72 0.696351915181314
0.73 0.706306965795759
0.74 0.716249385241568
0.75 0.726178420406907
0.76 0.736093323042501
0.77 0.745993349741904
0.78 0.755877761916884
0.79 0.765745825768054
0.8 0.775596812250966
0.81 0.785429997037687
0.82 0.79524466047416
0.83 0.805040087533332
0.84 0.814815567764293
0.85 0.824570395237546
0.86 0.83430386848649
0.87 0.844015290445324
0.88 0.853703968383427
0.89 0.863369213836381
0.9 0.87301034253378
0.91 0.882626674323873
0.92 0.892217533095247
0.93 0.901782246695579
0.94 0.91132014684767
0.95 0.920830569062709
0.96 0.930312852551046
0.97 0.939766340130451
0.98 0.949190378132051
0.99 0.958584316303919
1 0.967947507712504
1.01 0.977279308641987
1.02 0.986579078491527
1.03 0.995846179670685
1.04 1.00507997749287
1.05 1.01427984006706
1.06 1.02344513818775
1.07 1.0325752452233
1.08 1.04166953700259
1.09 1.05072739170032
1.1 1.05974818972054
1.11 1.06873131357907
1.12 1.07767614778426
1.13 1.08658207871669
1.14 1.09544849450736
1.15 1.10427478491469
1.16 1.11306034120044
1.17 1.12180455600429
1.18 1.1305068232174
1.19 1.13916653785479
1.2 1.14778309592669
1.21 1.1563558943089
1.22 1.16488433061191
1.23 1.17336780304929
1.24 1.18180571030485
1.25 1.19019745139908
1.26 1.19854242555441
1.27 1.20684003205965
1.28 1.21508967013352
1.29 1.2232907387872
1.3 1.23144263668607
1.31 1.23954476201045
1.32 1.24759651231553
1.33 1.25559728439043
1.34 1.26354647411627
1.35 1.27144347632355
1.36 1.2792876846485
1.37 1.28707849138865
1.38 1.29481528735754
1.39 1.30249746173851
1.4 1.31012440193768
1.41 1.3176954934361
1.42 1.32521011964096
1.43 1.33266766173603
1.44 1.34006749853121
1.45 1.34740900631125
1.46 1.35469155868359
1.47 1.3619145264254
1.48 1.3690772773298
1.49 1.3761791760512
1.5 1.38321958394985
1.51 1.39019785893561
1.52 1.39711335531087
1.53 1.40396542361271
1.54 1.41075341045437
1.55 1.41747665836578
1.56 1.42413450563364
1.57 1.43072628614053
1.58 1.43725132920353
1.59 1.44370895941212
1.6 1.45009849646555
1.61 1.45641925500956
1.62 1.46267054447261
1.63 1.46885166890174
1.64 1.47496192679786
1.65 1.48100061095081
1.66 1.48696700827415
1.67 1.49286039963968
1.68 1.49868005971181
1.69 1.50442525678212
1.7 1.5100952526036
1.71 1.5156893022255
1.72 1.52120665382817
1.73 1.5266465485584
1.74 1.5320082203654
1.75 1.53729089583732
1.76 1.54249379403881
1.77 1.54761612634949
1.78 1.55265709630367
1.79 1.55761589943155
1.8 1.56249172310195
1.81 1.56728374636703
1.82 1.57199113980902
1.83 1.57661306538937
1.84 1.58114867630057
1.85 1.58559711682092
1.86 1.58995752217255
1.87 1.59422901838314
1.88 1.59841072215157
1.89 1.60250174071801
1.9 1.60650117173871
1.91 1.61040810316613
1.92 1.61422161313477
1.93 1.61794076985306
1.94 1.62156463150218
1.95 1.62509224614199
1.96 1.62852265162487
1.97 1.63185487551811
1.98 1.63508793503531
1.99 1.63822083697778
2 1.64125257768625
2.01 1.64418214300419
2.02 1.64700850825295
2.03 1.64973063822009
2.04 1.65234748716148
2.05 1.65485799881813
2.06 1.65726110644888
2.07 1.65955573287975
2.08 1.66174079057113
2.09 1.66381518170399
2.1 1.66577779828608
2.11 1.66762752227947
2.12 1.66936322575064
2.13 1.67098377104431
2.14 1.67248801098256
2.15 1.67387478909044
2.16 1.67514293984959
2.17 1.67629128898131
2.18 1.67731865376064
2.19 1.67822384336294
2.2 1.6790056592447
2.21 1.67966289556002
2.22 1.68019433961461
2.23 1.68059877235887
2.24 1.68087496892183
2.25 1.68102169918766
2.26 1.68103772841656
2.27 1.68092181791161
2.28 1.68067272573352
2.29 1.68028920746492
2.3 1.67977001702591
2.31 1.67911390754254
2.32 1.67831963226994
2.33 1.67738594557164
2.34 1.67631160395647
2.35 1.67509536717487
2.36 1.67373599937546
2.37 1.6722322703234
2.38 1.67058295668175
2.39 1.66878684335646
2.4 1.66684272490588
2.41 1.66474940701572
2.42 1.66250570803938
2.43 1.66011046060401
2.44 1.65756251328234
2.45 1.65486073232965
2.46 1.65200400348551
2.47 1.64899123383895
2.48 1.64582135375605
2.49 1.64249331886797
2.5 1.63900611211762
2.51 1.63535874586205
2.52 1.631550264028
2.53 1.62757974431693
2.54 1.62344630045579
2.55 1.61914908448905
2.56 1.61468728910723
2.57 1.61006015000641
2.58 1.60526694827291
2.59 1.60030701278668
2.6 1.5951797226364
2.61 1.5898845095389
2.62 1.58442086025496
2.63 1.57878831899285
2.64 1.57298648979092
2.65 1.56701503886977
2.66 1.56087369694429
2.67 1.55456226148534
2.68 1.54808059892098
2.69 1.54142864676617
2.7 1.53460641567024
2.71 1.52761399137135
2.72 1.52045153654646
2.73 1.51311929254608
2.74 1.50561758100289
2.75 1.49794680530323
2.76 1.4901074519113
2.77 1.48210009153577
2.78 1.47392538012939
2.79 1.46558405971228
2.8 1.45707695901083
2.81 1.4484049939042
2.82 1.43956916767207
2.83 1.43057057103718
2.84 1.42141038199828
2.85 1.41208986544938
2.86 1.40261037258251
2.87 1.39297334007281
2.88 1.38318028904535
2.89 1.37323282382461
2.9 1.36313263046897
2.91 1.35288147509379
2.92 1.3424812019874
2.93 1.33193373152646
2.94 1.32124105789757
2.95 1.31040524663364
2.96 1.29942843197438
2.97 1.28831281406149
2.98 1.27706065597985
2.99 1.26567428065707
3 1.25415606763451
3.01 1.24250844972319
3.02 1.2307339095591
3.03 1.21883497607234
3.04 1.20681422088517
3.05 1.19467425465378
3.06 1.18241772336923
3.07 1.17004730463251
3.08 1.15756570391854
3.09 1.1449756508439
3.1 1.13227989545229
3.11 1.11948120453166
3.12 1.10658235797588
3.13 1.09358614520387
3.14 1.0804953616474
3.15 1.06731280531902
3.16 1.05404127347002
3.17 1.04068355934775
3.18 1.02724244906067
3.19 1.01372071855849
3.2 1.00012113073424
3.21 0.986446432653355
3.22 0.972699352914968
3.23 0.958882599148763
3.24 0.944998855650617
3.25 0.931050781158823
3.26 0.917041006772288
3.27 0.902972134011123
3.28 0.888846733019313
3.29 0.874667340908636
3.3 0.86043646024216
3.31 0.846156557655216
3.32 0.831830062611243
3.33 0.81745936628925
3.34 0.803046820599427
3.35 0.788594737322992
3.36 0.774105387372031
3.37 0.759581000164816
3.38 0.745023763111914
3.39 0.730435821208248
3.4 0.715819276726032
3.41 0.701176189003538
3.42 0.686508574324498
3.43 0.671818405883024
3.44 0.657107613828833
3.45 0.642378085387729
3.46 0.627631665052203
3.47 0.612870154837294
3.48 0.598095314596732
3.49 0.583308862394781
3.5 0.568512474929026
3.51 0.553707787999822
3.52 0.538896397021989
3.53 0.524079857574759
3.54 0.509259685985955
3.55 0.494437359946712
3.56 0.479614319153105
3.57 0.464791965971374
3.58 0.44997166612346
3.59 0.435154749389877
3.6 0.420342510327084
3.61 0.405536208996666
3.62 0.390737071703847
3.63 0.375946291742995
3.64 0.361165030147997
3.65 0.346394416445479
3.66 0.331635549409005
3.67 0.316889497812606
3.68 0.302157301182011
3.69 0.287439970542245
3.7 0.272738489160176
3.71 0.258053813280964
3.72 0.243386872857242
3.73 0.228738572270127
3.74 0.214109791041223
3.75 0.199501384534799
3.76 0.184914184649546
3.77 0.170349000499285
3.78 0.155806619082141
3.79 0.141287805937749
3.8 0.126793305792128
3.81 0.112323843189926
3.82 0.0978801231138059
3.83 0.0834628315907663
3.84 0.0690726362852624
3.85 0.0547101870790615
3.86 0.0403761166377401
3.87 0.0260710409638481
3.88 0.0117955599367326
3.89 -0.00244974216090219
3.9 -0.0166642961296404
3.91 -0.0308475473530781
3.92 -0.044998955310516
3.93 -0.0591179931024551
3.94 -0.0732041469887432
3.95 -0.0872569159391737
3.96 -0.10127581119636
3.97 -0.115260355850658
3.98 -0.129210084426933
3.99 -0.143124542482926
4 -0.157003286219014
4.01 -0.170845882099089
4.02 -0.184651906482329
4.03 -0.198420945265623
4.04 -0.212152593536364
4.05 -0.225846455235401
4.06 -0.239502142829841
4.07 -0.253119276995512
4.08 -0.266697486308775
4.09 -0.280236406947446
4.1 -0.293735682400618
4.11 -0.307194963187069
4.12 -0.320613906582062
4.13 -0.333992176352232
4.14 -0.347329442498431
4.15 -0.360625381006137
4.16 -0.37387967360334
4.17 -0.387092007525574
4.18 -0.400262075287934
4.19 -0.413389574463812
4.2 -0.426474207470168
4.21 -0.439515681359094
4.22 -0.452513707615493
4.23 -0.465468001960647
4.24 -0.478378284161486
4.25 -0.491244277845362
4.26 -0.50406571032014
4.27 -0.516842312399428
4.28 -0.529573818232741
4.29 -0.542259965140452
4.3 -0.554900493453369
4.31 -0.567495146356684
4.32 -0.580043669738305
4.33 -0.592545812041225
4.34 -0.60500132411987
4.35 -0.617409959100323
4.36 -0.629771472244145
4.37 -0.642085620815776
4.38 -0.654352163953346
4.39 -0.666570862542696
4.4 -0.678741479094596
4.41 -0.690863777624927
4.42 -0.702937523537791
4.43 -0.714962483511353
4.44 -0.726938425386358
4.45 -0.738865118057227
4.46 -0.75074233136552
4.47 -0.762569835995792
4.48 -0.774347403373622
4.49 -0.786074805565847
4.5 -0.79775181518277
4.51 -0.809378205282325
4.52 -0.820953749276099
4.53 -0.832478220837124
4.54 -0.843951393809335
4.55 -0.855373042118598
4.56 -0.866742939685309
4.57 -0.878060860338363
4.58 -0.889326577730509
4.59 -0.900539865255017
4.6 -0.911700495963486
4.61 -0.922808242484856
4.62 -0.933862876945466
4.63 -0.944864170890087
4.64 -0.955811895203954
4.65 -0.966705820035614
4.66 -0.977545714720624
4.67 -0.988331347705978
4.68 -0.999062486475262
4.69 -1.0097388974744
4.7 -1.020360346038
4.71 -1.03092659631623
4.72 -1.04143741120218
4.73 -1.05189255225959
4.74 -1.06229177965097
4.75 -1.07263485206609
4.76 -1.08292152665064
4.77 -1.0931515589352
4.78 -1.1033247027643
4.79 -1.11344071022565
4.8 -1.12349933157945
4.81 -1.13350031518777
4.82 -1.14344340744376
4.83 -1.15332835270107
4.84 -1.16315489320288
4.85 -1.17292276901108
4.86 -1.18263171793502
4.87 -1.19228147546022
4.88 -1.20187177467669
4.89 -1.21140234620708
4.9 -1.22087291813438
4.91 -1.2302832159293
4.92 -1.23963296237721
4.93 -1.24892187750471
4.94 -1.25814967850554
4.95 -1.26731607966615
4.96 -1.27642079229061
4.97 -1.2854635246249
4.98 -1.29444398178071
4.99 -1.30336186565839
5 -1.31221687486934
5.01 -1.32100870465759
5.02 -1.32973704682061
5.03 -1.33840158962943
5.04 -1.3470020177478
5.05 -1.3555380121505
5.06 -1.36400925004081
5.07 -1.37241540476708
5.08 -1.38075614573824
5.09 -1.38903113833836
5.1 -1.39724004384025
5.11 -1.40538251931793
5.12 -1.41345821755814
5.13 -1.42146678697054
5.14 -1.42940787149706
5.15 -1.43728111051988
5.16 -1.44508613876825
5.17 -1.45282258622428
5.18 -1.4604900780272
5.19 -1.46808823437666
5.2 -1.47561667043462
5.21 -1.48307499622586
5.22 -1.49046281653736
5.23 -1.49777973081621
5.24 -1.50502533306618
5.25 -1.51219921174298
5.26 -1.51930094964811
5.27 -1.52633012382125
5.28 -1.53328630543138
5.29 -1.54016905966628
5.3 -1.54697794562085
5.31 -1.55371251618381
5.32 -1.560372317923
5.33 -1.56695689096937
5.34 -1.57346576889936
5.35 -1.57989847861606
5.36 -1.5862545402287
5.37 -1.59253346693097
5.38 -1.59873476487788
5.39 -1.60485793306112
5.4 -1.61090246318332
5.41 -1.61686783953073
5.42 -1.62275353884483
5.43 -1.62855903019258
5.44 -1.63428377483546
5.45 -1.63992722609737
5.46 -1.6454888292315
5.47 -1.650968021286
5.48 -1.65636423096885
5.49 -1.66167687851168
5.5 -1.66690537553286
5.51 -1.67204912489989
5.52 -1.67710752059102
5.53 -1.68207994755651
5.54 -1.68696578157943
5.55 -1.69176438913613
5.56 -1.69647512725671
5.57 -1.7010973433854
5.58 -1.70563037524119
5.59 -1.71007355067878
5.6 -1.71442618755016
5.61 -1.71868759356689
5.62 -1.7228570661634
5.63 -1.72693389236151
5.64 -1.73091734863645
5.65 -1.73480670078465
5.66 -1.73860120379351
5.67 -1.74230010171369
5.68 -1.74590262753394
5.69 -1.74940800305912
5.7 -1.75281543879165
5.71 -1.75612413381676
5.72 -1.75933327569214
5.73 -1.76244204034223
5.74 -1.76544959195786
5.75 -1.76835508290159
5.76 -1.77115765361945
5.77 -1.7738564325595
5.78 -1.77645053609797
5.79 -1.77893906847358
5.8 -1.78132112173078
5.81 -1.78359577567243
5.82 -1.78576209782311
5.83 -1.78781914340341
5.84 -1.78976595531639
5.85 -1.7916015641469
5.86 -1.79332498817485
5.87 -1.79493523340334
5.88 -1.79643129360268
5.89 -1.79781215037148
5.9 -1.79907677321579
5.91 -1.80022411964768
5.92 -1.80125313530425
5.93 -1.80216275408864
5.94 -1.80295189833417
5.95 -1.80361947899305
5.96 -1.80416439585124
5.97 -1.8045855377707
5.98 -1.80488178296087
5.99 -1.8050519992808
6 -1.8050950445736
6.01 -1.80500976703489
6.02 -1.80479500561717
6.03 -1.80444959047156
6.04 -1.80397234342903
6.05 -1.80336207852282
6.06 -1.80261760255388
6.07 -1.80173771570142
6.08 -1.80072121218031
6.09 -1.7995668809473
6.1 -1.79827350645797
6.11 -1.79683986947647
6.12 -1.79526474793975
6.13 -1.79354691787812
6.14 -1.79168515439435
6.15 -1.78967823270251
6.16 -1.7875249292288
6.17 -1.7852240227755
6.18 -1.78277429574998
6.19 -1.78017453545988
6.2 -1.77742353547574
6.21 -1.77452009706242
6.22 -1.77146303067984
6.23 -1.76825115755403
6.24 -1.7648833113189
6.25 -1.7613583397289
6.26 -1.75767510644249
6.27 -1.75383249287634
6.28 -1.74982940012909
6.29 -1.74566475097405
6.3 -1.74133749191923
6.31 -1.73684659533284
6.32 -1.73219106163214
6.33 -1.7273699215326
6.34 -1.72238223835453
6.35 -1.71722711038317
6.36 -1.71190367327799
6.37 -1.70641110252657
6.38 -1.70074861593764
6.39 -1.69491547616713
6.4 -1.68891099327105
6.41 -1.68273452727789
6.42 -1.67638549077283
6.43 -1.66986335148572
6.44 -1.66316763487383
6.45 -1.65629792669015
6.46 -1.6492538755273
6.47 -1.64203519532672
6.48 -1.63464166784249
6.49 -1.6270731450484
6.5 -1.61932955147702
6.51 -1.6114108864788
6.52 -1.60331722638924
6.53 -1.59504872659223
6.54 -1.58660562346722
6.55 -1.57798823620807
6.56 -1.56919696850192
6.57 -1.56023231005611
6.58 -1.55109483796182
6.59 -1.54178521788371
6.6 -1.53230420506488
6.61 -1.52265264513777
6.62 -1.51283147473184
6.63 -1.50284172187027
6.64 -1.49268450614803
6.65 -1.48236103868608
6.66 -1.47187262185602
6.67 -1.4612206487722
6.68 -1.45040660254834
6.69 -1.439432055318
6.7 -1.42829866701911
6.71 -1.41700818394417
6.72 -1.40556243705956
6.73 -1.3939633400981
6.74 -1.38221288743124
6.75 -1.37031315172805
6.76 -1.35826628140941
6.77 -1.34607449790776
6.78 -1.3337400927432
6.79 -1.32126542442803
6.8 -1.30865291521318
6.81 -1.2959050476903
6.82 -1.28302436126454
6.83 -1.27001344851327
6.84 -1.25687495144679
6.85 -1.24361155768734
6.86 -1.23022599658288
6.87 -1.21672103527239
6.88 -1.20309947471916
6.89 -1.18936414572864
6.9 -1.17551790496686
6.91 -1.16156363099516
6.92 -1.1475042203365
6.93 -1.13334258358783
6.94 -1.11908164159241
6.95 -1.10472432168512
6.96 -1.09027355402296
6.97 -1.07573226801191
6.98 -1.06110338884061
6.99 -1.04638983412995
7 -1.03159451070706
7.01 -1.01672031151063
7.02 -1.00177011263411
7.03 -0.986746770511672
7.04 -0.971653119251232
7.05 -0.956491968117642
7.06 -0.941266099168401
7.07 -0.925978265043076
7.08 -0.91063118690704
7.09 -0.895227552549088
7.1 -0.879770014631925
7.11 -0.864261189093699
7.12 -0.848703653698187
7.13 -0.833099946730622
7.14 -0.817452565835691
7.15 -0.801763966993691
7.16 -0.786036563630504
7.17 -0.77027272585673
7.18 -0.754474779830905
7.19 -0.738645007241665
7.2 -0.722785644903423
7.21 -0.706898884460037
7.22 -0.690986872190866
7.23 -0.675051708913447
7.24 -0.659095449977273
7.25 -0.643120105342911
7.26 -0.627127639740781
7.27 -0.611119972904316
7.28 -0.595098979871792
7.29 -0.579066491351751
7.3 -0.563024294146784
7.31 -0.546974131630653
7.32 -0.530917704274066
7.33 -0.514856670214351
7.34 -0.498792645864651
7.35 -0.48272720655838
7.36 -0.466661887224948
7.37 -0.450598183092864
7.38 -0.434537550416637
7.39 -0.418481407224015
7.4 -0.402431134080377
7.41 -0.386388074867227
7.42 -0.370353537572001
7.43 -0.35432879508652
7.44 -0.338315086011638
7.45 -0.322313615465849
7.46 -0.306325555895717
7.47 -0.290352047886247
7.48 -0.27439420096936
7.49 -0.258453094428935
7.5 -0.242529778100845
7.51 -0.226625273166771
7.52 -0.210740572940469
7.53 -0.194876643645496
7.54 -0.179034425183379
7.55 -0.163214831891395
7.56 -0.147418753289195
7.57 -0.13164705481361
7.58 -0.115900578541077
7.59 -0.100180143897207
7.6 -0.0844865483530417
7.61 -0.0688205681077168
7.62 -0.0531829587572209
7.63 -0.0375744559490453
7.64 -0.0219957760225818
7.65 -0.00644761663513665
7.66 0.00906934262646939
7.67 0.0245544396487509
7.68 0.0400070292088209
7.69 0.0554264824060278
7.7 0.070812186107857
7.71 0.0861635424099886
7.72 0.101479968110367
7.73 0.116760894197128
7.74 0.132005765350168
7.75 0.147214039456189
7.76 0.162385187136967
7.77 0.177518691290647
7.78 0.192614046645778
7.79 0.207670759327878
7.8 0.222688346438238
7.81 0.237666335644717
7.82 0.252604264784234
7.83 0.267501681476707
7.84 0.282358142750125
7.85 0.297173214676523
7.86 0.311946472018506
7.87 0.326677497886112
7.88 0.34136588340368
7.89 0.356011227386476
7.9 0.370613136026758
7.91 0.385171222589068
7.92 0.399685107114399
7.93 0.414154416133025
7.94 0.428578782385714
7.95 0.442957844553045
7.96 0.45729124699259
7.97 0.471578639483669
7.98 0.485819676979521
7.99 0.500014019366508
8 0.514161331230244
8.01 0.528261281628315
8.02 0.542313543869447
8.03 0.556317795298814
8.04 0.570273717089349
8.05 0.584180994038708
8.06 0.598039314371917
8.07 0.611848369549224
8.08 0.62560785407913
8.09 0.639317465336382
8.1 0.652976903384676
8.11 0.666585870803957
8.12 0.680144072522096
8.13 0.693651215650779
8.14 0.707107009325426
8.15 0.72051116454903
8.16 0.733863394039673
8.17 0.747163412081615
8.18 0.760410934379802
8.19 0.773605677917651
8.2 0.786747360817874
8.21 0.799835702206416
8.22 0.812870422079098
8.23 0.825851241171057
8.24 0.838777880828749
8.25 0.85165006288439
8.26 0.864467509532749
8.27 0.877229943210148
8.28 0.889937086475572
8.29 0.902588661893783
8.3 0.915184391920272
8.31 0.927723998788034
8.32 0.940207204395956
8.33 0.952633730198852
8.34 0.965003297098911
8.35 0.977315625338542
8.36 0.989570434394477
8.37 1.00176744287313
8.38 1.01390636840697
8.39 1.02598692755197
8.4 1.03800883568604
8.41 1.04997180690813
8.42 1.06187555393831
8.43 1.07371978801837
8.44 1.08550421881317
8.45 1.09722855431246
8.46 1.10889250073309
8.47 1.12049576242177
8.48 1.132038041758
8.49 1.14351903905734
8.5 1.15493845247484
8.51 1.16629597790852
8.52 1.17759130890301
8.53 1.18882413655308
8.54 1.19999414940713
8.55 1.21110103337047
8.56 1.22214447160855
8.57 1.23312414444972
8.58 1.24403972928781
8.59 1.25489090048428
8.6 1.2656773292699
8.61 1.27639868364596
8.62 1.28705462828493
8.63 1.29764482443044
8.64 1.3081689297967
8.65 1.3186265984671
8.66 1.32901748079211
8.67 1.33934122328626
8.68 1.34959746852435
8.69 1.3597858550366
8.7 1.36990601720299
8.71 1.37995758514636
8.72 1.38994018462464
8.73 1.39985343692173
8.74 1.40969695873741
8.75 1.41947036207588
8.76 1.42917325413306
8.77 1.43880523718267
8.78 1.44836590846084
8.79 1.45785486004928
8.8 1.46727167875714
8.81 1.4766159460013
8.82 1.48588723768515
8.83 1.49508512407574
8.84 1.50420916967943
8.85 1.51325893311583
8.86 1.52223396699004
8.87 1.53113381776324
8.88 1.53995802562149
8.89 1.54870612434278
8.9 1.55737764116226
8.91 1.56597209663561
8.92 1.57448900450067
8.93 1.58292787153696
8.94 1.59128819742362
8.95 1.59956947459519
8.96 1.60777118809552
8.97 1.61589281542984
8.98 1.62393382641474
8.99 1.63189368302631
9 1.63977183924633
9.01 1.64756774090646
9.02 1.65528082553051
9.03 1.66291052217486
9.04 1.67045625126692
9.05 1.67791742444174
9.06 1.68529344437682
9.07 1.69258370462502
9.08 1.69978758944591
9.09 1.70690447363513
9.1 1.71393372235247
9.11 1.72087469094811
9.12 1.72772672478746
9.13 1.73448915907471
9.14 1.74116131867493
9.15 1.74774251793512
9.16 1.75423206050412
9.17 1.76062923915166
9.18 1.76693333558664
9.19 1.77314362027487
9.2 1.77925935225624
9.21 1.78527977896188
9.22 1.7912041360312
9.23 1.79703164712927
9.24 1.80276152376458
9.25 1.80839296510773
9.26 1.81392515781115
9.27 1.81935727583017
9.28 1.82468848024599
9.29 1.82991791909069
9.3 1.83504472717485
9.31 1.84006802591817
9.32 1.84498692318349
9.33 1.84980051311493
9.34 1.85450787598042
9.35 1.85910807801931
9.36 1.86360017129584
9.37 1.86798319355881
9.38 1.87225616810845
9.39 1.87641810367108
9.4 1.88046799428236
9.41 1.88440481918008
9.42 1.88822754270731
9.43 1.89193511422681
9.44 1.89552646804789
9.45 1.89900052336663
9.46 1.90235618422067
9.47 1.90559233945971
9.48 1.9087078627331
9.49 1.91170161249567
9.5 1.91457243203347
9.51 1.91731914951055
9.52 1.91994057803864
9.53 1.92243551577121
9.54 1.92480274602364
9.55 1.92704103742116
9.56 1.92914914407675
9.57 1.93112580580054
9.58 1.93296974834301
9.59 1.93467968367395
9.6 1.93625431029943
9.61 1.93769231361905
9.62 1.93899236632565
9.63 1.94015312885011
9.64 1.94117324985357
9.65 1.94205136676958
9.66 1.94278610639901
9.67 1.94337608555995
9.68 1.94381991179579
9.69 1.94411618414383
9.7 1.94426349396743
9.71 1.94426042585432
9.72 1.94410555858393
9.73 1.94379746616654
9.74 1.94333471895687
9.75 1.94271588484498
9.76 1.94193953052693
9.77 1.94100422285795
9.78 1.93990853029034
9.79 1.93865102439853
9.8 1.93723028149344
9.81 1.93564488432799
9.82 1.93389342389556
9.83 1.93197450132277
9.84 1.9298867298578
9.85 1.92762873695503
9.86 1.92519916645654
9.87 1.92259668087021
9.88 1.91981996374447
9.89 1.91686772213833
9.9 1.91373868918545
9.91 1.91043162675022
9.92 1.90694532817299
9.93 1.90327862110132
9.94 1.89943037040305
9.95 1.89539948115637
9.96 1.89118490171142
9.97 1.88678562681691
9.98 1.88220070080447
9.99 1.87742922082277
10 1.87247034011229
10.01 1.86732327131087
10.02 1.86198728977928
10.03 1.8564617369352
10.04 1.8507460235831
10.05 1.84483963322655
10.06 1.83874212534903
10.07 1.83245313864815
10.08 1.825972394208
10.09 1.81929969859318
10.1 1.81243494684817
10.11 1.80537812538466
10.12 1.79812931473984
10.13 1.79068869218791
10.14 1.78305653418737
10.15 1.7752332186466
10.16 1.7672192269908
10.17 1.75901514601338
10.18 1.75062166949588
10.19 1.7420395995811
};
\end{axis}

\end{tikzpicture}

%% file: fig/exp30_final_version.tex
\begin{tikzpicture}

\begin{groupplot}[group style={group size=2 by 2, vertical sep = 2cm, horizontal sep = 1.5cm}]
\nextgroupplot[
height=\figheight,
scale only axis,
tick align=outside,
tick pos=left,
title={linear ODE \cref{eq:ex_ODE_linear} \rev{and $R\equiv 0$}},
width=\figwidth,
x grid style={lightgray!92.02614379084967!black},
xlabel near ticks,
xmin=0.726123185772791, xmax=2.87505767255421,
xtick={0.5,1,1.5,2,2.5,3},
xticklabels={$10^{-0.5}$,$10^{-1.0}$,$10^{-1.5}$,$10^{-2.0}$,$10^{-2.5}$,$10^{-3.0}$},
y grid style={lightgray!92.02614379084967!black},
ylabel near ticks,
ylabel={$ \# \mathrm{Evals}\ \mathrm{of}\ f$},
ymin=2.42486396097624, ymax=3.44627446279562,
ytick={2.4,2.6,2.8,3,3.2,3.4,3.6},
yticklabels={$10^{2.4}$,$10^{2.6}$,$10^{2.8}$,$10^{3.0}$,$10^{3.2}$,$10^{3.4}$,$10^{3.6}$}
]
\addplot [semithick, blue, mark=*, mark size=2, mark options={solid}, only marks, forget plot]
table [row sep=\\]{%
2.68900273863374    3.39984671271292 \\
2.54623002501341    3.32837960343874 \\
2.40340783893506    3.25695815256093 \\
2.26057439679073    3.18554215485437 \\
2.11758026344912    3.11427729656159 \\
1.97500845394718    3.04257551244019 \\
1.83200954005845    2.97127584873811 \\
1.68912794026525    2.8998205024271 \\
1.54653711522508    2.82801506422398 \\
1.40342582985496    2.75663610824585 \\
1.26058585033573    2.68484536164441 \\
1.116739720397  2.61384182187607 \\
0.972580349388372   2.54282542695918 \\
0.828453399429049   2.47129171105894 \\
};
\addplot [semithick, blue]
table [row sep=\\]{%
2.68900273863374    3.39984671271292 \\
2.54623002501341    3.32837960343874 \\
2.40340783893506    3.25695815256093 \\
2.26057439679073    3.18554215485437 \\
2.11758026344912    3.11427729656159 \\
1.97500845394718    3.04257551244019 \\
1.83200954005845    2.97127584873811 \\
1.68912794026525    2.8998205024271 \\
1.54653711522508    2.82801506422398 \\
1.40342582985496    2.75663610824585 \\
1.26058585033573    2.68484536164441 \\
1.116739720397  2.61384182187607 \\
0.972580349388372   2.54282542695918 \\
0.828453399429049   2.47129171105894 \\
};
\addplot [thick, black, dotted]
table [row sep=\\]{%
0.828453399429049   2.47129171105894 \\
0.87022083357446    2.51305914520435 \\
0.911988267719872   2.55482657934976 \\
0.953755701865283   2.59659401349517 \\
0.995523136010694   2.63836144764058 \\
1.03729057015611    2.680128881786 \\
1.07905800430152    2.72189631593141 \\
1.12082543844693    2.76366375007682 \\
1.16259287259234    2.80543118422223 \\
1.20436030673775    2.84719861836764 \\
1.24612774088316    2.88896605251305 \\
1.28789517502857    2.93073348665846 \\
1.32966260917399    2.97250092080388 \\
1.3714300433194 3.01426835494929 \\
1.41319747746481    3.0560357890947 \\
1.45496491161022    3.09780322324011 \\
1.49673234575563    3.13957065738552 \\
1.53849977990104    3.18133809153093 \\
1.58026721404645    3.22310552567635 \\
1.62203464819187    3.26487295982176 \\
1.66380208233728    3.30664039396717 \\
1.70556951648269    3.34840782811258 \\
1.7473369506281 3.39017526225799 \\
1.78910438477351    3.4319426964034 \\
1.83087181891892    3.47371013054881 \\
1.87263925306434    3.51547756469423 \\
1.91440668720975    3.55724499883964 \\
1.95617412135516    3.59901243298505 \\
1.99794155550057    3.64077986713046 \\
2.03970898964598    3.68254730127587 \\
2.08147642379139    3.72431473542128 \\
2.1232438579368 3.76608216956669 \\
2.16501129208222    3.80784960371211 \\
2.20677872622763    3.84961703785752 \\
2.24854616037304    3.89138447200293 \\
2.29031359451845    3.93315190614834 \\
2.33208102866386    3.97491934029375 \\
2.37384846280927    4.01668677443916 \\
2.41561589695468    4.05845420858457 \\
2.4573833311001 4.10022164272999 \\
2.49915076524551    4.1419890768754 \\
2.54091819939092    4.18375651102081 \\
2.58268563353633    4.22552394516622 \\
2.62445306768174    4.26729137931163 \\
2.66622050182715    4.30905881345704 \\
2.70798793597256    4.35082624760245 \\
2.74975537011798    4.39259368174787 \\
2.79152280426339    4.43436111589328 \\
2.8332902384088 4.47612855003869 \\
2.87505767255421    4.5178959841841 \\
};
\addplot [thick, lightgray!20.0!black, dash pattern=on 1pt off 3pt on 3pt off 3pt]
table [row sep=\\]{%
0.828453399429049   2.47129171105894 \\
0.87022083357446    2.49217542813164 \\
0.911988267719872   2.51305914520435 \\
0.953755701865283   2.53394286227706 \\
0.995523136010694   2.55482657934976 \\
1.03729057015611    2.57571029642247 \\
1.07905800430152    2.59659401349517 \\
1.12082543844693    2.61747773056788 \\
1.16259287259234    2.63836144764058 \\
1.20436030673775    2.65924516471329 \\
1.24612774088316    2.680128881786 \\
1.28789517502857    2.7010125988587 \\
1.32966260917399    2.72189631593141 \\
1.3714300433194 2.74278003300411 \\
1.41319747746481    2.76366375007682 \\
1.45496491161022    2.78454746714952 \\
1.49673234575563    2.80543118422223 \\
1.53849977990104    2.82631490129494 \\
1.58026721404645    2.84719861836764 \\
1.62203464819187    2.86808233544035 \\
1.66380208233728    2.88896605251305 \\
1.70556951648269    2.90984976958576 \\
1.7473369506281 2.93073348665846 \\
1.78910438477351    2.95161720373117 \\
1.83087181891892    2.97250092080388 \\
1.87263925306434    2.99338463787658 \\
1.91440668720975    3.01426835494929 \\
1.95617412135516    3.03515207202199 \\
1.99794155550057    3.0560357890947 \\
2.03970898964598    3.07691950616741 \\
2.08147642379139    3.09780322324011 \\
2.1232438579368 3.11868694031282 \\
2.16501129208222    3.13957065738552 \\
2.20677872622763    3.16045437445823 \\
2.24854616037304    3.18133809153093 \\
2.29031359451845    3.20222180860364 \\
2.33208102866386    3.22310552567635 \\
2.37384846280927    3.24398924274905 \\
2.41561589695468    3.26487295982176 \\
2.4573833311001 3.28575667689446 \\
2.49915076524551    3.30664039396717 \\
2.54091819939092    3.32752411103987 \\
2.58268563353633    3.34840782811258 \\
2.62445306768174    3.36929154518528 \\
2.66622050182715    3.39017526225799 \\
2.70798793597256    3.4110589793307 \\
2.74975537011798    3.4319426964034 \\
2.79152280426339    3.45282641347611 \\
2.8332902384088 3.47371013054881 \\
2.87505767255421    3.49459384762152 \\
};
\addplot [semithick, blue, dashed]
table [row sep=\\]{%
2.28855274607818    3.39984671271292 \\
2.21726284543108    3.32837960343874 \\
2.1459371528562 3.25695815256093 \\
2.07459864378146    3.18554215485437 \\
2.00317702399036    3.11427729656159 \\
1.93196894090703    3.04257551244019 \\
1.86055682010212    2.97127584873811 \\
1.78922198507757    2.8998205024271 \\
1.71806400811034    2.82801506422398 \\
1.64669630191378    2.75663610824585 \\
1.57554100869181    2.68484536164441 \\
1.50399990780213    2.61384182187607 \\
1.43247642557505    2.54282542695918 \\
1.36122665838017    2.47129171105894 \\
};
\nextgroupplot[
height=\figheight,
scale only axis,
tick align=outside,
tick pos=left,
title={linear ODE \cref{eq:ex_ODE_linear} \rev{and $R \equiv K_R \cdot h^q$ with $K_R = 500$}},
width=\figwidth,
x grid style={lightgray!92.02614379084967!black},
xlabel near ticks,
xmin=-1.67033221231055, xmax=0.237770701245054,
xtick={-2,-1.5,-1,-0.5,0,0.5},
xticklabels={$10^{2.0}$,$10^{1.5}$,$10^{1.0}$,$10^{0.5}$,$10^{-0.0}$,$10^{-0.5}$},
y grid style={lightgray!92.02614379084967!black},
ylabel near ticks,
ymin=2.42486396097624, ymax=3.44627446279562,
ytick={2.4,2.6,2.8,3,3.2,3.4,3.6},
yticklabels={$10^{2.4}$,$10^{2.6}$,$10^{2.8}$,$10^{3.0}$,$10^{3.2}$,$10^{3.4}$,$10^{3.6}$}
]
\addplot [semithick, blue, mark=*, mark size=2, mark options={solid}, only marks, forget plot]
table [row sep=\\]{%
0.0725669857856943  3.39984671271292 \\
-0.0774058035759142 3.32837960343874 \\
-0.227963050484662  3.25695815256093 \\
-0.378184125547291  3.18554215485437 \\
-0.526594538184956  3.11427729656159 \\
-0.670793120614669  3.04257551244019 \\
-0.80794336794645   2.97127584873811 \\
-0.934026239083755  2.8998205024271 \\
-1.04501251027417   2.82801506422398 \\
-1.14059513927239   2.75663610824585 \\
-1.23055996287161   2.68484536164441 \\
-1.3403197473601    2.61384182187607 \\
-1.47061768107849   2.54282542695918 \\
-1.5794701688079    2.47129171105894 \\
};
\addplot [semithick, blue]
table [row sep=\\]{%
0.0725669857856943  3.39984671271292 \\
-0.0774058035759142 3.32837960343874 \\
-0.227963050484662  3.25695815256093 \\
-0.378184125547291  3.18554215485437 \\
-0.526594538184956  3.11427729656159 \\
-0.670793120614669  3.04257551244019 \\
-0.80794336794645   2.97127584873811 \\
-0.934026239083755  2.8998205024271 \\
-1.04501251027417   2.82801506422398 \\
-1.14059513927239   2.75663610824585 \\
-1.23055996287161   2.68484536164441 \\
-1.3403197473601    2.61384182187607 \\
-1.47061768107849   2.54282542695918 \\
-1.5794701688079    2.47129171105894 \\
};
\addplot [thick, black, dotted]
table [row sep=\\]{%
-1.5794701688079    2.47129171105894 \\
-1.54238362043948   2.50837825942737 \\
-1.50529707207105   2.54546480779579 \\
-1.46821052370262   2.58255135616422 \\
-1.43112397533419   2.61963790453265 \\
-1.39403742696577   2.65672445290108 \\
-1.35695087859734   2.69381100126951 \\
-1.31986433022891   2.73089754963793 \\
-1.28277778186048   2.76798409800636 \\
-1.24569123349205   2.80507064637479 \\
-1.20860468512363   2.84215719474322 \\
-1.1715181367552    2.87924374311164 \\
-1.13443158838677   2.91633029148007 \\
-1.09734504001834   2.9534168398485 \\
-1.06025849164992   2.99050338821693 \\
-1.02317194328149   3.02758993658535 \\
-0.98608539491306   3.06467648495378 \\
-0.948998846544633  3.10176303332221 \\
-0.911912298176205  3.13884958169064 \\
-0.874825749807777  3.17593613005906 \\
-0.837739201439349  3.21302267842749 \\
-0.800652653070922  3.25010922679592 \\
-0.763566104702494  3.28719577516435 \\
-0.726479556334066  3.32428232353278 \\
-0.689393007965638  3.3613688719012 \\
-0.652306459597211  3.39845542026963 \\
-0.615219911228783  3.43554196863806 \\
-0.578133362860356  3.47262851700649 \\
-0.541046814491928  3.50971506537491 \\
-0.5039602661235    3.54680161374334 \\
-0.466873717755072  3.58388816211177 \\
-0.429787169386645  3.6209747104802 \\
-0.392700621018217  3.65806125884863 \\
-0.355614072649789  3.69514780721705 \\
-0.318527524281361  3.73223435558548 \\
-0.281440975912934  3.76932090395391 \\
-0.244354427544506  3.80640745232234 \\
-0.207267879176078  3.84349400069076 \\
-0.170181330807651  3.88058054905919 \\
-0.133094782439223  3.91766709742762 \\
-0.0960082340707953 3.95475364579605 \\
-0.0589216857023676 3.99184019416447 \\
-0.0218351373339398 4.0289267425329 \\
0.0152514110344879  4.06601329090133 \\
0.0523379594029156  4.10309983926976 \\
0.0894245077713434  4.14018638763819 \\
0.126511056139771   4.17727293600661 \\
0.163597604508199   4.21435948437504 \\
0.200684152876627   4.25144603274347 \\
0.237770701245054   4.2885325811119 \\
};
\addplot [thick, lightgray!20.0!black, dash pattern=on 1pt off 3pt on 3pt off 3pt]
table [row sep=\\]{%
-1.5794701688079    2.47129171105894 \\
-1.54238362043948   2.48983498524315 \\
-1.50529707207105   2.50837825942737 \\
-1.46821052370262   2.52692153361158 \\
-1.43112397533419   2.54546480779579 \\
-1.39403742696577   2.56400808198001 \\
-1.35695087859734   2.58255135616422 \\
-1.31986433022891   2.60109463034844 \\
-1.28277778186048   2.61963790453265 \\
-1.24569123349205   2.63818117871686 \\
-1.20860468512363   2.65672445290108 \\
-1.1715181367552    2.67526772708529 \\
-1.13443158838677   2.69381100126951 \\
-1.09734504001834   2.71235427545372 \\
-1.06025849164992   2.73089754963793 \\
-1.02317194328149   2.74944082382215 \\
-0.98608539491306   2.76798409800636 \\
-0.948998846544633  2.78652737219057 \\
-0.911912298176205  2.80507064637479 \\
-0.874825749807777  2.823613920559 \\
-0.837739201439349  2.84215719474322 \\
-0.800652653070922  2.86070046892743 \\
-0.763566104702494  2.87924374311164 \\
-0.726479556334066  2.89778701729586 \\
-0.689393007965638  2.91633029148007 \\
-0.652306459597211  2.93487356566429 \\
-0.615219911228783  2.9534168398485 \\
-0.578133362860356  2.97196011403271 \\
-0.541046814491928  2.99050338821693 \\
-0.5039602661235    3.00904666240114 \\
-0.466873717755072  3.02758993658535 \\
-0.429787169386645  3.04613321076957 \\
-0.392700621018217  3.06467648495378 \\
-0.355614072649789  3.083219759138 \\
-0.318527524281361  3.10176303332221 \\
-0.281440975912934  3.12030630750642 \\
-0.244354427544506  3.13884958169064 \\
-0.207267879176078  3.15739285587485 \\
-0.170181330807651  3.17593613005906 \\
-0.133094782439223  3.19447940424328 \\
-0.0960082340707953 3.21302267842749 \\
-0.0589216857023676 3.23156595261171 \\
-0.0218351373339398 3.25010922679592 \\
0.0152514110344879  3.26865250098013 \\
0.0523379594029156  3.28719577516435 \\
0.0894245077713434  3.30573904934856 \\
0.126511056139771   3.32428232353278 \\
0.163597604508199   3.34282559771699 \\
0.200684152876627   3.3613688719012 \\
0.237770701245054   3.37991214608542 \\
};
\addplot [semithick, blue, dashed]
table [row sep=\\]{%
-0.114735464105809  3.39984671271292 \\
-0.185100723211205  3.32837960343874 \\
-0.255293975604331  3.25695815256093 \\
-0.325259295959975  3.18554215485437 \\
-0.395021596383741  3.11427729656159 \\
-0.464244032330484  3.04257551244019 \\
-0.533263599380265  2.97127584873811 \\
-0.601728754909415  2.8998205024271 \\
-0.669435355761672  2.82801506422398 \\
-0.736639729208211  2.75663610824585 \\
-0.802738466164626  2.68484536164441 \\
-0.86812613780925   2.61384182187607 \\
-0.932074278266458  2.54282542695918 \\
-0.993862748372162  2.47129171105894 \\
};
\nextgroupplot[
height=\figheight,
scale only axis,
tick align=outside,
tick pos=left,
title={logistic ODE \cref{eq:ex_ODE_logistic} \rev{and $R \equiv 0$}},
width=\figwidth,
x grid style={lightgray!92.02614379084967!black},
xlabel near ticks,
xlabel={global error $\Vert m(T) - x(T) \Vert$},
xmin=1.20619119244248, xmax=6.32863067387719,
xtick={0,2,4,6,8},
xticklabels={$10^{-0.0}$,$10^{-2.0}$,$10^{-4.0}$,$10^{-6.0}$,$10^{-8.0}$},
y grid style={lightgray!92.02614379084967!black},
ylabel near ticks,
ylabel={$ \# \mathrm{Evals}\ \mathrm{of}\ f$},
ymin=1.01964051992218, ymax=3.09430283238466,
ytick={1,1.25,1.5,1.75,2,2.25,2.5,2.75,3,3.25},
yticklabels={$10^{1.00}$,$10^{1.25}$,$10^{1.50}$,$10^{1.75}$,$10^{2.00}$,$10^{2.25}$,$10^{2.50}$,$10^{2.75}$,$10^{3.00}$,$10^{3.25}$}
]
\addplot [semithick, blue, mark=*, mark size=2, mark options={solid}, only marks, forget plot]
table [row sep=\\]{%
5.95491199266832	3 \\
5.66861986235897	2.85672889038288 \\
5.38228541157548	2.71349054309394 \\
5.09633976737106	2.5705429398819 \\
4.81118182062079	2.42813479402879 \\
4.52562415325671	2.28555730900777 \\
4.23495800264896	2.13987908640124 \\
3.95343587928154	2 \\
3.65806869331092	1.85125834871908 \\
3.37068419803129	1.70757017609794 \\
3.08952711490053	1.56820172406699 \\
2.79121875105843	1.41497334797082 \\
2.51563862530082	1.27875360095283 \\
2.21772518057958	1.11394335230684 \\
};
\addplot [semithick, blue, forget plot]
table [row sep=\\]{%
5.95491199266832	3 \\
5.66861986235897	2.85672889038288 \\
5.38228541157548	2.71349054309394 \\
5.09633976737106	2.5705429398819 \\
4.81118182062079	2.42813479402879 \\
4.52562415325671	2.28555730900777 \\
4.23495800264896	2.13987908640124 \\
3.95343587928154	2 \\
3.65806869331092	1.85125834871908 \\
3.37068419803129	1.70757017609794 \\
3.08952711490053	1.56820172406699 \\
2.79121875105843	1.41497334797082 \\
2.51563862530082	1.27875360095283 \\
2.21772518057958	1.11394335230684 \\
};
\addplot [thick, black, dotted, forget plot]
table [row sep=\\]{%
2.21772518057958	1.11394335230684 \\
2.30162121105504	1.1978393827823 \\
2.3855172415305	1.28173541325776 \\
2.46941327200596	1.36563144373322 \\
2.55330930248142	1.44952747420868 \\
2.63720533295688	1.53342350468414 \\
2.72110136343234	1.61731953515961 \\
2.80499739390781	1.70121556563507 \\
2.88889342438327	1.78511159611053 \\
2.97278945485873	1.86900762658599 \\
3.05668548533419	1.95290365706145 \\
3.14058151580965	2.03679968753691 \\
3.22447754628511	2.12069571801238 \\
3.30837357676058	2.20459174848784 \\
3.39226960723604	2.2884877789633 \\
3.4761656377115	2.37238380943876 \\
3.56006166818696	2.45627983991422 \\
3.64395769866242	2.54017587038968 \\
3.72785372913788	2.62407190086514 \\
3.81174975961334	2.70796793134061 \\
3.89564579008881	2.79186396181607 \\
3.97954182056427	2.87575999229153 \\
4.06343785103973	2.95965602276699 \\
4.14733388151519	3.04355205324245 \\
4.23122991199065	3.12744808371791 \\
4.31512594246611	3.21134411419338 \\
4.39902197294158	3.29524014466884 \\
4.48291800341704	3.3791361751443 \\
4.5668140338925	3.46303220561976 \\
4.65071006436796	3.54692823609522 \\
4.73460609484342	3.63082426657068 \\
4.81850212531888	3.71472029704615 \\
4.90239815579435	3.79861632752161 \\
4.98629418626981	3.88251235799707 \\
5.07019021674527	3.96640838847253 \\
5.15408624722073	4.05030441894799 \\
5.23798227769619	4.13420044942345 \\
5.32187830817165	4.21809647989892 \\
5.40577433864711	4.30199251037438 \\
5.48967036912258	4.38588854084984 \\
5.57356639959804	4.4697845713253 \\
5.6574624300735	4.55368060180076 \\
5.74135846054896	4.63757663227622 \\
5.82525449102442	4.72147266275168 \\
5.90915052149988	4.80536869322715 \\
5.99304655197534	4.88926472370261 \\
6.07694258245081	4.97316075417807 \\
6.16083861292627	5.05705678465353 \\
6.24473464340173	5.14095281512899 \\
6.32863067387719	5.22484884560445 \\
};
\addplot [thick, lightgray!20.0!black, dash pattern=on 1pt off 3pt on 3pt off 3pt, forget plot]
table [row sep=\\]{%
2.21772518057958	1.11394335230684 \\
2.30162121105504	1.15589136754457 \\
2.3855172415305	1.1978393827823 \\
2.46941327200596	1.23978739802003 \\
2.55330930248142	1.28173541325776 \\
2.63720533295688	1.32368342849549 \\
2.72110136343234	1.36563144373322 \\
2.80499739390781	1.40757945897095 \\
2.88889342438327	1.44952747420868 \\
2.97278945485873	1.49147548944641 \\
3.05668548533419	1.53342350468414 \\
3.14058151580965	1.57537151992188 \\
3.22447754628511	1.61731953515961 \\
3.30837357676058	1.65926755039734 \\
3.39226960723604	1.70121556563507 \\
3.4761656377115	1.7431635808728 \\
3.56006166818696	1.78511159611053 \\
3.64395769866242	1.82705961134826 \\
3.72785372913788	1.86900762658599 \\
3.81174975961334	1.91095564182372 \\
3.89564579008881	1.95290365706145 \\
3.97954182056427	1.99485167229918 \\
4.06343785103973	2.03679968753691 \\
4.14733388151519	2.07874770277464 \\
4.23122991199065	2.12069571801238 \\
4.31512594246611	2.16264373325011 \\
4.39902197294158	2.20459174848784 \\
4.48291800341704	2.24653976372557 \\
4.5668140338925	2.2884877789633 \\
4.65071006436796	2.33043579420103 \\
4.73460609484342	2.37238380943876 \\
4.81850212531888	2.41433182467649 \\
4.90239815579435	2.45627983991422 \\
4.98629418626981	2.49822785515195 \\
5.07019021674527	2.54017587038968 \\
5.15408624722073	2.58212388562741 \\
5.23798227769619	2.62407190086514 \\
5.32187830817165	2.66601991610288 \\
5.40577433864711	2.70796793134061 \\
5.48967036912258	2.74991594657834 \\
5.57356639959804	2.79186396181607 \\
5.6574624300735	2.8338119770538 \\
5.74135846054896	2.87575999229153 \\
5.82525449102442	2.91770800752926 \\
5.90915052149988	2.95965602276699 \\
5.99304655197534	3.00160403800472 \\
6.07694258245081	3.04355205324245 \\
6.16083861292627	3.08550006848018 \\
6.24473464340173	3.12744808371791 \\
6.32863067387719	3.16939609895565 \\
};
\addplot [semithick, blue, dashed, forget plot]
table [row sep=\\]{%
3.27336735570313	3 \\
3.13144532844971	2.85672889038288 \\
2.98936114000904	2.71349054309394 \\
2.84704818078719	2.5705429398819 \\
2.70443529933789	2.42813479402879 \\
2.5619284847727	2.28555730900777 \\
2.4210577321465	2.13987908640124 \\
2.27740536347172	2 \\
2.13845253329654	1.85125834871908 \\
1.99727763984357	1.70757017609794 \\
1.85435960174656	1.56820172406699 \\
1.71927900678798	1.41497334797082 \\
1.57634478099621	1.27875360095283 \\
1.45011688203461	1.11394335230684 \\
};
\nextgroupplot[
legend cell align={left},
legend entries={{q=1},{$h^1$ conv.},{$h^2$ conv.},{std. dev.}},
legend style={at={(0.97,0.03)}, anchor=south east, draw=white!80.0!black},
height=\figheight,
scale only axis,
tick align=outside,
tick pos=left,
title={logistic ODE \cref{eq:ex_ODE_logistic} \rev{and $R \equiv K_R \cdot h^q$ with $K_R = 500$}},
width=\figwidth,
x grid style={lightgray!92.02614379084967!black},
xlabel near ticks,
xlabel={global error $\Vert m(T) - x(T) \Vert$},
xmin=-0.26441005544885, xmax=2.45675233660335,
xtick={-1,0,1,2,3},
xticklabels={$10^{1.0}$,$10^{-0.0}$,$10^{-1.0}$,$10^{-2.0}$,$10^{-3.0}$},
y grid style={lightgray!92.02614379084967!black},
ylabel near ticks,
ymin=1.01964051992218, ymax=3.09430283238466,
ytick={1,1.25,1.5,1.75,2,2.25,2.5,2.75,3,3.25},
yticklabels={$10^{1.00}$,$10^{1.25}$,$10^{1.50}$,$10^{1.75}$,$10^{2.00}$,$10^{2.25}$,$10^{2.50}$,$10^{2.75}$,$10^{3.00}$,$10^{3.25}$}
]
\addlegendimage{no markers, blue}
\addlegendimage{no markers, black, dotted, thick}
\addlegendimage{no markers, lightgray!20.0!black, dash pattern=on 1pt off 3pt on 3pt off 3pt, thick}
\addlegendimage{no markers, blue, dashed}
\addplot [semithick, blue, mark=*, mark size=2, mark options={solid}, only marks, forget plot]
table [row sep=\\]{%
2.27063306016778	3 \\
2.01772239382359	2.85672889038288 \\
1.77648692650382	2.71349054309394 \\
1.55027719647494	2.5705429398819 \\
1.33493138897703	2.42813479402879 \\
1.11673195585417	2.28555730900777 \\
0.890454765539352	2.13987908640124 \\
0.711694261017527	2 \\
0.568395598827321	1.85125834871908 \\
0.487551632053815	1.70757017609794 \\
0.443957210892952	1.56820172406699 \\
0.418677549584382	1.41497334797082 \\
0.407383538625117	1.27875360095283 \\
0.409440295812004	1.11394335230684 \\
};
\addplot [semithick, blue, forget plot]
table [row sep=\\]{%
2.27063306016778	3 \\
2.01772239382359	2.85672889038288 \\
1.77648692650382	2.71349054309394 \\
1.55027719647494	2.5705429398819 \\
1.33493138897703	2.42813479402879 \\
1.11673195585417	2.28555730900777 \\
0.890454765539352	2.13987908640124 \\
0.711694261017527	2 \\
0.568395598827321	1.85125834871908 \\
0.487551632053815	1.70757017609794 \\
0.443957210892952	1.56820172406699 \\
0.418677549584382	1.41497334797082 \\
0.407383538625117	1.27875360095283 \\
0.409440295812004	1.11394335230684 \\
};
\addplot [thick, black, dotted, forget plot]
table [row sep=\\]{%
0.409440295812004	1.11394335230684 \\
0.451268346295615	1.15577140279045 \\
0.493096396779226	1.19759945327406 \\
0.534924447262836	1.23942750375767 \\
0.576752497746447	1.28125555424128 \\
0.618580548230058	1.32308360472489 \\
0.660408598713668	1.3649116552085 \\
0.702236649197279	1.40673970569211 \\
0.74406469968089	1.44856775617572 \\
0.785892750164501	1.49039580665933 \\
0.827720800648111	1.53222385714294 \\
0.869548851131722	1.57405190762655 \\
0.911376901615333	1.61587995811017 \\
0.953204952098943	1.65770800859378 \\
0.995033002582554	1.69953605907739 \\
1.03686105306616	1.741364109561 \\
1.07868910354978	1.78319216004461 \\
1.12051715403339	1.82502021052822 \\
1.162345204517	1.86684826101183 \\
1.20417325500061	1.90867631149544 \\
1.24600130548422	1.95050436197905 \\
1.28782935596783	1.99233241246266 \\
1.32965740645144	2.03416046294627 \\
1.37148545693505	2.07598851342988 \\
1.41331350741866	2.11781656391349 \\
1.45514155790227	2.1596446143971 \\
1.49696960838588	2.20147266488072 \\
1.53879765886949	2.24330071536433 \\
1.5806257093531	2.28512876584794 \\
1.62245375983671	2.32695681633155 \\
1.66428181032033	2.36878486681516 \\
1.70610986080394	2.41061291729877 \\
1.74793791128755	2.45244096778238 \\
1.78976596177116	2.49426901826599 \\
1.83159401225477	2.5360970687496 \\
1.87342206273838	2.57792511923321 \\
1.91525011322199	2.61975316971682 \\
1.9570781637056	2.66158122020043 \\
1.99890621418921	2.70340927068404 \\
2.04073426467282	2.74523732116765 \\
2.08256231515643	2.78706537165127 \\
2.12439036564004	2.82889342213488 \\
2.16621841612365	2.87072147261849 \\
2.20804646660726	2.9125495231021 \\
2.24987451709088	2.95437757358571 \\
2.29170256757449	2.99620562406932 \\
2.3335306180581	3.03803367455293 \\
2.37535866854171	3.07986172503654 \\
2.41718671902532	3.12168977552015 \\
2.45901476950893	3.16351782600376 \\
};
\addplot [thick, lightgray!20.0!black, dash pattern=on 1pt off 3pt on 3pt off 3pt, forget plot]
table [row sep=\\]{%
0.409440295812004	1.11394335230684 \\
0.451268346295615	1.13485737754864 \\
0.493096396779226	1.15577140279045 \\
0.534924447262836	1.17668542803225 \\
0.576752497746447	1.19759945327406 \\
0.618580548230058	1.21851347851586 \\
0.660408598713668	1.23942750375767 \\
0.702236649197279	1.26034152899947 \\
0.74406469968089	1.28125555424128 \\
0.785892750164501	1.30216957948308 \\
0.827720800648111	1.32308360472489 \\
0.869548851131722	1.3439976299667 \\
0.911376901615333	1.3649116552085 \\
0.953204952098943	1.38582568045031 \\
0.995033002582554	1.40673970569211 \\
1.03686105306616	1.42765373093392 \\
1.07868910354978	1.44856775617572 \\
1.12051715403339	1.46948178141753 \\
1.162345204517	1.49039580665933 \\
1.20417325500061	1.51130983190114 \\
1.24600130548422	1.53222385714294 \\
1.28782935596783	1.55313788238475 \\
1.32965740645144	1.57405190762655 \\
1.37148545693505	1.59496593286836 \\
1.41331350741866	1.61587995811017 \\
1.45514155790227	1.63679398335197 \\
1.49696960838588	1.65770800859378 \\
1.53879765886949	1.67862203383558 \\
1.5806257093531	1.69953605907739 \\
1.62245375983671	1.72045008431919 \\
1.66428181032033	1.741364109561 \\
1.70610986080394	1.7622781348028 \\
1.74793791128755	1.78319216004461 \\
1.78976596177116	1.80410618528641 \\
1.83159401225477	1.82502021052822 \\
1.87342206273838	1.84593423577002 \\
1.91525011322199	1.86684826101183 \\
1.9570781637056	1.88776228625363 \\
1.99890621418921	1.90867631149544 \\
2.04073426467282	1.92959033673725 \\
2.08256231515643	1.95050436197905 \\
2.12439036564004	1.97141838722086 \\
2.16621841612365	1.99233241246266 \\
2.20804646660726	2.01324643770447 \\
2.24987451709088	2.03416046294627 \\
2.29170256757449	2.05507448818808 \\
2.3335306180581	2.07598851342988 \\
2.37535866854171	2.09690253867169 \\
2.41718671902532	2.11781656391349 \\
2.45901476950893	2.1387305891553 \\
};
\addplot [semithick, blue, dashed, forget plot]
table [row sep=\\]{%
0.785707875809247	3 \\
0.646866216713073	2.85672889038288 \\
0.509761254267441	2.71349054309394 \\
0.375731575684263	2.5705429398819 \\
0.247868201009032	2.42813479402879 \\
0.131733700565944	2.28555730900777 \\
0.0355833064409561	2.13987908640124 \\
-0.0413955461552881	2 \\
-0.0868157199395567	1.85125834871908 \\
-0.115772586144979	1.70757017609794 \\
-0.134723159022289	1.56820172406699 \\
-0.12665856366403	1.41497334797082 \\
-0.132159320186257	1.27875360095283 \\
-0.094306949316359	1.11394335230684 \\
};
\end{groupplot}

\end{tikzpicture}

%% file: fig/exp33_AppendixB_qmax=4_final.tex
\begin{tikzpicture}

\definecolor{color0}{rgb}{0,0.75,0.75}

\begin{axis}[
height=\figheight,
legend cell align={left},
legend entries={{q=1},{q=2},{q=3},{q=4},{$h^1$ conv.},{$h^2$ conv.},{$h^3$ conv.},{$h^4$ conv.},{$h^5$ conv.}},
legend style={at={(0.97,0.03)}, anchor=south east, draw=lightgray!80.0!black, nodes={scale=0.5, transform shape}},
scale only axis,
tick align=outside,
tick pos=left,
width=\figwidth,
x grid style={lightgray!92.02614379084967!black},
xlabel near ticks,
xlabel={state misalignment $\delta^{(1)}(T)$},
xmin=1.95529285001413, xmax=17.7385027871045,
xtick={0,4,8,12,16},
xticklabels={$10^{-0.0}$,$10^{-4.0}$,$10^{-8.0}$,$10^{-12.0}$,$10^{-16.0}$},
y grid style={lightgray!92.02614379084967!black},
ylabel near ticks,
ylabel={$ \# \mathrm{Evals}\ \mathrm{of}\ f$},
ymin=1.41541033223476, ymax=3.38968254610791,
ytick={1.25,1.5,1.75,2,2.25,2.5,2.75,3,3.25,3.5},
yticklabels={$10^{1.25}$,$10^{1.50}$,$10^{1.75}$,$10^{2.00}$,$10^{2.25}$,$10^{2.50}$,$10^{2.75}$,$10^{3.00}$,$10^{3.25}$,$10^{3.50}$}
]
\addlegendimage{no markers, blue}
\addlegendimage{no markers, green!50.0!black}
\addlegendimage{no markers, red}
\addlegendimage{no markers, color0}
\addlegendimage{no markers, black, dotted, thick}
\addlegendimage{no markers, lightgray!20.0!black, dash pattern=on 1pt off 3pt on 3pt off 3pt, thick}
\addlegendimage{no markers, lightgray!40.0!black, dotted, thick}
\addlegendimage{no markers, lightgray!60.0!black, dash pattern=on 1pt off 3pt on 3pt off 3pt, thick}
\addlegendimage{no markers, lightgray!80.0!black, dotted, thick}
\addplot [semithick, blue, mark=*, mark size=2, mark options={solid}, only marks, forget plot]
table [row sep=\\]{%
8.49362761524062	3.29994290002277 \\
8.23542540150779	3.17114115102838 \\
7.97786158666444	3.04257551244019 \\
7.71952005976252	2.91381385238372 \\
7.46185573973777	2.78532983501077 \\
7.20458053512844	2.6570558528571 \\
6.94324561022929	2.52762990087134 \\
6.68608003173555	2.39967372148104 \\
6.42114065440746	2.26951294421792 \\
6.15687429441671	2.13987908640124 \\
5.91305821142192	2.01703333929878 \\
5.62415955273147	1.88081359228079 \\
5.36888126355697	1.75587485567249 \\
5.0844381384853	1.6232492903979 \\
4.84190135833002	1.50514997831991 \\
};
\addplot [semithick, blue]
table [row sep=\\]{%
8.49362761524062	3.29994290002277 \\
8.23542540150779	3.17114115102838 \\
7.97786158666444	3.04257551244019 \\
7.71952005976252	2.91381385238372 \\
7.46185573973777	2.78532983501077 \\
7.20458053512844	2.6570558528571 \\
6.94324561022929	2.52762990087134 \\
6.68608003173555	2.39967372148104 \\
6.42114065440746	2.26951294421792 \\
6.15687429441671	2.13987908640124 \\
5.91305821142192	2.01703333929878 \\
5.62415955273147	1.88081359228079 \\
5.36888126355697	1.75587485567249 \\
5.0844381384853	1.6232492903979 \\
4.84190135833002	1.50514997831991 \\
};
\addplot [semithick, green!50.0!black, mark=*, mark size=2, mark options={solid}, only marks, forget plot]
table [row sep=\\]{%
11.6427372772298	3.29994290002277 \\
11.2555842094498	3.17114115102838 \\
10.8692271575092	3.04257551244019 \\
10.4818367480849	2.91381385238372 \\
10.0952743448056	2.78532983501077 \\
9.70915799648641	2.6570558528571 \\
9.31775125428601	2.52762990087134 \\
8.93162785824038	2.39967372148104 \\
8.53541154554975	2.26951294421792 \\
8.13997811733932	2.13987908640124 \\
7.77085126677999	2.01703333929878 \\
7.34376180850464	1.88081359228079 \\
6.96044181059511	1.75587485567249 \\
6.54078227495138	1.6232492903979 \\
6.17810769264019	1.50514997831991 \\
};
\addplot [semithick, green!50.0!black]
table [row sep=\\]{%
11.6427372772298	3.29994290002277 \\
11.2555842094498	3.17114115102838 \\
10.8692271575092	3.04257551244019 \\
10.4818367480849	2.91381385238372 \\
10.0952743448056	2.78532983501077 \\
9.70915799648641	2.6570558528571 \\
9.31775125428601	2.52762990087134 \\
8.93162785824038	2.39967372148104 \\
8.53541154554975	2.26951294421792 \\
8.13997811733932	2.13987908640124 \\
7.77085126677999	2.01703333929878 \\
7.34376180850464	1.88081359228079 \\
6.96044181059511	1.75587485567249 \\
6.54078227495138	1.6232492903979 \\
6.17810769264019	1.50514997831991 \\
};
\addplot [semithick, red, mark=*, mark size=2, mark options={solid}, only marks, forget plot]
table [row sep=\\]{%
14.5547538674115	3.29994290002277 \\
14.0383826499978	3.17114115102838 \\
13.5230752355108	3.04257551244019 \\
13.0064672870647	2.91381385238372 \\
12.4908320529138	2.78532983501077 \\
11.9756435769487	2.6570558528571 \\
11.4538388706021	2.52762990087134 \\
10.9382996604675	2.39967372148104 \\
10.4101368284726	2.26951294421792 \\
9.88254538137122	2.13987908640124 \\
9.38665061871678	2.01703333929878 \\
8.81885591349522	1.88081359228079 \\
8.30342120986111	1.75587485567249 \\
7.73821401927899	1.6232492903979 \\
6.59538156817927	1.50514997831991 \\
};
\addplot [semithick, red]
table [row sep=\\]{%
14.5547538674115	3.29994290002277 \\
14.0383826499978	3.17114115102838 \\
13.5230752355108	3.04257551244019 \\
13.0064672870647	2.91381385238372 \\
12.4908320529138	2.78532983501077 \\
11.9756435769487	2.6570558528571 \\
11.4538388706021	2.52762990087134 \\
10.9382996604675	2.39967372148104 \\
10.4101368284726	2.26951294421792 \\
9.88254538137122	2.13987908640124 \\
9.38665061871678	2.01703333929878 \\
8.81885591349522	1.88081359228079 \\
8.30342120986111	1.75587485567249 \\
7.73821401927899	1.6232492903979 \\
6.59538156817927	1.50514997831991 \\
};
\addplot [semithick, color0, mark=*, mark size=2, mark options={solid}, only marks, forget plot]
table [row sep=\\]{%
17.1587097528469	3.29994290002277 \\
16.6468263918681	3.17114115102838 \\
16.0364938745741	3.04257551244019 \\
15.4028348971744	2.91381385238372 \\
14.758171763455	2.78532983501077 \\
14.1138761200071	2.6570558528571 \\
13.4615458208083	2.52762990087134 \\
12.8162821368882	2.39967372148104 \\
12.1557266468901	2.26951294421792 \\
11.4954343806957	2.13987908640124 \\
10.8721238449096	2.01703333929878 \\
10.172510814184	1.88081359228079 \\
8.17046250495408	1.75587485567249 \\
4.87117404792744	1.6232492903979 \\
2.84263977768074	1.50514997831991 \\
};
\addplot [semithick, color0]
table [row sep=\\]{%
17.1587097528469	3.29994290002277 \\
16.6468263918681	3.17114115102838 \\
16.0364938745741	3.04257551244019 \\
15.4028348971744	2.91381385238372 \\
14.758171763455	2.78532983501077 \\
14.1138761200071	2.6570558528571 \\
13.4615458208083	2.52762990087134 \\
12.8162821368882	2.39967372148104 \\
12.1557266468901	2.26951294421792 \\
11.4954343806957	2.13987908640124 \\
10.8721238449096	2.01703333929878 \\
10.172510814184	1.88081359228079 \\
8.17046250495408	1.75587485567249 \\
4.87117404792744	1.6232492903979 \\
2.84263977768074	1.50514997831991 \\
};
\addplot [thick, black, dotted]
table [row sep=\\]{%
4.84190135833002	1.50514997831991 \\
5.16328252103783	1.82653114102772 \\
5.48466368374565	2.14791230373553 \\
5.80604484645346	2.46929346644334 \\
6.12742600916127	2.79067462915116 \\
6.44880717186908	3.11205579185897 \\
6.7701883345769	3.43343695456678 \\
7.09156949728471	3.75481811727459 \\
7.41295065999252	4.07619927998241 \\
7.73433182270033	4.39758044269022 \\
8.05571298540815	4.71896160539803 \\
8.37709414811596	5.04034276810584 \\
8.69847531082377	5.36172393081366 \\
9.01985647353158	5.68310509352147 \\
9.34123763623939	6.00448625622928 \\
9.66261879894721	6.32586741893709 \\
9.98399996165502	6.6472485816449 \\
10.3053811243628	6.96862974435272 \\
10.6267622870706	7.29001090706053 \\
10.9481434497785	7.61139206976834 \\
11.2695246124863	7.93277323247615 \\
11.5909057751941	8.25415439518397 \\
11.9122869379019	8.57553555789178 \\
12.2336681006097	8.89691672059959 \\
12.5550492633175	9.2182978833074 \\
12.8764304260253	9.53967904601522 \\
13.1978115887331	9.86106020872303 \\
13.519192751441	10.1824413714308 \\
13.8405739141488	10.5038225341387 \\
14.1619550768566	10.8252036968465 \\
14.4833362395644	11.1465848595543 \\
14.8047174022722	11.4679660222621 \\
15.12609856498	11.7893471849699 \\
15.4474797276878	12.1107283476777 \\
15.7688608903956	12.4321095103855 \\
16.0902420531035	12.7534906730933 \\
16.4116232158113	13.0748718358012 \\
16.7330043785191	13.396252998509 \\
17.0543855412269	13.7176341612168 \\
17.3757667039347	14.0390153239246 \\
17.6971478666425	14.3603964866324 \\
18.0185290293503	14.6817776493402 \\
18.3399101920581	15.003158812048 \\
18.661291354766	15.3245399747558 \\
18.9826725174738	15.6459211374637 \\
19.3040536801816	15.9673023001715 \\
19.6254348428894	16.2886834628793 \\
19.9468160055972	16.6100646255871 \\
20.268197168305	16.9314457882949 \\
20.5895783310128	17.2528269510027 \\
};
\addplot [thick, lightgray!20.0!black, dash pattern=on 1pt off 3pt on 3pt off 3pt]
table [row sep=\\]{%
4.84190135833002	1.50514997831991 \\
5.16328252103783	1.66584055967381 \\
5.48466368374565	1.82653114102772 \\
5.80604484645346	1.98722172238162 \\
6.12742600916127	2.14791230373553 \\
6.44880717186908	2.30860288508944 \\
6.7701883345769	2.46929346644334 \\
7.09156949728471	2.62998404779725 \\
7.41295065999252	2.79067462915116 \\
7.73433182270033	2.95136521050506 \\
8.05571298540815	3.11205579185897 \\
8.37709414811596	3.27274637321287 \\
8.69847531082377	3.43343695456678 \\
9.01985647353158	3.59412753592069 \\
9.34123763623939	3.75481811727459 \\
9.66261879894721	3.9155086986285 \\
9.98399996165502	4.07619927998241 \\
10.3053811243628	4.23688986133631 \\
10.6267622870706	4.39758044269022 \\
10.9481434497785	4.55827102404412 \\
11.2695246124863	4.71896160539803 \\
11.5909057751941	4.87965218675194 \\
11.9122869379019	5.04034276810584 \\
12.2336681006097	5.20103334945975 \\
12.5550492633175	5.36172393081366 \\
12.8764304260253	5.52241451216756 \\
13.1978115887331	5.68310509352147 \\
13.519192751441	5.84379567487537 \\
13.8405739141488	6.00448625622928 \\
14.1619550768566	6.16517683758319 \\
14.4833362395644	6.32586741893709 \\
14.8047174022722	6.486558000291 \\
15.12609856498	6.6472485816449 \\
15.4474797276878	6.80793916299881 \\
15.7688608903956	6.96862974435272 \\
16.0902420531035	7.12932032570662 \\
16.4116232158113	7.29001090706053 \\
16.7330043785191	7.45070148841444 \\
17.0543855412269	7.61139206976834 \\
17.3757667039347	7.77208265112225 \\
17.6971478666425	7.93277323247615 \\
18.0185290293503	8.09346381383006 \\
18.3399101920581	8.25415439518397 \\
18.661291354766	8.41484497653787 \\
18.9826725174738	8.57553555789178 \\
19.3040536801816	8.73622613924569 \\
19.6254348428894	8.89691672059959 \\
19.9468160055972	9.0576073019535 \\
20.268197168305	9.2182978833074 \\
20.5895783310128	9.37898846466131 \\
};
\addplot [thick, lightgray!20.0!black, dash pattern=on 1pt off 3pt on 3pt off 3pt, forget plot]
table [row sep=\\]{%
6.17810769264019	1.50514997831991 \\
6.50021401788694	1.66620314094328 \\
6.82232034313368	1.82725630356665 \\
7.14442666838042	1.98830946619002 \\
7.46653299362716	2.14936262881339 \\
7.7886393188739	2.31041579143676 \\
8.11074564412064	2.47146895406013 \\
8.43285196936738	2.6325221166835 \\
8.75495829461413	2.79357527930687 \\
9.07706461986087	2.95462844193024 \\
9.39917094510761	3.11568160455361 \\
9.72127727035435	3.27673476717698 \\
10.0433835956011	3.43778792980035 \\
10.3654899208478	3.59884109242372 \\
10.6875962460946	3.7598942550471 \\
11.0097025713413	3.92094741767047 \\
11.3318088965881	4.08200058029384 \\
11.6539152218348	4.24305374291721 \\
11.9760215470815	4.40410690554058 \\
12.2981278723283	4.56516006816395 \\
12.620234197575	4.72621323078732 \\
12.9423405228218	4.88726639341069 \\
13.2644468480685	5.04831955603406 \\
13.5865531733152	5.20937271865743 \\
13.908659498562	5.3704258812808 \\
14.2307658238087	5.53147904390417 \\
14.5528721490555	5.69253220652754 \\
14.8749784743022	5.85358536915091 \\
15.197084799549	6.01463853177428 \\
15.5191911247957	6.17569169439765 \\
15.8412974500424	6.33674485702103 \\
16.1634037752892	6.4977980196444 \\
16.4855101005359	6.65885118226777 \\
16.8076164257827	6.81990434489114 \\
17.1297227510294	6.98095750751451 \\
17.4518290762761	7.14201067013788 \\
17.7739354015229	7.30306383276125 \\
18.0960417267696	7.46411699538462 \\
18.4181480520164	7.62517015800799 \\
18.7402543772631	7.78622332063136 \\
19.0623607025098	7.94727648325473 \\
19.3844670277566	8.1083296458781 \\
19.7065733530033	8.26938280850147 \\
20.0286796782501	8.43043597112484 \\
20.3507860034968	8.59148913374821 \\
20.6728923287436	8.75254229637159 \\
20.9949986539903	8.91359545899495 \\
21.317104979237	9.07464862161833 \\
21.6392113044838	9.2357017842417 \\
21.9613176297305	9.39675494686507 \\
};
\addplot [thick, lightgray!40.0!black, dotted]
table [row sep=\\]{%
6.17810769264019	1.50514997831991 \\
6.50021401788694	1.61251875340215 \\
6.82232034313368	1.7198875284844 \\
7.14442666838042	1.82725630356665 \\
7.46653299362716	1.93462507864889 \\
7.7886393188739	2.04199385373114 \\
8.11074564412064	2.14936262881339 \\
8.43285196936738	2.25673140389564 \\
8.75495829461413	2.36410017897788 \\
9.07706461986087	2.47146895406013 \\
9.39917094510761	2.57883772914238 \\
9.72127727035435	2.68620650422462 \\
10.0433835956011	2.79357527930687 \\
10.3654899208478	2.90094405438912 \\
10.6875962460946	3.00831282947137 \\
11.0097025713413	3.11568160455361 \\
11.3318088965881	3.22305037963586 \\
11.6539152218348	3.33041915471811 \\
11.9760215470815	3.43778792980035 \\
12.2981278723283	3.5451567048826 \\
12.620234197575	3.65252547996485 \\
12.9423405228218	3.7598942550471 \\
13.2644468480685	3.86726303012934 \\
13.5865531733152	3.97463180521159 \\
13.908659498562	4.08200058029384 \\
14.2307658238087	4.18936935537608 \\
14.5528721490555	4.29673813045833 \\
14.8749784743022	4.40410690554058 \\
15.197084799549	4.51147568062282 \\
15.5191911247957	4.61884445570507 \\
15.8412974500424	4.72621323078732 \\
16.1634037752892	4.83358200586957 \\
16.4855101005359	4.94095078095181 \\
16.8076164257827	5.04831955603406 \\
17.1297227510294	5.15568833111631 \\
17.4518290762761	5.26305710619855 \\
17.7739354015229	5.3704258812808 \\
18.0960417267696	5.47779465636305 \\
18.4181480520164	5.5851634314453 \\
18.7402543772631	5.69253220652754 \\
19.0623607025098	5.79990098160979 \\
19.3844670277566	5.90726975669204 \\
19.7065733530033	6.01463853177428 \\
20.0286796782501	6.12200730685653 \\
20.3507860034968	6.22937608193878 \\
20.6728923287436	6.33674485702103 \\
20.9949986539903	6.44411363210327 \\
21.317104979237	6.55148240718552 \\
21.6392113044838	6.65885118226777 \\
21.9613176297305	6.76621995735001 \\
};
\addplot [thick, lightgray!40.0!black, dotted, forget plot]
table [row sep=\\]{%
7.738214019279	1.6232492903979 \\
8.06032034452574	1.73061806548015 \\
8.38242666977248	1.83798684056239 \\
8.70453299501922	1.94535561564464 \\
9.02663932026596	2.05272439072689 \\
9.3487456455127	2.16009316580914 \\
9.67085197075944	2.26746194089138 \\
9.99295829600618	2.37483071597363 \\
10.3150646212529	2.48219949105588 \\
10.6371709464997	2.58956826613812 \\
10.9592772717464	2.69693704122037 \\
11.2813835969931	2.80430581630262 \\
11.6034899222399	2.91167459138487 \\
11.9255962474866	3.01904336646711 \\
12.2477025727334	3.12641214154936 \\
12.5698088979801	3.23378091663161 \\
12.8919152232269	3.34114969171385 \\
13.2140215484736	3.4485184667961 \\
13.5361278737203	3.55588724187835 \\
13.8582341989671	3.6632560169606 \\
14.1803405242138	3.77062479204284 \\
14.5024468494606	3.87799356712509 \\
14.8245531747073	3.98536234220734 \\
15.146659499954	4.09273111728958 \\
15.4687658252008	4.20009989237183 \\
15.7908721504475	4.30746866745408 \\
16.1129784756943	4.41483744253632 \\
16.435084800941	4.52220621761857 \\
16.7571911261877	4.62957499270082 \\
17.0792974514345	4.73694376778307 \\
17.4014037766812	4.84431254286531 \\
17.723510101928	4.95168131794756 \\
18.0456164271747	5.05905009302981 \\
18.3677227524215	5.16641886811205 \\
18.6898290776682	5.2737876431943 \\
19.0119354029149	5.38115641827655 \\
19.3340417281617	5.4885251933588 \\
19.6561480534084	5.59589396844104 \\
19.9782543786552	5.70326274352329 \\
20.3003607039019	5.81063151860554 \\
20.6224670291486	5.91800029368778 \\
20.9445733543954	6.02536906877003 \\
21.2666796796421	6.13273784385228 \\
21.5887860048889	6.24010661893453 \\
21.9108923301356	6.34747539401677 \\
22.2329986553824	6.45484416909902 \\
22.5551049806291	6.56221294418127 \\
22.8772113058758	6.66958171926351 \\
23.1993176311226	6.77695049434576 \\
23.5214239563693	6.88431926942801 \\
};
\addplot [thick, lightgray!60.0!black, dash pattern=on 1pt off 3pt on 3pt off 3pt]
table [row sep=\\]{%
7.738214019279	1.6232492903979 \\
8.06032034452574	1.70377587170959 \\
8.38242666977248	1.78430245302127 \\
8.70453299501922	1.86482903433296 \\
9.02663932026596	1.94535561564464 \\
9.3487456455127	2.02588219695633 \\
9.67085197075944	2.10640877826801 \\
9.99295829600618	2.1869353595797 \\
10.3150646212529	2.26746194089138 \\
10.6371709464997	2.34798852220307 \\
10.9592772717464	2.42851510351475 \\
11.2813835969931	2.50904168482644 \\
11.6034899222399	2.58956826613812 \\
11.9255962474866	2.67009484744981 \\
12.2477025727334	2.7506214287615 \\
12.5698088979801	2.83114801007318 \\
12.8919152232269	2.91167459138487 \\
13.2140215484736	2.99220117269655 \\
13.5361278737203	3.07272775400824 \\
13.8582341989671	3.15325433531992 \\
14.1803405242138	3.23378091663161 \\
14.5024468494606	3.31430749794329 \\
14.8245531747073	3.39483407925498 \\
15.146659499954	3.47536066056666 \\
15.4687658252008	3.55588724187835 \\
15.7908721504475	3.63641382319003 \\
16.1129784756943	3.71694040450172 \\
16.435084800941	3.7974669858134 \\
16.7571911261877	3.87799356712509 \\
17.0792974514345	3.95852014843677 \\
17.4014037766812	4.03904672974846 \\
17.723510101928	4.11957331106015 \\
18.0456164271747	4.20009989237183 \\
18.3677227524215	4.28062647368352 \\
18.6898290776682	4.3611530549952 \\
19.0119354029149	4.44167963630689 \\
19.3340417281617	4.52220621761857 \\
19.6561480534084	4.60273279893026 \\
19.9782543786552	4.68325938024194 \\
20.3003607039019	4.76378596155363 \\
20.6224670291486	4.84431254286531 \\
20.9445733543954	4.924839124177 \\
21.2666796796421	5.00536570548868 \\
21.5887860048889	5.08589228680037 \\
21.9108923301356	5.16641886811205 \\
22.2329986553824	5.24694544942374 \\
22.5551049806291	5.32747203073543 \\
22.8772113058758	5.40799861204711 \\
23.1993176311226	5.4885251933588 \\
23.5214239563693	5.56905177467048 \\
};
\addplot [thick, lightgray!60.0!black, dash pattern=on 1pt off 3pt on 3pt off 3pt, forget plot]
table [row sep=\\]{%
10.172510814184	1.88081359228079 \\
10.4946171394307	1.96134017359248 \\
10.8167234646775	2.04186675490416 \\
11.1388297899242	2.12239333621585 \\
11.4609361151709	2.20291991752753 \\
11.7830424404177	2.28344649883922 \\
12.1051487656644	2.3639730801509 \\
12.4272550909112	2.44449966146259 \\
12.7493614161579	2.52502624277427 \\
13.0714677414046	2.60555282408596 \\
13.3935740666514	2.68607940539764 \\
13.7156803918981	2.76660598670933 \\
14.0377867171449	2.84713256802102 \\
14.3598930423916	2.9276591493327 \\
14.6819993676384	3.00818573064439 \\
15.0041056928851	3.08871231195607 \\
15.3262120181318	3.16923889326776 \\
15.6483183433786	3.24976547457944 \\
15.9704246686253	3.33029205589113 \\
16.2925309938721	3.41081863720281 \\
16.6146373191188	3.4913452185145 \\
16.9367436443655	3.57187179982618 \\
17.2588499696123	3.65239838113787 \\
17.580956294859	3.73292496244955 \\
17.9030626201058	3.81345154376124 \\
18.2251689453525	3.89397812507292 \\
18.5472752705992	3.97450470638461 \\
18.869381595846	4.05503128769629 \\
19.1914879210927	4.13555786900798 \\
19.5135942463395	4.21608445031967 \\
19.8357005715862	4.29661103163135 \\
20.1578068968329	4.37713761294304 \\
20.4799132220797	4.45766419425472 \\
20.8020195473264	4.53819077556641 \\
21.1241258725732	4.61871735687809 \\
21.4462321978199	4.69924393818978 \\
21.7683385230667	4.77977051950146 \\
22.0904448483134	4.86029710081315 \\
22.4125511735601	4.94082368212483 \\
22.7346574988069	5.02135026343652 \\
23.0567638240536	5.1018768447482 \\
23.3788701493004	5.18240342605989 \\
23.7009764745471	5.26293000737157 \\
24.0230827997938	5.34345658868326 \\
24.3451891250406	5.42398316999495 \\
24.6672954502873	5.50450975130663 \\
24.9894017755341	5.58503633261832 \\
25.3115081007808	5.66556291393 \\
25.6336144260276	5.74608949524169 \\
25.9557207512743	5.82661607655337 \\
};
\addplot [thick, lightgray!80.0!black, dotted]
table [row sep=\\]{%
10.172510814184	1.88081359228079 \\
10.4946171394307	1.94523485733014 \\
10.8167234646775	2.00965612237949 \\
11.1388297899242	2.07407738742884 \\
11.4609361151709	2.13849865247818 \\
11.7830424404177	2.20291991752753 \\
12.1051487656644	2.26734118257688 \\
12.4272550909112	2.33176244762623 \\
12.7493614161579	2.39618371267558 \\
13.0714677414046	2.46060497772493 \\
13.3935740666514	2.52502624277427 \\
13.7156803918981	2.58944750782362 \\
14.0377867171449	2.65386877287297 \\
14.3598930423916	2.71829003792232 \\
14.6819993676384	2.78271130297167 \\
15.0041056928851	2.84713256802102 \\
15.3262120181318	2.91155383307036 \\
15.6483183433786	2.97597509811971 \\
15.9704246686253	3.04039636316906 \\
16.2925309938721	3.10481762821841 \\
16.6146373191188	3.16923889326776 \\
16.9367436443655	3.2336601583171 \\
17.2588499696123	3.29808142336645 \\
17.580956294859	3.3625026884158 \\
17.9030626201058	3.42692395346515 \\
18.2251689453525	3.4913452185145 \\
18.5472752705992	3.55576648356385 \\
18.869381595846	3.62018774861319 \\
19.1914879210927	3.68460901366254 \\
19.5135942463395	3.74903027871189 \\
19.8357005715862	3.81345154376124 \\
20.1578068968329	3.87787280881059 \\
20.4799132220797	3.94229407385994 \\
20.8020195473264	4.00671533890928 \\
21.1241258725732	4.07113660395863 \\
21.4462321978199	4.13555786900798 \\
21.7683385230667	4.19997913405733 \\
22.0904448483134	4.26440039910668 \\
22.4125511735601	4.32882166415603 \\
22.7346574988069	4.39324292920537 \\
23.0567638240536	4.45766419425472 \\
23.3788701493004	4.52208545930407 \\
23.7009764745471	4.58650672435342 \\
24.0230827997938	4.65092798940277 \\
24.3451891250406	4.71534925445211 \\
24.6672954502873	4.77977051950146 \\
24.9894017755341	4.84419178455081 \\
25.3115081007808	4.90861304960016 \\
25.6336144260276	4.97303431464951 \\
25.9557207512743	5.03745557969886 \\
};
\end{axis}

\end{tikzpicture}